\documentclass[a4paper,11pt]{article}
\usepackage{graphicx,xcolor}
\usepackage{hyperref}
\hypersetup{
setpagesize=false,
 bookmarksnumbered=true,%
 bookmarksopen=true,%
 colorlinks=true,%
 linkcolor=blue,
 citecolor=red,
}
\usepackage{float}
\usepackage{amsmath,amssymb,bm,amsthm}
\usepackage{ascmac}
\usepackage{array}
\usepackage[hang,small,bf]{caption}
\usepackage[subrefformat=parens]{subcaption}
\usepackage{natbib}
\bibpunct[:]{(}{)}{,}{a}{}{,}
\title{Parameter estimation for ergodic linear SDEs from partial and discrete observations}
\author{Masahiro Kurisaki\thanks{Graduate School of Mathematical Sciences, the University of Tokyo\protect\\ email: \texttt{makurisaki@g.ecc.u-tokyo.ac.jp}}}

\numberwithin{equation}{section}
\allowdisplaybreaks[3]
\newtheorem{theorem}{Theorem}[section]

\newtheorem{proposition}[theorem]{Proposition}
\newtheorem{lemma}[theorem]{Lemma}
\newtheorem{remark}{Remark}

\newtheorem{corollary}[theorem]{Corollary}

\begin{document}
\maketitle

\begin{abstract}
  We consider a problem of parameter estimation for the state space model described by linear stochastic differential equations. We assume that an unobservable Ornstein-Uhlenbeck process drives another observable process by the linear stochastic differential equation, and these two processes depend on some unknown parameters.  We construct the quasi-likelihood estimator (QMLE) of the unknown parameters and show asymptotic properties of the estimator.
\end{abstract}


\section{Introduction}
On the probability space $(\Omega,\mathcal{F},P)$ with a complete and right-continuous filtration $\{\mathcal{F}_t\}$, we consider a $(d_1+d_2)$-dimensional Gaussian process $(X_t,Y_t)$ satisfying the following stochastic differential equations:
\begin{align}
  \label{eq1.1}&dX_t=-a(\theta_2)X_tdt+b(\theta_2)dW_t^1,\\
  \label{eq1.2}&dY_t=c(\theta_2)X_tdt+\sigma(\theta_1) dW_t^2,
\end{align}
where $W^1$ and $W^2$ are independent $d_1$ and $d_2$-dimensional $\{\mathcal{F}_t\}$-Wiener processes, $\theta_1 \in \Theta_1 \subset \mathbb{R}^{m_1}$ and $\theta_2 \in \Theta_2 \subset \mathbb{R}^{m_2}$ are unknown parameters, and $a,b:\Theta_2 \to M_{d_1}(\mathbb{R}),c:\Theta_2 \to M_{d_2,d_1}(\mathbb{R})$ and $\sigma:\Theta_1 \to M_{d_2}(\mathbb{R})$ are known functions. Here $M_{m,n}(\mathbb{R})$ is the set of $m \times n$ matrices over $\mathbb{R}$ and $M_{n}(\mathbb{R})=M_{n,n}(\mathbb{R})$, $\Theta_1$ and $\Theta_2$ are known parameter spaces. We assume that the process $X$ is unobservable, and the purpose of this article is to construct estimators of $\theta_1$ and $\theta_2$ based on discrete observations of $Y$.

Note that we can not identify $b(\theta_2)$ and $c(\theta_2)$ simultaneously from observation of $\{Y_t\}$. In fact, the system
\begin{align*}
  &dX_t=-a(\theta_2)X_tdt+2b(\theta_2)dW_t^1\\
  &dY_t=\frac{1}{2}c(\theta_2)X_tdt+\sigma(\theta_1) dW_t^2
\end{align*}
generates the same $\{Y_t\}$ as (\ref{eq1.1}) and (\ref{eq1.2}). Therefore, we need to impose some restrictions on $a,b,c,\sigma$ and the dimensions of the parameter spaces.

When $\theta_1$ and $\theta_2$ are known, one can estimate the unobservable state $\{X_t\}$ from observations of $\{Y_t\}$ by the following well-known Kalman-Bucy filter.
\begin{theorem}\label{Kalman-filter} (Theorem 10.2, \cite{Liptser-Shiryaev})\\
  In (\ref{eq1.1}) and (\ref{eq1.2}), let $\sigma(\theta)\sigma(\theta)'$ be positive definite, where the prime means the transpose. Then $m_t=E[X_t|\{Y_t\}_{0\leq s\leq t}]$ and $\gamma_t=E[(X_t-m_t)(X_t-m_t)']$ are the solutions of the equations
  \begin{align}
    \label{eq-mt}&dm_t=-a(\theta_2)m_tdt+\gamma_tc(\theta_2)'\{\sigma(\theta_1)\sigma(\theta_1)'\}^{-1}\{dY_t-c(\theta_2)m_tdt\},\\
    \label{eq-Riccati}&\frac{d\gamma_t}{dt}=-a(\theta_2)\gamma_t-\gamma_ta(\theta_2)'-\gamma_tc(\theta_2)'\{\sigma(\theta_1)\sigma(\theta_1)'\}^{-1}c(\theta_2)\gamma_t+b(\theta_2)b(\theta_2)'.
  \end{align}
\end{theorem}

Equation (\ref{eq-Riccati}) is the matrix Riccati equation, which has been examined in the theory of linear quadratic control\citep{mathematical-control-theory}. It is known that (\ref{eq-Riccati}) has the unique positive-semidefinite solution\citep{Liptser-Shiryaev}. Moreover, under proper conditions, one can show that the corresponding algebraic Riccati equation
\begin{align}
  \label{ARE}-a(\theta_2)\gamma-\gamma a(\theta_2)'-\gamma c(\theta_2)'\{\sigma(\theta_1)\sigma(\theta_1)'\}^{-1}c(\theta_2)\gamma+b(\theta_2)b(\theta_2)'=O
\end{align}
has the maximal and minimal solutions\citep{coppel,robust-and-optimal-control}, and the solution of (\ref{eq-Riccati}) converges to the maximal solution of (\ref{ARE}) at an exponential rate\citep{canonical-form-Riccati}. Further details on this topic will be discussed in Section \ref{second-proof-section}.

There are already several studies on parameter estimation in the system (\ref{eq1.1}) and (\ref{eq1.2}) with the Kalman-Bucy filter. For example, \cite{Kutoyants-2004} discusses the ergodic case, \cite{Kutoyants-1994} and \cite{KUTOYANTS2019-smallnoise} small noise cases, and \cite{KUTOYANTS2019-onestep} the one-step estimator. However, all of them assume $d_1=d_2=1$ and need continuous observation of $Y$. The continuous observation case is simpler, because we do not have to estimate $\theta_1$. In fact, we have 
\begin{align*}
  {Y_t}^2=\int_0^tY_tdY_t+\sigma(\theta_1)^2t
\end{align*}
by It\^{o}'s formula and (\ref{eq1.1}), and therefore we can get the exact value of $\sigma(\theta_1)$.

On the other hand, parametric inference for discretely observed stochastic differential equations without an unobservable process has been studied for decades (for example \cite{Sorensen2002},\cite{SHIMIZU2006},\cite{YOSHIDA1992}). Especially, \cite{YOSHIDA2011} developed Ibragimov-Khasminskii theory\citep{Ibragimov} into the quasi-likelihood analysis, and investigated the behavior of the quasi-likelihood estimator and the adaptive Bayes estimator in the ergodic diffusion process. Quasi-likelihood analysis is helpful to discretely observed cases, and many works have been derived from it: see \cite{UCHIDA2012} for the non-ergodic case, \cite{Ogihara2011} for the jump case, \cite{masuda2019} for the L\'{e}vy driven case, \cite{gloter-yoshida} for the degenerate case, \cite{multi-step} for the multi-step estimator,  and \cite{nakakita} for the case with observation noises. 

This paper also makes use of quasi-likelihood analysis to investigate the behaviors of our estimators. In Section \ref{assumption-setcion}, we describe the more precise setup and present asymptotic properties of our estimators, which are main results of this paper.  Then we go on to proofs of these results in Sections \ref{first-proof-section} and \ref{second-proof-section}. We also examine the Riccati differential equation ({\ref{eq-Riccati}}) and algebraic Riccati equation (\ref{ARE}) in Section \ref{second-proof-section}. In section \ref{one-dimensional-case-section}, we discuss the special case where $d_1=d_2=1$. In the one-dimensional case, we can reduce our assumptions to simpler ones. Finally, we show in Section \ref{simulation-section} the result of computational simulation by YUIMA, an package on R, and suggest a way to improve our estimators when the wrong initial value is given.

\section{Notations, assumptions and main results}\label{assumption-setcion}
Let $\theta_1^* \in \mathbb{R}^{m_1}$ and $\theta_2^*\in \mathbb{R}^{m_2}$ be the true values of $\theta_1$ and $\theta_2$, respectively, and define the $(d_1+d_2)$-dimensional Gaussian process $(X_t,Y_t)$ by 
\begin{align}
  \label{2.1}&dX_t=-a(\theta_2^*)X_tdt+b(\theta_2^*)dW_t^1,\\
  \label{2.2}&dY_t=c(\theta_2^*)X_tdt+\sigma(\theta_1^*) dW_t^2,
\end{align}
where $W_1,W_2,a,b,c$ and $\sigma$ are the same as Section 1; $a,b:\Theta_2 \to M_{d_1}(\mathbb{R}),c:\Theta_2 \to M_{d_2,d_1}(\mathbb{R})$ and $\sigma:\Theta_1 \to M_{d_2}(\mathbb{R})$. In this article, we have access to the discrete observations $Y_{ih_n}~(i=0,1,\cdots,n)$, where $h_n$ is some positive constant, and we construct the estimators of $\theta_1$ and $\theta_2$ based on the observations.

We assume that $\Theta_1\subset \mathbb{R}^{m_1}$ and $\Theta_2\subset \mathbb{R}^{m_2}$ are open bounded subsets and that the Sobolev embedding inequality holds on $\Theta=\Theta_1\times \Theta_2$; for any $p>m_1+m_2$ and $f\in C^1(\Theta)$, there exists some constant $C$ depending only on $\Theta$ such that 
\begin{align}
  \label{2.3}\sup_{\theta\in \Theta}|f(\theta)|\leq C\left( \|f\|_{L^p}+\|\partial_{\theta_i}f\|_{L^p} \right).
\end{align}
For example, if each $\Theta_i$ ($i=1,2$) has a Lipchitz boundary, this inequality is valid\citep{Sobolev}. 

Let $Z(\theta)~(\theta \in \Theta=\Theta_1\times \Theta_2)$ be a class of random variables, where $Z(\theta)$ is continuously differentiable with respect to $\theta$.
Then by (\ref{2.3}) and Fubini's theorem, we get for any $p>m_1+m_2$
\begin{align*}
  E\left[ \sup_{\theta\in \Theta}|Z(\theta)|^p \right]
  &\leq C2^{p-1}\left( E\left[ \int_{\Theta_i}|Z(\theta)|^pd\theta_i+\int_{\Theta}|\partial_{\theta}Z(\theta)|^pd\theta_i \right] \right)\\
  &=C2^{p-1}\left( \int_{\Theta_i}E[|Z(\theta)|^p]d\theta+\int_{\Theta}E[|\partial_{\theta}Z(\theta)|^p]d\theta \right)\\
  &\leq C_{p}\sup_{\theta\in\Theta}\left(E[|Z(\theta)|^p]+E[|\partial_{\theta}Z(\theta)|^p]  \right),
\end{align*}
where $C_p$ is some constant depending on $p$ and $\Theta$. This result will be frequently referred to in the following sections.

In what follows, we use the following notations:
\begin{itemize}
  \item $\mathbb{R}_+=[0,\infty),\mathbb{N}=\{1,2,\cdots\}$.
  \item $\Theta=\Theta_1\times \Theta_2$,$\theta_1=(\theta_1^1,\cdots,\theta_1^{m_1}),\theta_2=(\theta_2^1,\cdots,\theta_2^{m_2}),\theta^*=(\theta_1^*,\theta_2^*).$
  \item For any subset $\Xi \subset \mathbb{R}^m$, $\overline{\Xi}$ is the closure of $\Xi$.
  \item For every set of matrices $A$, $B$ and $C$, $A'$ is the transpose of $A$, $A^{\otimes 2}=AA'$, $A[B,C]=B'AC$ and $A[B^{\otimes 2}]=B'AB$.
  \item For every matrix $A$, $|A|$ is the Frobenius norm of $A$. Namely, if $A=(a_{ij})_{1\leq i\leq n,1\leq j\leq m}$, $|A|$ is defined by
  \begin{align*}
    |A|=\sqrt{\sum_{i=1}^n\sum_{j=1}^m a_{ij}^2}.
  \end{align*}
  \item For every matrix $A$, $\lambda_{\min}(A)$ donates the smallest real part of eigenvalues of matrix $A$.
  \item For every symmetric matrix $A$ and $B \in M_{d}(\mathbb{R})$, $A> B$ (resp. $A\geq B$) means that $A-B$ is positive (resp. semi-positive) definite. 
  \item For any open subset $\Xi \subset \mathbb{R}^m$ and $A:\Xi\to M_{d}(\mathbb{R})$ of class $C^k$, $\partial_{\xi}^kA(\xi)$ donates the $k$-dimensional tensor on $M_{d}(\mathbb{R})$ whose $(j_1,j_2,\cdots,j_k)$ entry is $\displaystyle \frac{\partial}{\partial\xi_{j_1}}\cdots \frac{\partial}{\partial\xi_{j_k}}A(\theta_i)$, where $1\leq j_1,\cdots,j_k\leq m$ and $\xi=(\xi_1,\cdots,\xi_m)$.
  \item For every $k$-dimensional tensor $A$ with $(i_1,i_2,\cdots,i_k)$ entry $A_{i_1\cdots i_k} \in M_{d}(\mathbb{R})$ and every matrix $B \in M_d(\mathbb{R})$, $AB$ donates the tensor whose $(i_1,i_2,\cdots,i_k)$ entry is $A_{i_1\cdots i_k}B$. $BA$ is also defined in the same way. 
  \item For any partially differentiable function $f:\Theta_2 \to \mathbb{R}^{d_2}$ and $S \in M_{d_2}(\mathbb{R})$, $S[\partial_{\theta_2}^{\otimes 2}]f(\theta)$ is the matrix whose $(i,j)$-entry is $\displaystyle \frac{\partial}{\partial{\theta_2^i}}f(\theta_2)S_{ij}\frac{\partial}{\partial{\theta_2^j}}f(\theta_2)$. 
  \item If both $A$ and $B$ are matrices with $M_d(\mathbb{R})$ entries, $AB$ is the normal product of matrices.
  \item For every matrix $A$ on $M_d(\mathbb{R})$ with $(i,j)$ entry $A_{ij} \in M_{d}(\mathbb{R})$, $\mathrm{Tr}A$ is a matrix on $\mathbb{R}$ with $(i,j)$ entry $\mathrm{Tr}A_{ij}$.
  \item For every stochastic process $Z$, $\Delta_i Z=Z_{t_{i}}-Z_{t_{i-1}}$.
  \item We write $a^*,b^*,c^*,\sigma^*,\Sigma^*$ and $h$ for $a(\theta_2^*),b(\theta_2^*),c(\theta_2^*),\sigma(\theta_1^*),\Sigma(\theta_1^*)$ and $h_n$.
  \item We designate $\sigma(\theta_1)\sigma(\theta_1)'$ as $\Sigma(\theta_1)$.
  \item $C$ donates a generic positive constant. When $C$ depends on some parameter $p$, we might use $C_p$ instead of $C$.
\end{itemize}

Moreover, we need the following assumptions:
\begin{description}
  \item[{[A1]}] $nh_n \to \infty,~n{h_n}^2 \to 0$ as $n \to \infty$. Moreover, we assume $h_n \leq 1$ for every $n \in \mathbb{N}$.
  \item[{[A2]}] $a,b,c$ and $\sigma$ are of class $C^4$.
\end{description}

Then we can extend $a,b,c$ and $\sigma$ to continuous functions on $\overline{\Theta}_1$ and $\overline{\Theta}_2$.

\begin{description}
  \item[{[A3]}] 
  \begin{align*}
    &\inf_{\theta_2 \in \overline{\Theta}_2}\lambda_{\min}(a(\theta_2))>0\\
    &\inf_{\theta_2 \in \overline{\Theta}_2}\lambda_{\min}(b(\theta_2)^{\otimes 2})>0\\
    &\inf_{\theta_1 \in \overline{\Theta}_1}\lambda_{\min}(\Sigma(\theta_1))>0.
  \end{align*}
  \item[{[A4]}] For any $\theta_1 \in \overline{\Theta}_1$ and $\theta_2 \in \overline{\Theta}_2$, the pair of matrix $(a(\theta_2)', \Sigma(\theta_1)[c(\theta_2)^{\otimes 2}])$ is controllable; i.e. the matrix
  \begin{align*}
    \begin{pmatrix}
      \Sigma(\theta_1)[c(\theta_2)^{\otimes 2}]&a(\theta_2)'\Sigma(\theta_1)[c(\theta_2)^{\otimes 2}]&\cdots&{a(\theta_2)'}^{d_1}\Sigma(\theta_1)[c(\theta_2)^{\otimes 2}]
    \end{pmatrix}
  \end{align*}
  has full row rank.
  
  Moreover, the eigenvalues of the matrix 
  \begin{align}
    \label{def-H-matrix}H(\theta_1,\theta_2)=\begin{pmatrix}
      a(\theta_2)'&\Sigma(\theta_1)^{-1}[c(\theta_2)^{\otimes 2}]\\
      b(\theta_2)^{\otimes 2}&-a(\theta_2)
    \end{pmatrix}
  \end{align}
  are uniformly bounded away from the imaginary axis; i.e. there are some constant $C>0$ such that for any $\theta_1 \in \overline{\Theta}_1$ and $\theta_2 \in \overline{\Theta}_2$ and eigenvalue $\lambda$ of $H(\theta_1,\theta_2)$, it holds
  \begin{align*}
    |\mathrm{Re}(\lambda)|>C.
  \end{align*} 
\end{description}

Now we define $\mathbb{Y}_1$ and $\mathbb{Y}_2$ by
\begin{align}
  \label{def-Y1}&\mathbb{Y}_1(\theta_1)=-\frac{1}{2}\left\{ \mathrm{Tr}\Sigma(\theta_1)^{-1}\Sigma(\theta_1^*)-d_1+\log \frac{\mathrm{det}\Sigma(\theta_1)}{\mathrm{det}\Sigma(\theta_1^*)} \right\}
\end{align}
and
\begin{align}
  \label{def-Y2}\begin{split}
    &\mathbb{Y}_2(\theta_2)=-\frac{1}{2}\mathrm{Tr}\int_{0}^\infty{\Sigma^*}^{-1}\left.\Biggr[ \left\{ \int_0^{s}c(\theta_2)\exp(-\alpha(\theta_2)u)\gamma_{+}(\theta_1^*,\theta_2)c(\theta_2)'{\Sigma^*}^{-1}c^*\right.\right.\\
  &\qquad\qquad\qquad\times \exp(-a^*(s-u))\gamma_+(\theta^*){c^*}'du\\
  &\qquad\qquad\qquad+c(\theta_2)\exp(-\alpha(\theta_2)s)\gamma_{+}(\theta_1^*,\theta_2)c(\theta_2)'\\
  &\qquad\qquad\qquad\left.\left.-c^*\exp(-a^*s)\gamma_+(\theta^*){c^*}' \right.\biggr\}^{\otimes 2} \right.\Biggr][({{\sigma^*}'}^{-1})^{\otimes 2}]ds,
  \end{split}
\end{align}
respectively, where
\begin{align}
  \label{alpha-a-gamma}\alpha(\theta_2)=a(\theta_2)+\gamma_+(\theta_1^*,\theta_2)\Sigma(\theta_1^*)^{-1}[c(\theta_2)^{\otimes 2}],\end{align}
and assume the following condition.
\begin{description}
  \label{A-identify}\item[{[A5]}] There is some positive constant $C>0$ satisfying
  \begin{align}
    \label{Y1-ineq}\mathbb{Y}_1(\theta_1)\leq -C|\theta_1-\theta_1^*|^2
  \end{align}
  and
  \begin{align}
    \label{Y2-ineq}\mathbb{Y}_2(\theta_2)\leq -C|\theta_2-\theta_2^*|^2.
  \end{align}
\end{description}

\begin{remark}
  By (\ref{alpha-a-gamma}), it holds
  \begin{align*}
    &\int_0^{s}c^*\exp(-\alpha(\theta_2^*)u)\gamma_{+}(\theta^*){c^*}'{\Sigma^*}^{-1}c^*\exp(-a^*(s-u))\gamma_+(\theta^*){c^*}'du\\
    =&\int_0^{s}c^*\exp(-\alpha(\theta_2^*)u)\left\{ \alpha(\theta_2^*)-a^* \right\}\exp(-a^*(s-u))\gamma_+(\theta^*){c^*}'du\\
    =&c^*\exp(-\alpha(\theta_2^*)s)-c(\theta_2)\exp(-a^*s)\gamma_+(\theta^*){c^*}',
  \end{align*}
  and therefore $\mathbb{Y}_2(\theta_2)$ has the following expression:
  \begin{align*}
    &\mathbb{Y}_2(\theta_2)=-\frac{1}{2}\mathrm{Tr}\int_{0}^\infty{\Sigma^*}^{-1}\left.\Biggl[ \left\{ \int_0^{s}\{c(\theta_2)\exp(-\alpha(\theta_2)u)\gamma_{+}(\theta_1^*,\theta_2)c(\theta_2)'{\Sigma^*}^{-1}c^*\right.\right.\\
    &\qquad\qquad\qquad-c^*\exp(-\alpha(\theta_2^*)u)\gamma_{+}(\theta^*){c^*}'{\Sigma^*}^{-1}c^*\}du\\
    &\qquad\qquad\qquad+c(\theta_2)\exp(-\alpha(\theta_2)s)\gamma_{+}(\theta_1^*,\theta_2)c(\theta_2)'\\
    &\qquad\qquad\qquad\left.\left.-c^*\exp(-\alpha(\theta_2^*)s)\gamma_{+}(\theta^*){c^*}' \right.\biggr\}^{\otimes 2} \right][({{\sigma^*}'}^{-1})^{\otimes 2}]ds.
  \end{align*}
  In particular, we have $\mathbb{Y}_2(\theta_2^*)=0$.
\end{remark}

Under these assumptions above, we set 
\begin{align}
  &\mathbb{H}_n^1(\theta_1)=-\frac{1}{2}\sum_{j=1}^n\left\{ \frac{1}{h}\Sigma^{-1}(\theta_1)[(\Delta_jY)^{\otimes 2}]+\log \det\Sigma(\theta_1) \right\}  
\end{align}
and
\begin{align*}
  &\Gamma^1=\frac{1}{2}\left[\mathrm{Tr}\{{\Sigma^*}^{-1}\partial_{\theta_1}\Sigma(\theta_1^*)\}\right]^{\otimes 2},
\end{align*}
and we define our estimator of $\theta_1$ as the maximizer of $\mathbb{H}_n^1(\theta_1)$. Note that $\mathrm{Tr}\{{\Sigma^*}^{-1}\partial_{\theta_1}\Sigma(\theta_1^*)\}$ is a vector whose $j$-th entry is $\displaystyle \mathrm{Tr}\left\{{\Sigma^*}^{-1}\frac{\partial}{\partial{\theta_1^j}}\Sigma(\theta_1^*)\right\}$.
Then the following theorem holds:
\begin{theorem}\label{main-theorem-1}
  We assume [A1]-[A5], and for each $n \in \mathbb{N}$, let $\hat{\theta}^n_1$ be a random variable satisfying
  \begin{align*}
    \mathbb{H}_n^1(\hat{\theta}^n_1)=\max_{\theta_1 \in \overline{\Theta}_1}\mathbb{H}_n^1(\theta_1).
  \end{align*} 
  Then for every $p>0$ and any continuous function $f:\mathbb{R}^d \to \mathbb{R}$ such that
  \begin{align*}
    \limsup_{|x|\to \infty}\frac{|f(x)|}{|x|^p}<\infty,
  \end{align*} 
  it holds that 
  \begin{align*}
    E[f(\sqrt{n}(\hat{\theta}^n_1-\theta_1^*))]\to E[f(Z)]~(n \to \infty),
  \end{align*}
  where $Z \sim N(0,(\Gamma^1)^{-1})$.

  In particular, it holds that 
  \begin{align*}
    \sqrt{n}(\hat{\theta}^n_1-\theta_1^*)\xrightarrow{d}N(0,(\Gamma^1)^{-1})~(n \to \infty).
  \end{align*}
\end{theorem}

Next we construct the estimator of $\theta_2$, which is the central part of this article. By Assumption [A4] and the corollary of Theorem 6 in \cite{coppel}, for every $\theta_1 \in \overline{\Theta}_1$ and $\theta_2 \in \overline{\Theta}_2$, equation (\ref{ARE}) has the maximal solution $\gamma=\gamma_+(\theta_1,\theta_2)$ and minimal solution $\gamma=\gamma_-(\theta_1,\theta_2)$, where $\gamma_+(\theta_1,\theta_2)>\gamma_-(\theta_1,\theta_2)$. The meaning of the maximal and minimal solutions is that for any symmetric solution $\gamma$ of (\ref{ARE}), it holds $\gamma_-\leq \gamma \leq \gamma_+$.

  Now we replace $\gamma_t$ with $\gamma_+(\theta_1,\theta_2)$ in (\ref{eq-mt}), and define $m_t(\theta_1,\theta_2;m_0)$ by
\begin{align}
  \label{eq-m}\begin{cases}
    dm_t=-a(\theta_2)m_tdt+\gamma_{+}(\theta_1,\theta_2)c(\theta_2)'\{\sigma(\theta_1)\sigma(\theta_1)'\}^{-1}\{dY_t-c(\theta_2)m_tdt\}\\
    m_0(\theta_1,\theta_2;m_0)=m_0,
  \end{cases}
\end{align}
where $m_0 \in \mathbb{R}^{d_1}$ is an arbitrary initial estimated value of $X_0$.

Due to It\^{o}'s formula, the solution of (\ref{eq-m}) can be written as
\begin{align}
  \label{def-m}\begin{split}
    &m_t(\theta_1,\theta_2)=\exp\left( -\alpha(\theta_1,\theta_2)t\right)m_0\\
  &+\int_0^t\exp\left( -\alpha(\theta_1,\theta_2)(t-s)\right)\gamma_+(\theta_1,\theta_2)c(\theta_2)'\Sigma(\theta_1)^{-1}dY_s,
  \end{split}  
\end{align}
where
\begin{align}
  \label{def-alpha}\alpha(\theta_1,\theta_2)=a(\theta_2)+\gamma_+(\theta_1,\theta_2)\Sigma(\theta_1)^{-1}[c(\theta_2)^{\otimes 2}].
\end{align}
The eigenvalues of $\alpha(\theta_1,\theta_2)$ coincides with those of $H(\theta_1,\theta_2)$ in (\ref{def-H-matrix}) with positive real part (see \cite{robust-and-optimal-control}), so there exists some constant $C>0$ such that for any $\theta_1 \in \Theta_1$ and $\theta_2 \in \Theta_2$,
\begin{align*}
  \inf_{\lambda \in \sigma(\alpha(\theta_1,\theta_2))}\mathrm{Re}\lambda>C,
\end{align*}
where $\sigma(\alpha(\theta_1,\theta_2))$ is the set of all eigenvalues of $\alpha(\theta_1,\theta_2)$.

According to (\ref{def-m}), we set for $i,n \in \mathbb{N}$,
\begin{align}
  \label{def-m-hat}\begin{split}
   &\hat{m}_i^n(\theta_2;m_0)=\exp\left( -\alpha(\hat{\theta}_1^n,\theta_2)t_i\right)m_0\\
  &+\sum_{j=1}^i\exp\left( -\alpha(\hat{\theta}_1^n,\theta_2)(t_i-t_{j-1})\right)\gamma_+(\hat{\theta}_1^n,\theta_2)c(\theta_2)'\Sigma(\hat{\theta}_1^n)^{-1}\Delta_jY,
  \end{split}\\
  \label{def-H2}
  \begin{split}
    &\mathbb{H}_n^2(\theta_2;m_0)=\frac{1}{2}\sum_{i=1}^n\left\{-h\Sigma(\hat{\theta}_1^n)^{-1}[(c(\theta_2)\hat{m}_{j-1}^n(\theta_2))^{\otimes 2}]\right.\\
  &\left.+\hat{m}_{j-1}^n(\theta_2)'c(\theta_2)'\Sigma(\hat{\theta}_1^n)^{-1}\Delta_jY+\Delta_jY'\Sigma(\hat{\theta}_1^n)^{-1}c(\theta_2)\hat{m}_{j-1}^n(\theta_2)\right\},
  \end{split}
\end{align}
and
\begin{align}
   \label{def-Gamma2}\begin{split}
    &\Gamma^2=\mathrm{Tr}\int_0^\infty{\Sigma^*}^{-1}[\partial_{\theta_2}^{\otimes 2}]\left\{\int_0^sc(\theta_2)\exp(-\alpha(\theta_2)u)\gamma_{+}(\theta_2)c(\theta_2)'{\Sigma^*}^{-1}c^*\right.\\
    &\qquad\qquad\qquad\times \exp(-a^*(s-u))\gamma_+(\theta^*){c^*}'du\\
    &\qquad\qquad\qquad\left.\left.+c(\theta_2)\exp(-\alpha(\theta_2)s)\gamma_{+}(\theta_2)c(\theta_2)'\right.\biggr\}\right|_{\theta_2=\theta_2^*}ds,
   \end{split}
\end{align}
where $\hat{\theta}_1^n$ is the estimator of $\theta_1$ defined in Theorem \ref{main-theorem-1}. 
Then the following theorem holds:
\begin{theorem}\label{main-theorem-2}
  We assume [A1]-[A5], and let $m_0 \in \mathbb{R}^{d_1}$ be an arbitrary initial value and $\hat{\theta}^n_2=\hat{\theta}^n_2(m_0)$ be a random variable satisfying
  \begin{align*}
    \mathbb{H}_n^2(\hat{\theta}^n_2)=\max_{\theta_2 \in \overline{\Theta}_2}\mathbb{H}_n^2(\theta_2)
  \end{align*} 
  for each $n \in \mathbb{N}$. Moreover, let $\Gamma^2$ be positive definite. Then for any $p>0$ and continuous function $f:\mathbb{R}^d \to \mathbb{R}$ such that
  \begin{align*}
    \limsup_{|x|\to \infty}\frac{|f(x)|}{|x|^p}<\infty,
  \end{align*} 
  it holds that 
  \begin{align*}
    E[f(\sqrt{t_n}(\hat{\theta}^n_2-\theta_2^*))]\to E[f(Z)]~(n \to \infty),
  \end{align*}
  where $Z \sim N(0,(\Gamma^2)^{-1})$.

  In particular, it holds that 
  \begin{align*}
    \sqrt{t_n}(\hat{\theta}^n_2-\theta_2^*)\xrightarrow{d}N(0,(\Gamma^2)^{-1})~(n \to \infty).
  \end{align*}
\end{theorem}

\begin{remark}
  (1) In order to calculate $\hat{m}_i^n$, one can use the autoregressive formula
  \begin{align*}
    \hat{m}_{i+1}^n(\theta_2;m_0)=&\exp\left( -\alpha(\hat{\theta}_1^n,\theta_2)h\right)\hat{m}_{i}^n(\theta_2;m_0)\\
    &+\exp\left( -\alpha(\hat{\theta}_1^n,\theta_2)h\right)\gamma_+(\hat{\theta}_1^n,\theta_2)c(\theta_2)'\Sigma(\hat{\theta}_1^n)^{-1}\Delta_{i+1}Y.
  \end{align*}
  (2) One can obtain $\gamma(\theta_1,\theta_2)$ in the following way (see \cite{robust-and-optimal-control} for details). Let $v_1,v_2,\cdots, v_{d_1}$ be generalized eigenvectors of $H(\theta_1,\theta_2)$ in (\ref{def-H-matrix}) with positive real part eigenvalues. Note that $H(\theta_1,\theta_2)$ has $d_1$ eigenvalues (with multiplicity) in the right half-plane and $d_1$ in the left half-plane. We define the matrices $X_1(\theta_1,\theta_2)$ and $X_2(\theta_1,\theta_2)$ by
  \begin{align*}
    \begin{pmatrix}
      v_1&v_2&\cdots&v_{d_1}
    \end{pmatrix}=\begin{pmatrix}
      X_1(\theta_1,\theta_2)\\X_2(\theta_1,\theta_2)
    \end{pmatrix}.
  \end{align*}
  Then $X_1(\theta_1,\theta_2)$ is invertible and it holds $\gamma_+(\theta_1,\theta_2)=X_2(\theta_1,\theta_2)X_1(\theta_1,\theta_2)^{-1}$.\\
  (3) $\mathbb{H}^2(\theta_2)$ can be interpreted as a approximated log-likelihood function with $\theta_1$ given. In fact, if $X_t=X_t(\theta)$ and $Y_t=Y_t(\theta)$ are generated by (\ref{eq1.1}) and (\ref{eq1.2}), and we set $m_0=E[X_0|Y_0]$ and $\gamma_0=E[(m_0-X_0)^{\otimes 2}]$, then it follows $m_t(\theta)=E[X_t(\theta)|\{Y_s(\theta)\}_{0\leq s\leq t}]$ by Theorem \ref{Kalman-filter}. Thus by the innovation theorem\citep{Kallianpur}, we can replace $X_t(\theta)$ with $m_t(\theta)$ in Equation (\ref{eq1.2}), and consider the equation
\begin{align*}
  dY_t(\theta)=c(\theta_2)m_t(\theta)dt+\sigma(\theta_1)d\overline{W}_t
\end{align*}
where $\overline{W}$ is a $d_2$-dimensional Wiener process. We can approximate this equation as 
  \begin{align*}
    \Delta_i Y(\theta)\approx c(\theta_2)m_{t_{i-1}}(\theta)h+\sigma(\theta_1)\Delta_i \overline{W},
  \end{align*}
  when $h \approx 0$. Then we obtain the approximated likelihood function
  \begin{align*}
    p(\theta)\approx&\prod_{i=1}^n\frac{1}{(2\pi h)^{\frac{d}{2}}\{\mathrm{det}\Sigma(\theta_1)\}^{-\frac{1}{2}}}\\
    &\times \exp\left( -\frac{1}{2h}\Sigma(\theta_1)^{-1}\left[ (\Delta_iY-c(\theta_2)m_{t_{i-1}}(\theta)h)^{\otimes 2} \right] \right).
  \end{align*} 
\end{remark}

\section{Proof of Theorem \ref{main-theorem-1}}\label{first-proof-section}
In this section, we prove Theorem \ref{main-theorem-1}, which can be proved in the same way as the diffusion case in \cite{YOSHIDA2011}. 
\begin{lemma}\label{BDG}
  Let $\{W_t\}$ be a $d$-dimensional $\{\mathcal{F}_t\}$-Wiener process.\\
  (1) Let $f:\mathbb{R}_+ \to \mathbb{R}^m$ be a measurable function. Then for any $p\geq1$ and $0\leq s\leq t$, it holds
  \begin{align*}
    \left( \int_s^t|f(u)|du \right)^p\leq (t-s)^{p-1}\int_s^t|f(u)|^pdu
  \end{align*}
  (2) Let $\{A_t\}$ be a $M_{k,d}(\mathbb{R})$-valued progressively measurable process and $\{W_t\}$ be a $d$-dimensional Wiener process. Then for every $0\leq s\leq t\leq T$ and $p \geq 2$, it holds
  \begin{align*}
    E\left[\sup_{s\leq t \leq T}\left|\int_s^t A_udW_u\right|^p \right]&\leq C_{p,d,k}E\left[\left(\int_s^T|A_u|^2du\right)^{\frac{p}{2}}  \right]\\
    &\leq C_{p,d,k}(T-s)^{\frac{p}{2}-1}\int_s^tE[|A_u|^p]du.
  \end{align*}
\end{lemma}
\begin{proof}
  (1) By H\"{o}lder's inequality, we obtain
  \begin{align*}
    \int_s^t|f(u)|du&\leq \left(\int_s^t|f(u)|^{p}du\right)^{\frac{1}{p}} \left(\int_s^tdu\right)^{1-\frac{1}{p}}\\
    &=(t-s)^{1-\frac{1}{p}}\left(\int_s^t|f(u)|^{p}du\right)^{\frac{1}{p}},
  \end{align*}
  and it shows the desired inequality.\\
  (2) Let $A_t^{(ij)}$ be the $(i,j)$ entry of $A_t$, and $W_t^{(j)}$ be the $j$-th element of $W_t$. Then the Burkholder-Davis-Gundy inequality gives
  \begin{align*}
    E\left[\sup_{s\leq t\leq T}\left|\int_s^t A_udW_u\right|^p\right]
    &=E\left[\sup_{s\leq t\leq T}\left\{ \sum_{i=1}^k\left( \sum_{j=1}^d\int_{s}^tA_u^{(ij)}dW_u^{(j)} \right)^2 \right\}^\frac{p}{2}\right]\\
    &\leq C_{p,d,k}\sum_{i=1}^k \sum_{j=1}^dE\left[\sup_{s\leq t\leq T}\left|\int_{s}^tA_u^{(ij)}dW_u^{(j)} \right|^p\right]\\
    &\leq C_{p,d,k}\sum_{i=1}^k \sum_{j=1}^dE\left[\left|\int_{s}^T(A_u^{(ij)})^2du \right|^{\frac{p}{2}}\right]\\
    &\leq C_{p,d,k}E\left[\left|\int_{s}^T\sum_{i=1}^k \sum_{j=1}^d(A_u^{(ij)})^2du \right|^{\frac{p}{2}}\right]\\
    &=C_{p,d,k}E\left[\left|\int_{s}^T|A_u|^2du \right|^{\frac{p}{2}}\right].
  \end{align*}
  Hence we have proved the first inequality, and together with (1) we obtain the second one.
\end{proof}

\begin{lemma}\label{matrix-exponential}
  Let $A$ be a $d\times d$ matrix having eigenvalues $\lambda_1,\cdots,\lambda_k$. Then for all $\epsilon>0$, there exists some constant $C_{\epsilon,d}$ depending on $\epsilon$ and $d$ such that
  \begin{align*}
    |\exp(At)|\leq C_{\epsilon,d}(1+|A|^{d-1})e^{(\lambda_{\max}+\epsilon)t}~(t \geq 0),
  \end{align*}
  where
  \begin{align*}
    \lambda_{\max}=\max_{i=1,\cdots,k}\mathrm{Re}\lambda_k.
  \end{align*}
\end{lemma}
\begin{proof}
  Let 
  \begin{align*}
    A=U^*(D+N)U,~D=\mathrm{diag}(\lambda_1,\lambda_2,\cdots,\lambda_{d})
  \end{align*}
  be a Schur decomposition of $A$, where $\lambda_1,\lambda_2,\cdots,\lambda_{d}$ are the eigenvalues of $A$, $U$ is an unitary matrix, and $N$ is a strictly upper triangular matrix. Then we have
  \begin{align*}
    |\exp(At)|&=|\exp((D+N)t)|=|\exp(Dt)\exp(Nt)|\\
    &\leq |\exp(Dt)||\exp(Nt)|\\
    &\leq C_de^{\lambda_{\max}t}\sum_{k=1}^{d-1}\frac{|N|^{k}}{k!}t^{k}\\
    &\leq C_de^{\lambda_{\max}t}\sum_{k=1}^{d-1}\frac{|A|^{k}}{k!}t^{k}\\
    &\leq C_de^{(\lambda_{\max}+\epsilon)t}\sum_{k=1}^{d-1}\frac{|A|^{k}}{k!}t^{k}e^{-\epsilon t}\\
    &\leq C_{\epsilon,d}(1+|A|^{d-1})e^{(\lambda_{\max}+\epsilon)t},
  \end{align*}
  noting that $U$ is unitary, $N^{d}=O$, and $|A|=|D+N|\geq N$. 
\end{proof}

\begin{lemma}\label{X-Y-norm}
  For any $s,t\geq 0$ such that $0\leq t-s\leq 1$ and $p \geq 1$, it holds
  \begin{align}
    &\label{X-norm}\sup_{t\geq 0}E[|X_t|^p]\leq C_p,\\
    &E[|Y_s-Y_t|^p]\leq C_p|s-t|^{\frac{p}{2}}    
  \end{align}
  and
  \begin{align}
    &\label{X-conti}E[|X_s-X_t|^p]\leq C_p|s-t|^{\frac{p}{2}}.
  \end{align}
\end{lemma}
\begin{proof}
  By It\^{o}'s formula, the solution of (\ref{2.1}) can be expressed as
  \begin{align}
    \label{sol-X}X_t=\exp(-a^*t)X_0+\int_0^t\exp(-a^*(t-s))b^*dW_s^1,
  \end{align}
  where $\exp$ is the matrix exponential. Hence by Lemmas \ref{BDG} and \ref{matrix-exponential}, we have
  \begin{align*}
    E[|X_t|^p]\leq &|\exp(-a^*t)|^pE[|X_0|^p]\\
    &+C_p\left(\int_0^t|b^*|^2|\exp(-a^*(t-s))|^2ds\right)^{\frac{p}{2}}\\
    \leq &C_pe^{-\eta pt}+C_p\left(\int_0^te^{-2\eta s}ds\right)^{\frac{p}{2}}\leq C_p
  \end{align*}
  for some constant $\eta>0$. Therefore for $s\leq t$ we obtain
  \begin{align*}
    E[|Y_t-Y_s|^p]&=E\left[ \left|c^*\int_s^tX_udu+\sigma^*(W_t^2-W_s^2)\right|^p \right]\\
    &\leq C_p\left( (t-s)^{p-1}\int_s^tE[|X_u|^p]du+(t-s)^{\frac{p}{2}} \right)\\
    &\leq C_p\left( (t-s)^{p}+(t-s)^{\frac{p}{2}} \right)\\
    &\leq C_p(t-s)^{\frac{p}{2}}.
  \end{align*}
  We can show (\ref{X-conti}) in the same way.
\end{proof}

\begin{lemma}\label{Y-W-dif-lemma}
  For every $p\geq 2$ and $A \in M_{d_1}(\mathbb{R})$, it holds
  \begin{align*}
    E\left[ \left|\sum_{j=1}^nA[(\Delta_jY)^{\otimes2}]-\sum_{j=1}^nA[(\sigma^*\Delta_jW^2)^{\otimes2}]\right|^p \right]^\frac{1}{p}\leq C_p|A|(nh^2+n^\frac{1}{2}h^\frac{3}{2}).
  \end{align*}
\end{lemma}
\begin{proof}
  First we get
  \begin{align*}
    &E\left[ \left|\sum_{j=1}^nA[(\Delta_jY)^{\otimes2}]-\sum_{j=1}^nA[(\sigma^*\Delta_jW^2)^{\otimes2}]\right|^p \right]^\frac{1}{p}\nonumber\\
    =&E\left[ \left|\sum_{j=1}^nA\left[\left(c^*\int_{t_{j-1}}^{t_j}X_tds+\sigma^*\Delta_jW^2\right)^{\otimes2}\right]-\sum_{j=1}^nA[(\sigma^*\Delta_jW^2)^{\otimes2}]\right|^p \right]^\frac{1}{p}\nonumber\\
    \leq&E\left[ \left|\sum_{j=1}^nA\left[\left(c^*\int_{t_{j-1}}^{t_j}X_tds\right)^{\otimes2}\right]\right|^p \right]^\frac{1}{p}\nonumber\\
    &+E\left[ \left|\sum_{j=1}^nA\left[c^*\int_{t_{j-1}}^{t_j}X_tds,\sigma^*\Delta_jW^2\right]\right|^p \right]^\frac{1}{p}\nonumber\\
    &+E\left[ \left|\sum_{j=1}^nA\left[\sigma^*\Delta_jW^2,c^*\int_{t_{j-1}}^{t_j}X_tds\right]\right|^p \right]^\frac{1}{p}.
  \end{align*}
  For the first term of the rightest-hand side, we obtain by Lemmas \ref{BDG} and \ref{X-Y-norm}
  \begin{align*}
    &E\left[ \left|\sum_{j=1}^nA\left[\left(c^*\int_{t_{j-1}}^{t_j}X_tds\right)^{\otimes2}\right]\right|^p \right]^\frac{1}{p}\\
    \leq &|A||c^*| \sum_{j=1}^nE\left[ \left|\int_{t_{j-1}}^{t_j}X_tds\right|^{2p} \right]^\frac{1}{p}\\
    \leq &|A||c^*|\sum_{j=1}^nh^{2p-1}\left(\int_{t_{j-1}}^{t_j}E[ \left|X_t\right|^{2p} ]\right)^\frac{1}{p}ds\\
    \leq &C_p|A|nh^2.
  \end{align*}
  For the second and third terms, it holds
  \begin{align*}
    &E\left[ \left|\sum_{j=1}^nA\left[c^*\int_{t_{j-1}}^{t_j}X_tds,\sigma^*\Delta_jW^2\right]\right|^p \right]^\frac{1}{p}\\
    \leq &E\left[ \left|\sum_{j=1}^nA\left[c^*\int_{t_{j-1}}^{t_j}(X_t-X_{t_{j-1}})ds,\sigma^*\Delta_jW^2\right]\right|^p \right]^\frac{1}{p}\\
    &+E\left[ \left|\sum_{j=1}^nA\left[c^*X_{t_{j-1}}h,\sigma^*\Delta_jW^2\right]\right|^p \right]^\frac{1}{p}\\
    \leq &|A||c^*||\sigma^*|\sum_{j=1}^nE\left[\left|\int_{t_{j-1}}^{t_j}(X_t-X_{t_{j-1}})ds\right|^p|\Delta_jW^2|^p\right]^\frac{1}{p}\\
    &+hE\left[ \left|\sum_{j=1}^nX_{t_{j-1}}'{c^*}'A\sigma^*\Delta_jW^2\right|^p \right]^\frac{1}{p}\\
    \leq &|A||c^*||\sigma^*|\sum_{j=1}^nE\left[\left|\int_{t_{j-1}}^{t_j}(X_t-X_{t_{j-1}})ds\right|^{2p}\right]^\frac{1}{2p}E\left[|\Delta_jW^2|^{2p}\right]^\frac{1}{2p}\\
    &+hE\left[ \left|\sum_{j=1}^nX_{t_{j-1}}'{c^*}'A\sigma^*\Delta_jW^2\right|^p \right]^\frac{1}{p}\\
    \leq &C_p|A|\sum_{j=1}^n\left( h^{2p-1}\int_{t_{j-1}}^{t_j}E[|X_u-X_{t_{j-1}}|^{2p}]du \right)^\frac{1}{2p}h^\frac{1}{2}\\
    &+C_p|A|h\left((nh)^{\frac{p}{2}-1}\sum_{j=1}^nE[|X_{t_{j-1}}|^p]h\right)^\frac{1}{p}\\
    \leq &C_p|A|(nh^2+n^\frac{1}{2}h^\frac{3}{2}).
  \end{align*}
  Therefore we get the desired result.
\end{proof}

\begin{lemma}\label{martingale-lemma}
  Let $A_k \in M_{d_1}(\mathbb{R})~(k=1,2,\cdots,d)$, $A=(A_1,\cdots,A_d)$ and 
  \begin{align*}
    M_n(A)=\sum_{j=1}^n\left\{\frac{1}{h}A[(\Delta_jW^2)^{\otimes 2}]-\mathrm{Tr}A\right\}.
  \end{align*}
  Then it holds that
  \begin{align}
    \label{M-norm}E[|M_n(A)|^p]\leq C_p|A|\sqrt{n}
  \end{align}
  and
  \begin{align}
    \label{M-CLT}\frac{1}{\sqrt{n}}M_n \xrightarrow{d} N(0,2(\mathrm{Tr}A)^{\otimes 2})~(n\to\infty).
  \end{align}
\end{lemma}
\begin{proof}
  On account of
  \begin{align*}
    E\left[ \frac{1}{h}A[(\Delta_jW^2)^{\otimes 2}]-\mathrm{Tr}A\middle|\mathcal{F}_{t_{j-1}} \right]
    =\frac{1}{h}E\left[ A[(\Delta_jW^2)^{\otimes 2}]\right]-\mathrm{Tr}A=0,
  \end{align*}
  $\{A[(\Delta_jW^2)^{\otimes 2}]/h-\mathrm{Tr}A\}_j$ is a martingale difference sequence with respect to $\{\mathcal{F}_{t_j}\}$. Hence the Burkholder inequality gives 
  \begin{align*}
    E[|M_n|^p]^\frac{p}{2}&\leq C_pE\left[ \left| \sum_{j=1}^n\frac{1}{h}A[(\Delta_jW^2)^{\otimes 2}]-\mathrm{Tr}A \right|^{2p} \right]^\frac{1}{2p}\\
    &\leq C_p\sum_{j=1}^nE\left[ \left| \frac{1}{h}A[(\Delta_jW^2)^{\otimes 2}]-\mathrm{Tr}A \right|^{2p} \right]^\frac{1}{2p}\\
    &\leq C_pn\left\{ |A|^{2p}E\left[ \left|\frac{1}{h}W_h^2\right|^{4p}+|A|^2p \right] \right\}\\
    &\leq C_p|A|n,
  \end{align*}
  and we obtain (\ref{M-norm}).\\
  Moreover, due to the fact that $\{A[(\Delta_jW^2)^{\otimes 2}]/h-\mathrm{Tr}A\}_j$ is independent and identically distributed, we have
  \begin{align*}
    &E\left[ \left( \frac{1}{h}A[(\Delta_jW^2)^{\otimes 2}]-\mathrm{Tr}A \right)^{\otimes 2} \right]\\
    =&\frac{1}{h^2}E\left[  \{A[(\Delta_jW^2)^{\otimes 2}]\}^{\otimes 2} \right]-\frac{1}{h}E[A[(\Delta_jW^2)^{\otimes 2}]]'\mathrm{Tr}A\\
    &-(\mathrm{Tr}A)'\frac{1}{h}E[A[(\Delta_jW^2)^{\otimes 2}]]+(\mathrm{Tr}A)^{\otimes 2}\\
    =&3(\mathrm{Tr}A)^{\otimes 2}-2(\mathrm{Tr}A)^{\otimes 2}+(\mathrm{Tr}A)^{\otimes 2}=2(\mathrm{Tr}A)^{\otimes 2}.
  \end{align*}
  Thus we obtain (\ref{M-CLT}).
\end{proof}

By Lemmas \ref{Y-W-dif-lemma} and \ref{martingale-lemma}, we get the following lemma.
\begin{lemma}\label{CLT-lemma=Y}
  Let $A_k \in M_{d_1}(\mathbb{R})~(k=1,2,\cdots,d)$, $A=(A_1,\cdots,A_d)$ and 
  \begin{align*}
    L_n(A)=\sum_{j=1}^n\left\{\frac{1}{h}A[(\Delta_jY)^{\otimes 2}]-\mathrm{Tr}A\Sigma\right\}.
  \end{align*}
  Then it holds that
  \begin{align}
    E[|L_n(A)|^p]\leq C_p|A|(nh+n^\frac{1}{2}h^\frac{1}{2}+n^\frac{1}{2})
  \end{align}
  and
  \begin{align}
    \frac{1}{\sqrt{n}}L_n \xrightarrow{d} N(0,2(\mathrm{Tr}A\Sigma)^{\otimes 2})~(n\to\infty).
  \end{align}
\end{lemma}

\begin{lemma}\label{third-differential-lemma}
  For every $p>0$, it holds
  \begin{align*}
    \sup_{n \in \mathbb{N}}E\left[ \left|\frac{1}{n}\sup_{\theta_1 \in \Theta_1}\partial_{\theta_1}^3 \mathbb{H}_n^1(\theta_1)\right|^p \right]<\infty
  \end{align*}
\end{lemma}
\begin{proof}
  It is enough to prove the inequality for sufficiently large $p$. By Lemmas \ref{martingale-lemma} and \ref{Y-W-dif-lemma} and Assumptions [A2] and [A4] we get
  \begin{align*}
    &E\left[\left|\frac{1}{n}\partial_{\theta_1}^3\mathbb{H}_n^1(\theta_1)\right|^p\right]^\frac{1}{p}\\
    =&E\left[ \left|\frac{1}{2n}\sum_{j=1}^n\left\{ \frac{1}{h}\partial_{\theta_1}^3\Sigma^{-1}(\theta_1)[(\Delta_jY)^{\otimes 2}]+\partial_{\theta_1}^3\log \det\Sigma(\theta_1) \right\}\right|^p \right]^\frac{1}{p}\\
    \leq&E\left[ \left|\frac{1}{2nh}\sum_{j=1}^n\left\{ \partial_{\theta_1}^3\Sigma^{-1}(\theta_1)[(\Delta_jY)^{\otimes 2}-\mathrm{Tr}\partial_{\theta_1}^3\Sigma^{-1}(\theta_1) \right\}\right|^p \right]^\frac{1}{p}\\
    +&\frac{1}{2}\left\{ |\mathrm{Tr}\partial_{\theta_1}^3\Sigma^{-1}(\theta_1)|+|\partial_{\theta_1}^3\log \det\Sigma(\theta_1)| \right\}\\
    \leq &C_p|\partial_{\theta_1}^3\Sigma^{-1}(\theta_1)|(h+n^{-\frac{1}{2}}h^\frac{1}{2}+n^{-\frac{1}{2}})\\
    &+\frac{1}{2}\left\{ |\mathrm{Tr}\partial_{\theta_1}^3\Sigma^{-1}(\theta_1)|+|\partial_{\theta_1}^3\log \det\Sigma(\theta_1)| \right\}\\
    \leq &C_p,
  \end{align*}
  and similarly
  \begin{align*}
    E\left[\left|\frac{1}{n}\partial_{\theta_1}^4\mathbb{H}_n^1(\theta_1)\right|^p\right]^\frac{1}{p}\leq C_p.
  \end{align*}
  Thus we get the desired result for $p>d_1$ by the Sobolev inequality.
\end{proof}

\begin{proof}[\bf{Proof of Theorem \ref{main-theorem-1}}]
  Let
  \begin{align*}
    &\Delta_n^1=\frac{1}{\sqrt{n}}\partial_{\theta_1}\mathbb{H}_n^1(\theta_1^*),\\
    &\Gamma_n^1=-\frac{1}{n}\partial_{\theta_1}^2\mathbb{H}_n^1(\theta_1^*)
  \end{align*}
  and
  \begin{align*}
    \mathbb{Y}_n^1(\theta_1)=\frac{1}{n}\{\mathbb{H}_n^1(\theta_1)-\mathbb{H}_n^1(\theta_1^*)\}.
  \end{align*}
  Then 
  \begin{align*}
    \Delta_n^1
    &=-\frac{1}{2\sqrt{n}}\sum_{j=1}^n\left\{ \frac{1}{h}\partial_{\theta_1}\Sigma^{-1}(\theta_1^*)[(\Delta_jY)^{\otimes 2}]+\frac{\partial_{\theta_1}\det \Sigma(\theta_1^*)}{\det \Sigma^*} \right\}\\
    &=\frac{1}{2\sqrt{n}}\sum_{j=1}^n\left\{ \frac{1}{h}{\Sigma^*}^{-1}\partial_{\theta_1}\Sigma(\theta_1^*){\Sigma^*}^{-1}[(\Delta_jY)^{\otimes 2}]-\mathrm{Tr}{\Sigma^*}^{-1}\partial_{\theta_1}\Sigma(\theta_1^*)\right\},
  \end{align*}
  and hence by Lemma \ref{CLT-lemma=Y}, we obtain 
  \begin{align}
    \label{Delta-norm}E[|\Delta_n^1|^p]\leq C_p 
  \end{align}
  and
  \begin{align}
    \label{Delta-CLT}\Delta_n^1 \xrightarrow{d} N(0,(\Gamma^1)^{-1}) ~(n \to \infty).
  \end{align}
  By the same lemma, it follows that 
  \begin{align}
    \label{Gamma-dif}\begin{aligned}
      &E\left[ \left|n^\frac{1}{2}(\Gamma_n^1-\Gamma^1)\right|^p \right]^\frac{1}{p}\\
    =&E\left[ \left|n^\frac{1}{2}\left[\frac{1}{n}\sum_{j=1}^n\left\{ \frac{1}{h}\partial_{\theta_1}^2\Sigma^{-1}(\theta_1^*)[(\Delta_jY)^{\otimes 2}]+\partial_{\theta_1}^2\log\det \Sigma(\theta_1^*) \right\}-\Gamma^1\right]\right|^p \right]^\frac{1}{p}\\
    \leq &C_p(n^\frac{1}{2}h^2+h^\frac{3}{2}+1)=O(1),
    \end{aligned}
  \end{align}
  noting that 
  \begin{align*}
    &\mathrm{Tr}\partial_{\theta_1}^2\Sigma^{-1}(\theta_1^*)\Sigma^*\\
    =&\mathrm{Tr}\left\{2\{{\Sigma^*}^{-1}\partial_{\theta_1}\Sigma(\theta_1^*)\}^{\otimes2}-\Sigma^{-1}\frac{\partial}{\partial\theta_{1}^i}\frac{\partial}{\partial\theta_{1}^j}\Sigma(\theta_1^*)\right\}
  \end{align*}
  is equal to $-\partial_{\theta_1}^2\log\det \Sigma(\theta_1^*)+\Gamma^1$.\\
  Moreover, we can show
  \begin{align*}
  &\sup_{\theta_1 \in \Theta_1}E\left[ (n^\frac{1}{2}|\mathbb{Y}_n^1(\theta_1)-\mathbb{Y}^1(\theta_1)|)^p \right]^\frac{1}{p}\\
  = &\sup_{\theta_1 \in \Theta_1}E\left[ \left|-\frac{1}{2\sqrt{n}}\sum_{j=1}^n\left\{\frac{1}{h}\{\Sigma^{-1}(\theta_1)-\Sigma^{-1}(\theta_1^*)\}[(\Delta_jY)^{\otimes 2}]\right.\right.\right.\\
  &\left.\left.\left.-\mathrm{Tr}\{\Sigma(\theta_1)^{-1}-I\}\right\}\right|^p \right]^\frac{1}{p}\\
  \leq &C_p(n^\frac{1}{2}h+h+1)=O(1)
  \end{align*}
  and in the same way
  \begin{align*}
    \sup_{\theta_1 \in \Theta_1}E\left[ (n^\frac{1}{2}|\partial_{\theta_1}\{\mathbb{Y}_n^1(\theta_1)-\mathbb{Y}^1(\theta_1)\}|)^p \right]^\frac{1}{p}
    =O(1).
  \end{align*}
  Thus by the Sobolev inequality, it holds for $p>d_1$
  \begin{align}
    \label{Y-dif}E\left[ \sup_{\theta_1 \in \Theta_1}(n^\frac{1}{2}|\mathbb{Y}_n^1(\theta_1)-\mathbb{Y}^1(\theta_1)|)^p \right]^\frac{1}{p}
    =O(1).
  \end{align}
  Then we have proved the theorem by the assumption [A5], Lemma \ref{third-differential-lemma}, (\ref{Delta-norm}), (\ref{Delta-CLT}), (\ref{Gamma-dif}), (\ref{Y-dif}) and Theorem 5 in \cite{YOSHIDA2011}.
\end{proof}

\section{Proof of Theorem \ref{main-theorem-2}}\label{second-proof-section}
In this section, we write $m_t(\theta_2)$, $\hat{m}_i^n(\theta_2)$, $\mathbb{H}_n^2(\theta_2)$, $\gamma_+(\theta_2)$ and $\alpha(\theta_2)$ instead of $m_t(\theta_1^*,\theta_2;m_0)$, $\hat{m}_i^n(\theta_2;m_0)$, $\mathbb{H}_n^2(\theta_2;m_0)$, $\gamma_+(\theta_1^*,\theta_2)$ and $\alpha(\theta_1^*,\theta_2)$, respectively, for simplicity.

Moreover, let $m_t^*=E[X_t|\{Y_t\}_{0\leq s\leq t}]$ and $\gamma_t^*=E[(X_t-m_t)(X_t-m_t)']$. Then by Theorem \ref{Kalman-filter}, they are the solutions of
\begin{align}
  &\label{eq-m*}dm_t^*=-a^*m_t^*dt+\gamma_t^*{c^*}'{\Sigma^*}^{-1}\{dY_t-c^*m_t^*dt\}\\
  &\label{eq-gamma}\frac{d\gamma_t^*}{dt}=-a^*\gamma_t^*-\gamma_t^*(a^*)'-{\Sigma^*}^{-1}[(c^*\gamma_t^*)^{\otimes 2}]+{b^*}^{\otimes 2}.
\end{align}

We start with discussing properties of $\gamma_+(\theta_1,\theta_2)$ and $\gamma_t^*$.
\begin{proposition}
  The maximal solution of (\ref{ARE}) $\gamma_+(\theta_1,\theta_2)$ is of class $C^4$.
\end{proposition}
\begin{proof}
  Let $\theta^0=(\theta_1^0,\theta_2^0) \in \Theta_1\times \Theta_2$, and we consider the mapping $f:M_{d_1}(\mathbb{R}) \to M_{d_1}(\mathbb{R})$ such that
  \begin{align*}
    f:X\mapsto a(\theta_2^0)X+Xa(\theta_2^0)'+\Sigma(\theta_1^0)^{-1}[(c(\theta_2^0)X)^{\otimes 2}]-b(\theta_2^0)^{\otimes 2}.
  \end{align*}
  Since for every $T \in M_{d_1}(\mathbb{R})$, we have
  \begin{align*}
    f(X+T)-f(T)=&\left\{a(\theta_2^0)+X'\Sigma(\theta_1^0)^{-1}[c(\theta_2^0)^{\otimes 2}]\right\}T\\
    &+T\left\{a(\theta_2^0)'+\Sigma(\theta_1)^{-1}[c(\theta_2^0)^{\otimes 2}]X\right\}\\
    &+\Sigma(\theta_1^0)^{-1}[(c(\theta_2^0)T)^{\otimes 2}]
  \end{align*}
  and
  \begin{align*}
    \lim_{|T|\to 0}\frac{|\Sigma(\theta_1^0)^{-1}[(c(\theta_2^0)T)^{\otimes 2}]|}{|T|}=0,
  \end{align*}
  the differential of $f$ at $X=\gamma_+(\theta^0)$ is given by
  \begin{align*}
    (df)_{\gamma_+(\theta^0)}:T \mapsto \alpha(\theta_0)T+T\alpha(\theta_0),
  \end{align*}
  where $\alpha$ is defined by (\ref{def-alpha}).

  If $(df)_{\gamma_+(\theta^0)}$ is not injective, $\alpha(\theta_0)$ has eigenvalues $\mu$ and $\lambda$ such that $\mu+\overline{\lambda}=0$(see lemma 2.7 in \cite{robust-and-optimal-control}). However, noting that $\gamma_+(\theta_1,\theta_2)$ is the unique symmetric solution of $f(X)=O$ such that $-\alpha(\theta_0)$ is stable\citep{coppel,robust-and-optimal-control}, there are no such eigenvalues. Therefore $(df)_{\gamma_+(\theta^0)}$ is injective, and by the implicit function theorem, there exists a neighborhood $U \subset \Theta_1 \times \Theta_2$ containing $\theta^0$ and a mapping $\phi:U \to M_{d_1}(\mathbb{R})$ of class $C^4$ such that
  \begin{align*}
    \phi(\theta^0)=\gamma_+(\theta^0),~~f(\phi(\theta))=O~(\theta \in U).
  \end{align*}

  Since $-a(\theta_2)-\phi(\theta)\Sigma(\theta_1)^{-1}[c(\theta_2)^{\otimes 2}]$ is stable at $\theta=(\theta_1,\theta_2)=\theta^0$, it is also stable on a neighborhood of $\theta^0$. Thus by the uniqueness of $\gamma_+$, we obtain $\gamma_+(\theta)=\phi(\theta)$ on that neighborhood and therefore the desired result.  
\end{proof}

By this proposition, Theorem \ref{main-theorem-1} and the mean value theorem, we get the following corollary.
\begin{corollary}
  For any $p\geq 1$, it holds
  \begin{align*}
    &E\left[\sup_{\theta_2 \in \Theta_2}|\gamma_+(\hat{\theta}_1^n,\theta_2)-\gamma_+(\theta_1^*,\theta_2)|^p \right]^\frac{1}{p}\leq Cn^{-\frac{1}{2}}
  \end{align*}
  and
  \begin{align*}
    E\left[\sup_{\theta_2 \in \Theta_2}|\alpha(\hat{\theta}_1^n,\theta_2)-\alpha(\theta_1^*,\theta_2)|^p \right]^\frac{1}{p}\leq Cn^{-\frac{1}{2}}.
  \end{align*}
\end{corollary}

\begin{proposition}\label{gamma-definite}
  For every $\theta_1 \in \overline{\Theta}_1$ and $\theta_2 \in \overline{\Theta}_2$,
  \begin{align}
    \label{gamma-positive}\gamma_+(\theta_1,\theta_2)>0
  \end{align}
  and
  \begin{align}
    \label{gamma-negative}\gamma_-(\theta_1,\theta_2)<0.
  \end{align}
\end{proposition}
\begin{proof}
  Noting that for $A$ and $\gamma \in M_{d_1}(\mathbb{R})$,
    \begin{align*}
      \frac{d}{dt}(\exp(At)\gamma\exp(A't))=\exp(At)(A\gamma+\gamma A')\exp(A't),
    \end{align*}
    and the equation (\ref{ARE}) is equivalent to
    \begin{align*}
      &\left\{a(\theta_2)+\gamma\Sigma(\theta_1)^{-1}[c(\theta_2)^{\otimes 2}]\right\}\gamma+\gamma\left\{a(\theta_2)+\gamma\Sigma(\theta_1)^{-1}[c(\theta_2)^{\otimes 2}]\right\}'\\
      &=\gamma\Sigma(\theta_1)^{-1}[c(\theta_2)^{\otimes 2}]\gamma+b(\theta_2)^{\otimes 2},
    \end{align*}
    we obtain
    \begin{align*}
      &\gamma_+(\theta_1,\theta_2)\\
      =&\int_{-\infty}^0\exp(\alpha(\theta_1,\theta_2)t)\{\alpha(\theta_1,\theta_2)\gamma+\gamma\alpha(\theta_1,\theta_2)'\}\exp(\alpha(\theta_1,\theta_2)'t)dt\\
       =&\int_{-\infty}^0\exp(\alpha(\theta_1,\theta_2)t)\left\{ \Sigma(\theta_1)^{-1}[c(\theta_2)^{\otimes 2}][\gamma_+(\theta_1,\theta_2)^{\otimes 2}]+b(\theta_2)^{\otimes 2} \right\}\\
       &\qquad\times \exp(\alpha(\theta_1,\theta_2)t)dt>0
    \end{align*}
    by assumption [A3], (\ref{def-alpha}) and the stability of $-\alpha(\theta_1,\theta_2)$. In the same way, we can show $\gamma_-(\theta_1,\theta_2)<0.$
\end{proof}

Combining this result with assumption [A3], (\ref{def-alpha}) and Lemma \ref{matrix-exponential}, we obtain the following corollary. 
\begin{corollary}\label{alpha-exponential}
  There exists some constant $C_1>0$ and $C_2>0$ such that
  \begin{align*}
    \sup_{\theta_1\in \Theta_1,\theta_2 \in \Theta_2}|\exp(-\alpha(\theta_1,\theta_2))|\leq C_1e^{-C_2t}.
  \end{align*}
\end{corollary}

Now we go on to the convergence of $\gamma_t^*$. Concerning the convergence rate of Riccati equations, \cite{canonical-form-Riccati} presents the following result.
\begin{theorem}\label{Riccati-converge}(Section 5, \cite{canonical-form-Riccati})\\
  Let $A,B,C \in M_d(\mathbb{R})$ and consider the equation
  \begin{align*}
    \frac{d}{dP}(t)=-A-P(t)B-B'P(t)-P(t)CP(t).
  \end{align*}
  Moreover, assume $C$ is symmetric, $C \leq 0$, $(B,C)$ is controllable and the matrix
  \begin{align*}
    H=\begin{pmatrix}
      B&C\\
      -A&-B'
    \end{pmatrix}
  \end{align*}
  has no pure imaginary eigenvalues.\\
  Then if $P_0-P^+$ is non-singular, then it holds for any $\epsilon>0$ that
  \begin{align*}
    |P(t)-P^-| \leq Ce^{2(r+\epsilon)t} ~(t\to \infty)
  \end{align*}
  and if $P_0-P^-$ is non-singular, then it holds for any $\epsilon>0$ that
  \begin{align*}
    |P(t)-P^+| \leq Ce^{-2(r-\epsilon)t} ~(t\to -\infty),
  \end{align*}
  where $P^+$ and $P^-$ are the maximal and minimal solutions of the algebraic Riccati equation
  \begin{align*}
    A+PB+B'P+PCP=O
  \end{align*}
  respectively, $r<0$ is the maximum real part of the eigenvalues of $B+CP^+$.
\end{theorem}

\begin{proposition}\label{gamma-converge}
  For any $\epsilon>0$, there exists some constant $C>0$ such that
  \begin{align*}
    |\gamma_t^*-\gamma_+(\theta^*)|\leq Ce^{-2\{\lambda_{\min}(\alpha(\theta_2^*))-\epsilon\}t}.
  \end{align*}
  In particular, $|\gamma_t^*|$ is bounded.
\end{proposition}
\begin{proof}
  According to (\ref{eq-gamma}) and Theorem \ref{Riccati-converge}, it is enough show that $\gamma_0^*-\gamma_-(\theta^*)$ is non-singular, where $\gamma_-(\theta_1,\theta_2)$ is the minimal solution of (\ref{ARE}). If we assume $\gamma_0^*-\gamma_-(\theta^*)$ is singular, there exists $x \in \mathbb{R}^{d_1}\backslash \{0\}$ such that $\{\gamma_0^*-\gamma_-(\theta^*)\}x=0$, and we get $x\gamma_0^*x=x\gamma_-(\theta^*)x$. However, since $\gamma_0^*\geq 0$ and we have $\gamma_-(\theta^*)<0$ by Proposition \ref{gamma-definite}, that is a contradiction.
\end{proof}

Next we consider the innovation process
\begin{align*}
  \overline{W}_t=(\sigma^*)^{-1}\left( Y_t-\int_0^tc^*m_s^*ds \right).
\end{align*}
Note that the right-hand side is well-defined since $\{m_t^*\}$ has a progressively measurable modification, and that $\overline{W}_t$ is also a Wiener process\citep{Kallianpur}. Since $Y_t$ is the solution of
\begin{align}
  \label{innovation}dY_t=c^*m_t^*dt+\sigma^*d\overline{W}_t,
\end{align}
we obtain together with (\ref{eq-m*})
\begin{align*}
  dm_t^*=-a^*m_t^*dt+\gamma_t^*{c^*}'{{\sigma^*}'}^{-1}d\overline{W}_t.
\end{align*}
Therefore It\^{o}'s formula gives 
\begin{align}
  \label{sol-m^*}m_t^*=\exp(-a^*t)m_0^*+\int_0^t\exp(-a^*(t-s))\gamma_s^*{c^*}'{{\sigma^*}'}^{-1}d\overline{W}_s.
\end{align}
Moreover, using Proposition \ref{gamma-converge}, we can show for any $p\geq 1$,
\begin{align}
  \label{m*-norm}\sup_{t\geq 0}E[|m_t^*|^p]\leq C_p
\end{align}
and
\begin{align}
  \label{m*-conti}\sup_{0\leq t-s\leq 1}E[|m_t^*-m_s^*|^p]\leq C_p(t-s)^{\frac{p}{2}}
\end{align}
in the same way as Lemma \ref{X-Y-norm}.

\begin{lemma}\label{pre-ineq}
  For $j=0,1,2,\cdots$ and $\theta \in\Theta$, let $Z_j(\theta)$ be a $M_{k,l}(\mathbb{R})$-valued and $\mathcal{F}_{t_j}$-measurable random variable, and $U(\theta)$ be an $M_{l,d}(\mathbb{R})$-valued random variable. Moreover, we assume $Z_j(\theta)$ is continuously differentiable with respect to $\theta$. Then for any $n \in \mathbb{N}$ and $p>m_1+m_2$, it holds
    \begin{align*}
      &E\left[\sup_{\theta \in \Theta}\left|\sum_{j=1}^nZ_{j-1}(\theta)U(\theta)\Delta_jW\right|^p \right]\\
      \leq &C_{d,k,l}E\left[\sup_{\theta \in \Theta}\left|U(\theta)\right|^{2p}\right]^\frac{1}{2}\\
      &\times \sup_{\theta \in \Theta}\left\{E\left[ \left\{\sum_{j=1}^n|Z_{j-1}(\theta)|^2h\right\}^{p}\right]+E\left[ \left\{\sum_{j=1}^n\left|\partial_{\theta}Z_{j-1}(\theta)\right|^2h\right\}^{p}\right]  \right\}^\frac{1}{2}.
    \end{align*}
\end{lemma}
\begin{proof}
  Let $Z_j^{(ij)}$, $U^{(ij)}$ and $(Z_jU)^{(ij)}$ be the $(i,j)$ entries of $Z_j$, $U$ and $Z_jU$, respectively, and $W^{(j)}$ be the $j$-th element of $W^{(j)}$. Then we have
  \begin{align}
    \label{eq8}\begin{split}
      &E\left[ \sup_{\theta \in \Theta}\left|\sum_{j=1}^nZ_{j-1}(\theta)U(\theta)\Delta_jW\right|^p \right]\\
    = &E\left[ \sup_{\theta \in \Theta}\left\{\sum_{p=1}^k\left(\sum_{j=1}^n\sum_{q=1}^d(Z_{j-1}U)^{(pq)}(\theta)\Delta_jW^{(q)}\right)^2 \right\}^{\frac{p}{2}}\right]\\
    =&E\left[ \sup_{\theta \in \Theta}\left\{\sum_{p=1}^k\left(\sum_{j=1}^n\sum_{q=1}^d\sum_{r=1}^lZ_{j-1}^{(pr)}(\theta)U^{(rq)}(\theta)\Delta_jW^{(q)}\right)^2 \right\}^{\frac{p}{2}}\right]\\
    \leq &C_{d,k,l}\sum_{p=1}^k\sum_{q=1}^d\sum_{r=1}^lE\left[ \sup_{\theta \in \Theta}\left|\sum_{j=1}^nZ_{j-1}^{(pr)}(\theta)U^{(rq)}(\theta)\Delta_jW^{(q)}\right|^p \right]\\
    \leq &C_{d,k,l}\sum_{p=1}^k\sum_{q=1}^d\sum_{r=1}^lE\left[\sup_{\theta \in \Theta}\left|U^{(rq)}(\theta)\right|^{2p}\right]^\frac{1}{2}E\left[\sup_{\theta \in \Theta}\left|\sum_{j=1}^nZ_{j-1}^{(pr)}(\theta)\Delta_jW^{(q)}\right|^{2p} \right]^\frac{1}{2}\\
    \leq &C_{d,k,l}E\left[\sup_{\theta \in \Theta}\left|U(\theta)\right|^{2p}\right]^\frac{1}{2}\sum_{p=1}^k\sum_{r=1}^lE\left[\sup_{\theta \in \Theta}\left|\sum_{j=1}^nZ_{j-1}^{(pr)}(\theta)\Delta_jW^{(q)}\right|^{2p} \right]^\frac{1}{2}.
    \end{split}    
  \end{align}
  Moreover, the Sobolev inequality and the Burkholder-Davis-Gundy inequality gives
  \begin{align}
    \label{eq9}\begin{split}
      &E\left[\sup_{\theta \in \Theta}\left|\sum_{j=1}^nZ_{j-1}^{(pr)}(\theta)\Delta_jW^{(q)}\right|^{2p} \right]\\
    \leq &C_p\sup_{\theta \in \Theta}\left\{E\left[ \left|\sum_{j=1}^nZ_{j-1}^{(pr)}(\theta)\Delta_jW^{(q)}\right|^{2p}\right]+E\left[ \left|\sum_{j=1}^n\frac{\partial}{\partial\theta}Z_{j-1}^{(pr)}(\theta)\Delta_jW^{(q)}\right|^{2p}\right]  \right\}\\
    \leq &C_p\sup_{\theta \in \Theta}\left\{E\left[ \left|\sum_{j=1}^nZ_{j-1}^{(pr)}(\theta)^2h\right|^{p}\right]+E\left[ \left|\sum_{j=1}^n\left\{\frac{\partial}{\partial\theta}Z_{j-1}^{(pr)}(\theta)\right\}^2h\right|^{p}\right]  \right\}\\
    \leq &C_p\sup_{\theta \in \Theta}\left\{E\left[ \left(\sum_{j=1}^n|Z_{j-1}(\theta)|^2h\right)^{p}\right]+E\left[ \left(\sum_{j=1}^n\left|\partial_{\theta}Z_{j-1}(\theta)\right|^2h\right)^{p}\right]  \right\}.
    \end{split}    
  \end{align}
  By (\ref{eq8}) and (\ref{eq9}), we obtain the desired result.
\end{proof}

\begin{lemma}\label{differentiable-lemma}
  For every $\theta \in \Theta$, let $\{Z_t(\theta)\}$ be a $\mathbb{M}_{d,d_1}(\mathbb{R})$-valued progressively measurable process. Moreover, we assume $Z_t(\theta)$ is differentiable with respect to $\theta$, and for any $T>0, p>0$ and $\theta,\theta' \in \Theta$
  \begin{align*}
    &\sup_{0\leq t\leq T}E\left[|Z_t(\theta)-Z_t(\theta')|^p \right]\leq C_{T,p}|\theta-\theta'|^p,\\
    &\sup_{0\leq t\leq T}E\left[|\partial_{\theta}Z_t(\theta)-\partial_{\theta}Z_t(\theta')|^p \right]\leq C_{T,p}|\theta-\theta'|^p.
  \end{align*}
  Then $\{\xi_{\cdot}(\theta)\}_{\theta \in \Theta}$ with $\displaystyle \xi_t(\theta)=\int_0^tZ_t(\theta)d\overline{W}_s$ has a modification $\{\tilde{\xi}_{\cdot}(\theta)\}_{\theta \in \Theta}$ which is continuously differentiable with respect to $\theta$. Moreover, it holds almost surely for any $t\geq 0$ and $\theta \in \Theta$
  \begin{align*}
    \partial_{\theta}\tilde{\xi}_t(\theta)=\int_0^t\partial_{\theta}Z_t(\theta)d\overline{W}_s.
  \end{align*}
\end{lemma}
\begin{proof}
  For any matrix valued function $\phi$ on $\mathbb{R}^{m_1+m_2}$ and $\epsilon>0$, let
  \begin{align*}
    \Delta^j\phi(\theta;\epsilon)=\frac{1}{\epsilon}\{\xi_t(\theta+\epsilon e_j)-\xi_t(\theta)\},
  \end{align*} 
  where $e_1,\cdots,e_{m_1+m_2}$ is the standard basis of $\mathbb{R}^{m_1+m_2}$.
  Then for $\theta,\theta'\in \Theta,\epsilon,\epsilon'>0$ and $p\geq 1$, we have 
  \begin{align*}
    &\sup_{0\leq t\leq T}E\left[\left|\Delta^jZ_t(\theta;\epsilon)-\Delta^jZ_t(\theta;\epsilon')\right|^p\right]\\
    =&\sup_{0\leq t\leq T}E\left[\left|\int_0^1\frac{\partial}{\partial\theta^j}Z_t(\theta+u\epsilon e_j)du-\int_0^1\frac{\partial}{\partial\theta^j}Z_s(\theta+u\epsilon' e_j)du\right|^p\right]\\
    \leq &\int_0^1\sup_{0\leq t\leq T}E\left[ \left|\frac{\partial}{\partial\theta^j}Z_t(\theta+u\epsilon e_j)-\frac{\partial}{\partial\theta^j}Z_t(\theta'+u\epsilon' e_j)\right| \right]du\\
    \leq &C_{p,T}(|\theta-\theta'|+|\epsilon-\epsilon'|),
  \end{align*}
  where $\theta=(\theta^1,\cdots,\theta^{m_1+m_2})$.

  Hence by Lemma \ref{BDG}, it follows for any $\theta,\theta' \in \Theta,\epsilon,\epsilon'>0$ and $N \in \mathbb{N}$
  \begin{align*}
    &E\left[ \sup_{0\leq t \leq N}\left|\Delta^j \xi_t(\theta;\epsilon)-\Delta^j \xi_t(\theta';\epsilon')\right|^p \right]\\
    =&E\left[ \sup_{0\leq t \leq N}\left|\int_{0}^t\{\Delta^jZ_t(\theta;\epsilon)-\Delta^j Z_t(\theta';\epsilon')\}d\overline{W}_s\right|^p \right]\\
    \leq &C_pN^{\frac{p}{2}-1}\int_{0}^NE\left[\left|\Delta^jZ_t(\theta;\epsilon)-\Delta^j Z_t(\theta';\epsilon')\right|^p\right]ds\\
    \leq &C_{p,N}(|\theta-\theta'|+|\epsilon-\epsilon'|).
  \end{align*}
  Now for this $C_{p,N}$, we take a sequence $\alpha_N>0~(N \in \mathbb{N})$ so that
  \begin{align*}
    S_p=\sum_{n=1}^\infty \alpha_NC_{p,N}<\infty,~~\sum_{n=1}^\infty \alpha_N<\infty,
  \end{align*}
  and define the norm on $C(\mathbb{R}_+;M_{d,d_1}(\mathbb{R}))$ by
  \begin{align*}
    \|A\|=\sum_{N=1}^\infty \alpha_N\left( \sup_{0\leq t \leq N}|A(s)|\wedge 1 \right).
  \end{align*}
  Then the topology induced by this norm is equivalent to the topology of uniform convergence, and we have
  \begin{align}
    \label{delta-xi-continuity}E\left[ \left\|\Delta^j \xi.(\theta;\epsilon)-\Delta^j \xi.(\theta';\epsilon')\right\|^p \right]\leq C_p(|\theta-\theta'|+|\epsilon-\epsilon'|).
  \end{align}
  Therefore, by the the Kolmogorov continuity theorem, $\{\Delta^j \xi.(\theta;\epsilon)\}_{\theta \in \Theta,0<|\epsilon|\leq 1}$ has a uniformly continuous modification $\{\zeta.(\theta;\epsilon)\}_{\theta \in \Theta,0<|\epsilon|\leq 1}$. Because of the uniform continuity, $\zeta.(\theta;\epsilon)$ can be extended to a continuous process on $\theta \in \Theta,|\epsilon|\leq 1$.

  On the other hand, we can show in the same way that $\{\xi.(\theta;\epsilon)\}_{\theta \in \Theta}$ has a continuous modification $\{\tilde{\xi}.(\theta;\epsilon)\}_{\theta \in \Theta}$. Then $\Delta^j\tilde{\xi}.(\theta;\epsilon)$ and $\zeta.(\theta;\epsilon)$ are both continuous modifications of $\Delta^j \xi.(\theta;\epsilon)$, and thus they are indistinguishable. Therefore almost surely for any $t \geq 0$ and $\theta \in \Theta$,
  \begin{align*}
    \frac{\partial \tilde{\xi}_t}{\partial\theta_j}(\theta)=\lim_{\epsilon \to 0}\frac{\xi_t(\theta+\epsilon e_j)-\xi(\theta)}{\epsilon}=\lim_{\epsilon \to 0}\Delta^j \xi_t(\theta;\epsilon)
  \end{align*}
  exists. The continuity of $\displaystyle \frac{\partial \xi_t}{\partial\theta_j}(\theta)$ follows from the continuity of $\zeta.(\theta,\epsilon)$.

  Moreover, by the assumption and Lemma \ref{BDG} (2), we have for $p\geq 2$,
  \begin{align*}
    &E\left[ \left|\int_0^t\left\{\frac{1}{\epsilon}(Z_s(\theta+\epsilon e_j)-Z_s(\theta))-\frac{\partial}{\partial \theta^j}Z_s(\theta)\right\}d\overline{W}_s\right|^p \right]\\
    =&E\left[ \left|\int_0^t\left\{\frac{\partial Z_s}{\partial\theta_j}(\theta+\eta_s\epsilon e_j)-\frac{\partial}{\partial \theta^j}Z_s(\theta)\right\}d\overline{W}_s\right|^p \right]\\
    \leq &C_pt^{\frac{p}{2}-1}\int_0^t E\left[ \left|\frac{\partial Z_s}{\partial\theta_j}(\theta+\eta_s\epsilon e_j)-\frac{\partial Z_s}{\partial \theta^j}(\theta)\right|^p \right]ds\\
    \leq &C_{p,t}\epsilon \to 0~(\epsilon \to 0),
  \end{align*}
  where $0\leq \eta_s\leq 1$. This means
  \begin{align*}
    \Delta^j\xi_t(\theta;\epsilon)\to \int_0^s\frac{\partial}{\partial \theta^j}Z_s(\theta)ds~~(\epsilon \to 0)
  \end{align*}
  in $L^p$, and hence there exists a subsequence $\{\epsilon_n\}_{n \in \mathbb{N}}$ such that $\epsilon_n \to 0$ and
  \begin{align*}
    \Delta^j\xi_t(\theta;\epsilon_n)\xrightarrow{\mathrm{a.s.}} \int_0^s\frac{\partial}{\partial \theta^j}Z_s(\theta)ds~~(n \to \infty).
  \end{align*}
  Therefore we obtain almost surely 
  \begin{align*}
    \frac{\partial }{\partial \theta^j}\tilde{\xi}_t(\theta)=\Delta^j\tilde{\xi}_t(\theta;0)=\int_0^s\frac{\partial}{\partial \theta^j}Z_s(\theta)ds.
  \end{align*}

\end{proof}

\begin{lemma}
  (1) For $j \in \mathbb{N}$, let $f_j:[t_{j-1},t_j]\times \Theta \to M_{k,d_2}(\mathbb{R})$ be of class $C^1$. Then for any $p>m_1+m_2$, it holds
  \begin{align}
    \begin{split}
      \label{ineq-lemma1}&E\left[ \sup_{\theta \in\Theta}\left|\sum_{j=1}^i\int_{t_{j-1}}^{t_j}f_{j-1}(s,\theta)dY_s\right|^p \right]^{\frac{1}{p}}\\
    \leq &C_p\sup_{\theta \in \Theta}\left\{\sum_{j=1}^i\int_{t_{j-1}}^{t_j}|f_j(s,\theta)|ds+\sum_{j=1}^i\int_{t_{j-1}}^{t_j}|\partial_{\theta}f_j(s,\theta)|ds\right.\\
    &\left.+\left(\sum_{j=1}^i\int_{t_{j-1}}^{t_j}|f_{j-1}(s,\theta)|^2ds\right)^{\frac{1}{2}}
    +\left(\sum_{j=1}^i\int_{t_{j-1}}^{t_j}|\partial_{\theta}f_{j-1}(s,\theta)|^2ds\right)^{\frac{1}{2}}\right\},
    \end{split}
  \end{align}
  where $C_p$ is a constant which depends only on $p$.\\
  (2) For $j=0,1,2,\cdots$ and $\theta \in\Theta$, let $Z_j(\theta)$ be a $M_{k,l}(\mathbb{R})$-valued and $\mathcal{F}_{t_j}$-measurable random variable, and $U(\theta)$ be an $M_{l,d}(\mathbb{R})$-valued random variable. Moreover, we assume $Z_j(\theta)$ is continuously differentiable with respect to $\theta$. Then for any $p>m_1+m_2$, it holds
  \begin{align}
    \begin{split}
    \label{ineq-lemma2}&E\left[ \sup_{\theta \in\Theta}\left|\sum_{j=1}^iZ_{j-1}(\theta)U(\theta)\Delta_j Y\right|^p \right]^{\frac{1}{p}}\\
    \leq &C_pE\left[\sup_{\theta \in\Theta}|U(\theta)|^{4p}\right]^\frac{1}{4p}\\
    &\times \sup_{\theta \in \Theta}\left\{ \sum_{j=1}^iE\left[|Z_{j-1}(\theta)|^{2p}\right]^\frac{1}{2p}h+\sum_{j=1}^iE\left[|\partial_{\theta}Z_{j-1}(\theta)|^{2p}\right]^\frac{1}{2p}h\right\}\\
    &+C_{p}E\left[\sup_{\theta \in \Theta}\left|U(\theta)\right|^{2p}\right]^\frac{1}{2p}\\
    &\times \sup_{\theta \in \Theta}\left\{ \sum_{j=1}^nE[|Z_{j-1}(\theta)|^{2p}]^\frac{1}{p}h+\sum_{j=1}^nE[\left|\partial_{\theta}Z_{j-1}(\theta)\right|^{2p}]^\frac{1}{p}h  \right\}^\frac{1}{2}.
    \end{split}
  \end{align}
\end{lemma}
\begin{proof}
  (1) By Lemma \ref{differentiable-lemma}, we can assume for every $j$
  \begin{align*}
    \int_{t_{j-1}}^{t_j}f_{j-1}(s,\theta)dY_s
    =\int_{t_{j-1}}^{t_j}f_{j-1}(s,\theta)c^*m_s^*ds+\int_{t_{j-1}}^{t_j}f_{j-1}(s,\theta)\sigma^*\overline{W}_s
  \end{align*}
  is continuously differentiable, and 
  \begin{align*}
    \partial_{\theta}\int_{t_{j-1}}^{t_j}f_{j-1}(s,\theta)dY_s=\int_{t_{j-1}}^{t_j}\partial_{\theta}f_{j-1}(s,\theta)dY_s.
  \end{align*}
  Therefore by the Sobolev inequality and (\ref{innovation}),
  \begin{align}
      &E\left[ \sup_{\theta \in\Theta}\left|\sum_{j=1}^i\int_{t_{j-1}}^{t_j}f_{j-1}(s,\theta)dY_s\right|^p \right]^{\frac{1}{p}}
    \nonumber\\
    \leq &C_p\sup_{\theta \in\Theta}\left(E\left[\left|\sum_{j=1}^i\int_{t_{j-1}}^{t_j}f_{j-1}(s,\theta)dY_s\right|^p \right]^{\frac{1}{p}}\right.\nonumber\\
    &\left.+E\left[\left|\sum_{j=1}^i\int_{t_{j-1}}^{t_j}\partial_{\theta}f_{j-1}(s,\theta)dY_s\right|^p \right]^{\frac{1}{p}}\right)\nonumber\\
    \label{eq1}\begin{split}
    \leq &C_p\sup_{\theta \in\Theta}\left(E\left[ \left|\sum_{j=1}^i\int_{t_{j-1}}^{t_j}f_{j-1}(s,\theta)c^*m_s^*ds\right|^p \right]^{\frac{1}{p}}\right.\\
    &+E\left[ \left|\sum_{j=1}^i\int_{t_{j-1}}^{t_j}f_{j-1}(s,\theta)\sigma^*d\overline{W}_s\right|^p \right]^{\frac{1}{p}}\\
    &+E\left[ \left|\sum_{j=1}^i\int_{t_{j-1}}^{t_j}\partial_{\theta}f_{j-1}(s,\theta)c^*m_s^*ds\right|^p \right]^{\frac{1}{p}}\\
    &\left.+E\left[ \left|\sum_{j=1}^i\int_{t_{j-1}}^{t_j}\partial_{\theta}f_{j-1}(s,\theta)\sigma^*d\overline{W}_s\right|^p \right]^{\frac{1}{p}}\right).
    \end{split}
  \end{align}
  In order to bound the first term of (\ref{eq1}), we set
  \begin{align*}
    f(s,\theta)=\sum_{j=1}^i f_j(s,\theta)1_{(t_{j-1},t_j]}(s).
  \end{align*}
  Then we have 
  \begin{align*}
    \begin{split}
      &E\left[\left|\sum_{j=1}^i\int_{t_{j-1}}^{t_j}f_{j-1}(s,\theta)m_s^*ds\right|^p  \right]
    =E\left[ \left|\int_{0}^{t_{i}}f(s,\theta)c^*m_s^*ds\right|^p \right]\\
    \leq &E\left[ \left(\int_{0}^{t_{i}}|f(s,\theta)c^*m_s^*|ds\right)^p \right]\\
    \leq &|c^*|^pE\left[\left|\left(\int_{0}^{t_{i}}|f(s,\theta)|^\frac{1}{p}|m_s^*|^pds\right)^\frac{1}{p}\left(\int_0^{t_i}|f(s,\theta)|ds\right)^{1-\frac{1}{p}}\right|^p \right]\\
    \leq &|c^*|^p\left(\int_0^{t_i}|f(s,\theta)|ds\right)^{p-1}\int_{0}^{t_{i}}|f(s,\theta)|E[|m_s^*|^p]ds\\
    \leq &C_p\left(\int_0^{t_i}|f(s,\theta)|ds\right)^{p}
    =C_p\left( \sum_{j=1}^i\int_{t_{j-1}}^{t_j}|f_j(s,\theta)|ds \right)^p.
    \end{split}    
  \end{align*}
  In the same way, it holds for the third term
  \begin{align*}
    E\left[\left|\sum_{j=1}^i\int_{t_{j-1}}^{t_j}\partial_{\theta}Z_{j-1}(s,\theta)c^*m_s^*ds\right|^p  \right]^{\frac{1}{p}}\leq C_p\left( \sum_{j=1}^i\int_{t_{j-1}}^{t_j}|\partial_{\theta}f_j(s,\theta)|ds \right)^p.
  \end{align*}

  Next by Lemma \ref{BDG} (2), we obtain for the second term
  \begin{align*}
    \begin{split}
      &E\left[ \left|\sum_{j=1}^i\int_{t_{j-1}}^{t_j}f_{j-1}(s,\theta)\sigma^*d\overline{W}_{s}\right|^p \right]
    = E\left[ \left|\int_{0}^{t_i}f(s,\theta)\sigma^*d\overline{W}_{s}\right|^p \right]^{\frac{1}{p}}\\
    \leq &C_p \left(\int_{0}^{t_i}\sum_{j=1}^i|f(s,\theta)|^2ds\right)^{\frac{1}{2}} 
    =C_p\left(\sum_{j=1}^i\int_{t_{j-1}}^{t_j}|f_{j-1}(s,\theta)|^2ds\right)^{\frac{1}{2}}
    \end{split}    
  \end{align*}
  and in the same way it holds for the fourth term
  \begin{align*}
    E\left[ \left|\sum_{j=1}^i\int_{t_{j-1}}^{t_j}\partial_{\theta}f_{j-1}(s,\theta)d\overline{W}_{s}\right|^p \right]
    \leq C_p\left(\sum_{j=1}^i\int_{t_{j-1}}^{t_j}|\partial_{\theta}f_{j-1}(s,\theta)|^2ds\right)^{\frac{1}{2}}.
  \end{align*}
  We complete the proof by the above inequalities.\\
  (2) By the Sobolev inequality and (\ref{innovation}),
  \begin{align}
    &E\left[ \sup_{\theta \in\Theta}\left|\sum_{j=1}^iZ_{j-1}(\theta)U(\theta)\Delta_j Y\right|^p \right]^{\frac{1}{p}}\nonumber\\
    \begin{split}
      \leq &C_p\left( E\left[\sup_{\theta \in\Theta} \left|\sum_{j=1}^iZ_{j-1}(\theta)\int_{t_{j-1}}^{t_j}U(\theta)c^*m_s^*ds\right|^p \right]^{\frac{1}{p}}\right.\\
      &\left.+E\left[\sup_{\theta \in\Theta}\left|\sum_{j=1}^iZ_{j-1}(\theta)U(\theta)\sigma^*(W_{t_j}-W_{t_{j-1}})\right|^p \right]^{\frac{1}{p}}\right).
    \end{split}
  \end{align}
  For the first term of the right-hand side, it follows from Lemma \ref{BDG} (1), (\ref{m*-norm}) and the Sobolev inequality
  \begin{align*}
    &E\left[\sup_{\theta \in\Theta}\left|\sum_{j=1}^iZ_{j-1}(\theta)\int_{t_{j-1}}^{t_j}U(\theta)c^*m_s^*ds\right|^p \right]^{\frac{1}{p}}\\
    \leq &E\left[ \sum_{j=1}^i\left|\sup_{\theta \in\Theta}\int_{t_{j-1}}^{t_j}Z_{j-1}(\theta)U(\theta)c^*m_s^*ds\right|^p \right]^{\frac{1}{p}}\\
    \leq &\left( \sum_{j=1}^ih^{p-1}\int_{t_{j-1}}^{t_j}E\left[\sup_{\theta \in\Theta}|Z_{j-1}(\theta)U(\theta)c^*m_s^*|^p\right]ds \right)^{\frac{1}{p}}\\
    \leq &|c|^*\left( \sum_{j=1}^ih^{p-1}\int_{t_{j-1}}^{t_j}E\left[\sup_{\theta \in\Theta}|Z_{j-1}(\theta)|^{2p}\right]^\frac{1}{2}\right.\\
    &\left.\qquad\qquad\times E\left[\sup_{\theta \in\Theta}|U(\theta)|^{4p}\right]^\frac{1}{4}E[|m_s^*|^{4p}]^\frac{1}{4}ds \right)^{\frac{1}{p}}\\
    \leq &C_pE\left[\sup_{\theta \in\Theta}|U(\theta)|^{4p}\right]^\frac{1}{4}\left( \sum_{j=1}^ih^{p}E\left[\sup_{\theta \in\Theta}|Z_{j-1}(\theta)|^{2p}\right]^\frac{1}{2} \right)^{\frac{1}{p}}\\
    \leq &C_pE\left[\sup_{\theta \in\Theta}|U(\theta)|^{4p}\right]^\frac{1}{4}\\
    &\times \left( \sum_{j=1}^ih^{p}\sup_{\theta \in\Theta}\left\{E\left[|Z_{j-1}(\theta)|^{2p}\right]+E\left[|\partial_{\theta}Z_{j-1}(\theta)|^{2p}\right]\right\}^\frac{1}{2} \right)^{\frac{1}{p}}\\
    \leq &C_pE\left[\sup_{\theta \in\Theta}|U(\theta)|^{4p}\right]^\frac{1}{4}\\
    &\times \sup_{\theta \in\Theta}\sum_{j=1}^i\left\{E\left[|Z_{j-1}(\theta)|^{2p}\right]^\frac{1}{2p}+E\left[|\partial_{\theta}Z_{j-1}(\theta)|^{2p}\right]^\frac{1}{2p}\right\}h.
  \end{align*}

  As for the second term, we have by Lemma \ref{pre-ineq}
  \begin{align*}
    &E\left[\sup_{\theta \in\Theta}\left|\sum_{j=1}^iZ_{j-1}(\theta)U(\theta)\sigma^*(W_{t_j}-W_{t_{j-1}})\right|^p \right]^{\frac{1}{p}}\\
    \leq &C_{p}E\left[\sup_{\theta \in \Theta}\left|U(\theta)\right|^{2p}\right]^\frac{1}{2p}\\
      &\times \sup_{\theta \in \Theta}\left\{E\left[ \left(\sum_{j=1}^n|Z_{j-1}(\theta)|^2h\right)^{p}\right]+E\left[ \left(\sum_{j=1}^n\left|\partial_{\theta}Z_{j-1}(\theta)\right|^2h\right)^{p}\right]  \right\}^\frac{1}{2p}\\
      \leq &C_{p}E\left[\sup_{\theta \in \Theta}\left|U(\theta)\right|^{2p}\right]^\frac{1}{2p}\\
      &\times \sup_{\theta \in \Theta}\left\{E\left[ \left(\sum_{j=1}^n|Z_{j-1}(\theta)|^2h\right)^{p}\right]^\frac{1}{p}+E\left[ \left(\sum_{j=1}^n\left|\partial_{\theta}Z_{j-1}(\theta)\right|^2h\right)^{p}\right]^\frac{1}{p}  \right\}^\frac{1}{2}\\
      \leq &C_{p}E\left[\sup_{\theta \in \Theta}\left|U(\theta)\right|^{2p}\right]^\frac{1}{2p}\\
      &\times \sup_{\theta \in \Theta}\left\{ \sum_{j=1}^nE[|Z_{j-1}(\theta)|^{2p}]^\frac{1}{p}h+\sum_{j=1}^nE[\left|\partial_{\theta}Z_{j-1}(\theta)\right|^{2p}]^\frac{1}{p}h  \right\}^\frac{1}{2}.
  \end{align*}
  Thus we completed the proof.
\end{proof}

\begin{proposition}\label{m-hat-norm}
  For any $p>m_1+m_2$, it holds
  \begin{align*}
    &\sup_{i \in \mathbb{N}}E\left[\sup_{\theta_2 \in \Theta_2}|\hat{m}_i^n(\theta_2)|^{p}\right]<\infty\\
    &\sup_{i \in \mathbb{N}}E\left[\sup_{\theta_2 \in \Theta_2}\left|\partial_{\theta_2}\hat{m}_i^n(\theta_2)\right|^{p}\right]<\infty\\
    &\sup_{i \in \mathbb{N}}E\left[\sup_{\theta_2 \in \Theta_2}\left|\partial_{\theta_2}^2\hat{m}_i^n(\theta_2)\right|^{p}\right]<\infty
  \end{align*}
  and
  \begin{align*}
    \sup_{i \in \mathbb{N}}E\left[\sup_{\theta_2 \in \Theta_2}\left|\partial_{\theta_2}^3\hat{m}_i^n(\theta_2)\right|^{p}\right]<\infty.
  \end{align*}
\end{proposition}
\begin{proof}
  We only prove the first one; the rest can be shown in the same way. By (\ref{def-m-hat}) and the stability of $-\alpha(\theta_1,\theta_2)$, it is enough to show
  \begin{align*}
    \sup_{i \in \mathbb{N}}E\left[\left|\sup_{\theta=(\theta_1,\theta_2)\in\Theta}\sum_{j=1}^i\exp\left( -\alpha(\theta)(t_i-t_{j-1})\right)\right.\right.\\
    \left.\left.\gamma_+(\theta)c(\theta_2)'\Sigma(\theta_1)^{-1}\Delta_jY\right.\Biggr|  \right.\Biggr]<\infty.
  \end{align*}
  To accomplish this, it is enough to show 
  \begin{align}
    \label{eq3}&\sum_{j=1}^i\left|\exp\left( -\alpha(\theta)(t_i-t_{j-1})\right)\gamma_+(\theta)c(\theta_2)'\Sigma(\theta_1)^{-1}\right|h<C\\
    \label{eq4}&\sum_{j=1}^i\left|\exp\left( -\alpha(\theta)(t_i-t_{j-1})\right)\gamma_+(\theta)c(\theta_2)'\Sigma(\theta_1)^{-1}\right|^2h<C\\
    \label{eq5}&\sum_{j=1}^i\left|\partial_{\theta}\left\{\exp\left( -\alpha(\theta)(t_i-t_{j-1})\right)\gamma_+(\theta)c(\theta_2)'\Sigma(\theta_1)^{-1}\right\}\right|h<C\\
    \label{eq6}&\sum_{j=1}^i\left|\partial_{\theta}\left\{\exp\left( -\alpha(\theta)(t_i-t_{j-1})\right)\gamma_+(\theta)c(\theta_2)'\Sigma(\theta_1)^{-1}\right\}\right|^2h<C
  \end{align}
  according to (\ref{ineq-lemma2}).

  Using Corollary \ref{alpha-exponential}, we can show (\ref{eq3}):
  \begin{align*}
    &\sum_{j=1}^i\left|\exp\left( -\alpha(\theta)(t_i-t_{j-1})\right)\gamma_+(\theta)c(\theta_2)'\Sigma(\theta_1)^{-1}\right|h\\
    \leq &\sum_{j=1}^i\left|\exp\left( -\alpha(\theta)(t_i-t_{j-1})\right)\right|h\\
    \leq &\sum_{j=1}^iC_1e^{-C_2(t_i-t_{j-1})}h\\
    \leq &C_1\int_0^{t_i}e^{-C_2(t_i-s)}ds\leq \frac{C_1}{C_2},
  \end{align*}
  where $C_1$ and $C_2$ are positive constants. In the same way, we obtain (\ref{eq4})-(\ref{eq6}) noting that
  \begin{align*}
    &\sum_{j=1}^i(t_i-t_{j-1})e^{-C(t_i-t_{j-1})}h\leq \sum_{j=1}^i2e^{-\frac{1}{2}C(t_i-t_{j-1})}h,\\
    &\sum_{j=1}^i(t_i-t_{j-1})^2e^{-C(t_i-t_{j-1})}h\leq \sum_{j=1}^i8e^{-\frac{1}{2}C(t_i-t_{j-1})}h
  \end{align*}
  and it holds by \cite{MatrixExponentialNote} 
  \begin{align*}
    &\left|\partial_{\theta}\exp\left( -\alpha(\theta)(t_i-t_{j-1})\right)\right|\\
    =&\left|-\int_0^1\exp(-s\alpha(\theta)(t_i-t_{j-1}))\partial_{\theta}\alpha(\theta)(t_i-t_{j-1})\right.\\
    &\left.\exp(-(1-s)\alpha(\theta)(t_i-t_{j-1}))ds\right.\biggr|\\
    \leq &C(t_i-t_{j-1})e^{-C(t_i-t_{j-1})}.
  \end{align*}
\end{proof}

\begin{proposition}\label{m-m-hat-diff}
  For any $n,i \in \mathbb{N}$ and $p>m_1+m_2$,
    \begin{align*}
      E\left[\sup_{\theta_2 \in \Theta_2}|m_{t_i}(\theta_2)-\hat{m}_{i}^n(\theta_2)|^{p}\right]^\frac{1}{p}\leq C_p(n^{-\frac{1}{2}}+h).
    \end{align*}
\end{proposition}
\begin{proof}
  By (\ref{def-m}) and (\ref{def-m-hat}), we have
  \begin{align}
    &E\left[\sup_{\theta_2 \in \Theta_2}|m_{t_i}(\theta_2)-\hat{m}_{i}^n(\theta_2)|^{p}\right]^\frac{1}{p}\nonumber\\
    \leq &E\left[ \sup_{\theta_2 \in \Theta_2}\left|\left\{\exp(-\alpha(\theta_1^*,\theta_2)t)-\exp(-\alpha(\hat{\theta}_1^n,\theta_2)t)\right\}m_0\right|^p \right]^\frac{1}{p}\nonumber\\
    &+E\left[\sup_{\theta_2 \in \Theta_2}\left|\int_0^{t_i}\exp\left( -\alpha(\theta_1^*,\theta_2)(t_i-s)\right)\gamma_+(\theta_1^*,\theta_2)c(\theta_2)'{\Sigma^*}^{-1}dY_s\right.\right.\nonumber\\
    &\left.\left.\qquad-\sum_{j=1}^i\exp\left( -\alpha(\hat{\theta}_1^n,\theta_2)(t_i-t_{j-1})\right)\gamma_+(\hat{\theta}_1^n,\theta_2)c(\theta_2)'\Sigma(\hat{\theta}_1^n)^{-1}\Delta_jY\right|^p \right]^\frac{1}{p}\nonumber\\
      \leq &E\left[ \sup_{\theta_2 \in \Theta_2}\left|\left\{\exp(-\alpha(\theta_1^*,\theta_2)t)-\exp(-\alpha(\hat{\theta}_1^n,\theta_2)t)\right\}m_0\right|^p \right]^\frac{1}{p}\nonumber\\
    &+E\left.\Biggl[\sup_{\theta_2 \in \Theta_2}\left.\Biggl|\int_0^{t_i}\exp\left( -\alpha(\theta_1^*,\theta_2)(t_i-s)\right)\gamma_+(\theta_1^*,\theta_2)c(\theta_2)'{\Sigma^*}^{-1}dY_s\right.\right.\nonumber\\
    &\left.\left.\qquad-\sum_{j=1}^i\exp\left( -\alpha(\theta_1^*,\theta_2)(t_i-t_{j-1})\right)\gamma_+(\theta_1^*,\theta_2)c(\theta_2)'{\Sigma^*}^{-1}\Delta_jY\right.\Biggr|^p \right.\Biggr]^\frac{1}{p}\nonumber\\
    &+E\left[ \sup_{\theta_2 \in \Theta_2}\left|\sum_{j=1}^i\exp\left( -\alpha(\theta_1^*,\theta_2)(t_i-t_{j-1})\right)\right.\right.\nonumber\\
    &\left.\left.\left\{\gamma_+(\theta_1^*,\theta_2)c(\theta_2)'{\Sigma^*}^{-1}
    -\gamma_+(\hat{\theta}_1^n,\theta_2)c(\theta_2)'\Sigma(\hat{\theta}_1^n)^{-1}\right\}\Delta_jY\right.\Biggr|^p \right.\Biggr]^\frac{1}{p}\nonumber\\
    &+E\left[ \sup_{\theta_2 \in \Theta_2}\left|\sum_{j=1}^i\left\{\exp\left( -\alpha(\theta_1^*,\theta_2)(t_i-t_{j-1})\right)\right.\right.\right.\nonumber\\
    \label{eq7}&\qquad \left.\left.\left.-\exp\left( -\alpha(\hat{\theta}_1^n,\theta_2)(t_i-t_{j-1})\right)\right\}\gamma_+(\hat{\theta}_1^n,\theta_2)c(\theta_2)'\Sigma(\hat{\theta}_1^n)^{-1}\Delta_jY\right.\Biggr|^p\right.\Biggr]^\frac{1}{p}.
  \end{align}
  The first term of the right-hand side can be bounded by the mean value theorem and Theorem \ref{main-theorem-1} :
  \begin{align}
    \label{eq10}\begin{split}
      &E\left[ \sup_{\theta_2 \in \Theta_2}\left|\left\{\exp(-\alpha(\theta_1^*,\theta_2)t)-\exp(-\alpha(\hat{\theta}_1^n,\theta_2)t)\right\}m_0\right|^p \right]^\frac{1}{p}\\
    \leq &CE\left[ |\hat{\theta}_1^n-\theta_1^*|^p \right]^\frac{1}{p}\leq Cn^{-\frac{1}{2}}.
    \end{split}
  \end{align}

  Next we evaluate the second term using (\ref{ineq-lemma1}). Noting that by the mean value theorem and Lemma \ref{matrix-exponential}, we have 
    \begin{align*}
      &\left|\exp(-\alpha(\theta_1^*,\theta_2)(t_i-s))-\exp(-\alpha(\theta_1^*,\theta_2)(t_i-t_{j-1}))\right|\\
      =&\left|\alpha(\theta_1^*,\theta_2)\exp(-\alpha(\theta_1^*,\theta_2)(t_i-u))(s-t_{j-1})\right|\\
      \leq &Ce^{-C(t_i-u)}(s-t_{j-1})\\
      \leq &Ce^{-C(t_i-s)}h
    \end{align*}
    and
    \begin{align*}
      &\left|(t_i-s)\exp(-\alpha(\theta_1^*,\theta_2)(t_i-s))-(t_i-t_{j-1})\exp(-\alpha(\theta_1^*,\theta_2)(t_i-t_{j-1}))\right|\\
      \leq &|(t_{j-1}-s)\exp(-\alpha(\theta_1^*,\theta_2)(t_i-s)|\\
      &+(t_i-t_{j-1})\left|\exp(-\alpha(\theta_1^*,\theta_2)(t_i-s))-\exp(-\alpha(\theta_1^*,\theta_2)(t_i-t_{j-1}))\right|\\
      \leq &Ce^{-C(t_i-s)}h+C(t_i-t_{j-1})e^{-C(t_i-s)}h\\
      =&Ce^{-C(t_i-s)}h+C(t_i-s)e^{-C(t_i-s)}h+C(s-t_{j-1})e^{-C(t_i-s)}h\\
      \leq &Ce^{-C(t_i-s)}h+Ce^{-\frac{1}{2}C(t_i-s)}h+Ce^{-C(t_i-s)}h^2\\
      \leq &Ce^{-C(t_i-s)}h,
    \end{align*}
    where $t_{j-1}\leq u\leq s\leq t_j$, it follows from (\ref{ineq-lemma1})
    \begin{align}
      \label{eq11}
        &E\left.\Biggl[\sup_{\theta_2 \in \Theta_2}\left.\Biggl|\int_0^{t_i}\exp\left( -\alpha(\theta_1^*,\theta_2)(t_i-s)\right)\gamma_+(\theta_1^*,\theta_2)c(\theta_2)'{\Sigma^*}^{-1}dY_s\right.\right.\nonumber\\
        &\left.\left.\qquad-\sum_{j=1}^i\exp\left( -\alpha(\theta_1^*,\theta_2)(t_i-t_{j-1})\right)\gamma_+(\theta_1^*,\theta_2)c(\theta_2)'{\Sigma^*}^{-1}\Delta_jY\right|^p \right]^\frac{1}{p}\nonumber\\
        =&E\left[\sup_{\theta_2 \in \Theta_2} \left|\sum_{j=1}^i\int_{t_{j-1}}^{t_j}\left\{ \exp(-\alpha(\theta_1^*,\theta_2)(t_i-s))\right.\right.\right.\nonumber\\
        &\left.\left.\left.-\exp(-\alpha(\theta_1^*,\theta_2)(t_i-t_{j-1}))\right\}\gamma_+(\theta_1^*,\theta_2)c(\theta_2)'{\Sigma^*}^{-1}dY_s \right.\Biggr|^p \right.\Biggr]^\frac{1}{p}\nonumber\\
        \leq &C_p\sup_{\theta_2 \in \Theta_2}\left\{\sum_{j=1}^i\int_{t_{j-1}}^{t_j}|\exp(-\alpha(\theta_1^*,\theta_2)(t_i-s))\right.\nonumber\\
        &\qquad\qquad\qquad-\exp(-\alpha(\theta_1^*,\theta_2)(t_i-t_{j-1}))|ds\nonumber\\
        &+\sum_{j=1}^i\int_{t_{j-1}}^{t_j}|\partial_{\theta_2}\alpha(\theta_1^*,\theta_2)(t_i-s)\exp(-\alpha(\theta_1^*,\theta_2)(t_i-s))\nonumber\\
        &\qquad\qquad-\partial_{\theta_2}\alpha(\theta_1^*,\theta_2)(t_i-t_{j-1})\exp(-\alpha(\theta_1^*,\theta_2)(t_i-t_{j-1}))|ds\nonumber\\
        &+\left(\sum_{j=1}^i\int_{t_{j-1}}^{t_j}|\exp(-\alpha(\theta_1^*,\theta_2)(t_i-s))-\exp(-\alpha(\theta_1^*,\theta_2)(t_i-t_{j-1}))|^2ds\right)^\frac{1}{2}\nonumber\\
        &+\left(\sum_{j=1}^i\int_{t_{j-1}}^{t_j}|\partial_{\theta_2}\alpha(\theta_1^*,\theta_2)(t_i-s)\exp(-\alpha(\theta_1^*,\theta_2)(t_i-s))\right.\nonumber\\
        &\left.\left.\qquad\qquad-\partial_{\theta_2}\alpha(\theta_1^*,\theta_2)(t_i-t_{j-1})\exp(-\alpha(\theta_1^*,\theta_2)(t_i-t_{j-1}))|^2ds\right.\Biggr)^\frac{1}{2}\right.\Biggr\}\nonumber\\
        \leq &C_p\sum_{j=1}^i\int_{t_{j-1}}^{t_j}e^{-C(t_i-s)}dsh+C_p\left( \sum_{j=1}^ie^{-C(t_i-s)}h^2 \right)^\frac{1}{2}\nonumber\\
        \leq &C_p\int_0^{t_i}e^{-C(t_i-s)}dsh+C_p\left( \int_0^{t_i}e^{-C(t_i-s)}dsh^2 \right)^\frac{1}{2}\nonumber\\
        \leq &C_ph.    
    \end{align}

    As for the third term, in the same way as Proposition \ref{m-hat-norm}, we have
    \begin{align*}
      &E\left[ \sup_{\theta_2 \in \Theta_2}\left|\sum_{j=1}^i\exp\left( -\alpha(\theta_1^*,\theta_2)(t_i-t_{j-1})\right)\right.\right.\\
    &\left.\left.\left\{\gamma_+(\theta_1^*,\theta_2)c(\theta_2)'{\Sigma^*}^{-1}
    -\gamma_+(\hat{\theta}_1^n,\theta_2)c(\theta_2)'\Sigma(\hat{\theta}_1^n)^{-1}\right\}\Delta_jY\right.\Biggr|^p \right.\Biggr]^\frac{1}{p}\leq C_pn^{-\frac{1}{2}},
    \end{align*}    
    since it holds
    \begin{align}
      \label{eq12}E\left[\sup_{\theta_2 \in \Theta_2}\left|\gamma_+(\theta_1^*,\theta_2)c(\theta_2)'{\Sigma^*}^{-1}
      -\gamma_+(\hat{\theta}_1^n,\theta_2)c(\theta_2)'\Sigma(\hat{\theta}_1^n)^{-1}\right|^p \right]^\frac{1}{p}\leq C_pn^{-\frac{1}{2}}
    \end{align}
    by the mean value theorem and Theorem \ref{main-theorem-1}.

    Finally, we consider the forth term of (\ref{eq7}).
    Noting that it follows from Lemma \ref{matrix-exponential} and the stability of $-\alpha(\theta_1,\theta_2)$, 
    \begin{align*}
      &\left|\exp\left( -\left[\alpha(\theta_1,\theta_2)+\left\{ \alpha(\theta_1^*,\theta_2)-\alpha(\theta_1,\theta_2) \right\}u\right](t_i-t_{j-1})\right)\right|\\
      =&\left|\exp\left( -\alpha(\theta_1,\theta_2)(1-u)(t_i-t_{j-1})\right)||\exp\left(\alpha(\theta_1^*,\theta_2)u(t_i-t_{j-1})\right)\right|\\
      \leq &Ce^{-C(1-u)(t_{i}-t_{j-1})}e^{-Cu(t_{i}-t_{j-1})}=Ce^{-C(t_i-t_{j-1})},
    \end{align*}   
    we have
    \begin{align*}
      &\sum_{j=1}^i\left.\biggl|(t_i-t_{j-1})\right.\\
      &\left.\int_0^1\exp\left( -\left[\alpha(\theta_1,\theta_2)+\left\{ \alpha(\theta_1^*,\theta_2)-\alpha(\theta_1,\theta_2) \right\}u\right](t_i-t_{j-1})\right)du\right|h\\
      \leq &\sum_{j=1}^i(t_i-t_{j-1})\\
      &\int_0^1\left|\exp\left( -\left[\alpha(\theta_1,\theta_2)+\left\{ \alpha(\theta_1^*,\theta_2)-\alpha(\theta_1,\theta_2) \right\}u\right](t_i-t_{j-1})\right)\right|duh\\
      \leq &C\sum_{j=1}^i(t_i-t_{j-1})e^{-C(t_i-t_{j-1})}h\leq C.
    \end{align*}
    In the same way, we obtain the boundedness of 
    \begin{align*}
      &\sum_{j=1}^i\left.\biggl|(t_i-t_{j-1})\right.\\
      &\left.\times \partial_{(\theta_1,\theta_2)}\int_0^1\exp\left( -\left[\alpha(\theta_1,\theta_2)+\left\{ \alpha(\theta_1^*,\theta_2)-\alpha(\theta_1,\theta_2) \right\}u\right](t_i-t_{j-1})\right)du\right|h,\\
      &\sum_{j=1}^i\left.\biggl|(t_i-t_{j-1})\right.\\
      &\left.\times \int_0^1\exp\left( -\left[\alpha(\theta_1,\theta_2)+\left\{ \alpha(\theta_1^*,\theta_2)-\alpha(\theta_1,\theta_2) \right\}u\right](t_i-t_{j-1})\right)du\right|^2h\\
    \end{align*}
    and
    \begin{align*}
      &\sum_{j=1}^i\left.\biggl|(t_i-t_{j-1})\right.\\
      &\left.\times \partial_{(\theta_1,\theta_2)}\int_0^1\exp\left( -\left[\alpha(\theta_1,\theta_2)+\left\{ \alpha(\theta_1^*,\theta_2)-\alpha(\theta_1,\theta_2) \right\}u\right](t_i-t_{j-1})\right)du\right|^2h.
    \end{align*}
    Thus by (\ref{ineq-lemma2}) we obtain
    \begin{align*}
      &\sum_{j=1}^iE\left[\sup_{\theta_2 \in \Theta_2}\left.\biggl|(t_i-t_{j-1})\right.\right.\\
      &\int_0^1\exp\left( -\left[\alpha(\hat{\theta}_1^n,\theta_2)+\left\{ \alpha(\theta_1^*,\theta_2)-\alpha(\hat{\theta}_1^n,\theta_2) \right\}u\right](t_i-t_{j-1})\right)du\\
      &\left.\left.\gamma_+(\hat{\theta}_1^n,\theta_2)c(\theta_2)'\Sigma(\hat{\theta}_1^n)^{-1}\Delta_jY\right.\biggr|^p \right.\Biggr]\\
      \leq &\sum_{j=1}^iE\left[\sup_{\theta_1\in \Theta_1,\theta_2 \in \Theta_2}\left.\biggl|(t_i-t_{j-1})\right.\right.\\
      &\int_0^1\exp\left( -\left[\alpha(\theta_1,\theta_2)+\left\{ \alpha(\theta_1^*,\theta_2)-\alpha(\theta_1,\theta_2) \right\}u\right](t_i-t_{j-1})\right)du\\
      &\left.\left.\gamma_+(\theta_1,\theta_2)c(\theta_2)'\Sigma(\theta_1)^{-1}\Delta_jY\right.\biggr|^p \right]\\
      \leq &C_p.
    \end{align*}
    Therefore it follows
    \begin{align}
      \label{eq13}\begin{split}
        &E\left[ \sup_{\theta_2 \in \Theta_2}\left|\sum_{j=1}^i\left.\Bigl\{\exp\left( -\alpha(\theta_1^*,\theta_2)(t_i-t_{j-1})\right)\right.\right.\right.\\
    &\qquad \left.\left.\left.-\exp\left( -\alpha(\hat{\theta}_1^n,\theta_2)(t_i-t_{j-1})\right)\right\}\gamma_+(\hat{\theta}_1^n,\theta_2)c(\theta_2)'\Sigma(\hat{\theta}_1^n)^{-1}\Delta_jY\right.\Biggr|^p\right.\Biggr]^\frac{1}{p}\\    
    =&E\left[ \sup_{\theta_2 \in \Theta_2}\left|\sum_{j=1}^i \left\{\alpha(\theta_1^*,\theta_2)-\alpha(\hat{\theta}_1^n,\theta_2)\right\}(t_i-t_{j-1})\right.\right.\\    
    &\int_0^1\exp\left( -\left[\alpha(\hat{\theta}_1^n,\theta_2)+\left\{ \alpha(\theta_1^*,\theta_2)-\alpha(\hat{\theta}_1^n,\theta_2) \right\}u\right](t_i-t_{j-1})\right)du\\
    &\qquad \left.\left.\gamma_+(\hat{\theta}_1^n,\theta_2)c(\theta_2)'\Sigma(\hat{\theta}_1^n)^{-1}\Delta_jY\right.\Biggl|^p\right]^\frac{1}{p}\\
    \leq &C_pE\left[ \sup_{\theta_2 \in \Theta_2}\left|\alpha(\theta_1^*,\theta_2)-\alpha(\hat{\theta}_1^n,\theta_2)\right|^{2p}\right]^\frac{1}{2p}\\
    \leq &C_pn^{-\frac{1}{2}}.
      \end{split}      
    \end{align}

    Now we completed the proof by (\ref{eq7})-(\ref{eq13}). 
    \end{proof}

    Next, we replace $m_0^*$ and $\gamma_s^*$ with $m_0$ and $\gamma_+(\theta^*)$ in (\ref{sol-m^*}), and introduce
    \begin{align}
      \label{def-tilde-m*}\tilde{m}_t^*=\exp(-a^*t)m_0+\int_0^t\exp(-a^*(t-s))\gamma_+(\theta^*){c^*}'{{\sigma^*}'}^{-1}d\overline{W}_s.
    \end{align}
    Furthermore, we consider for every $n,i \in \mathbb{N}$,
    \begin{align}
      \label{def-Y-tilde}&\tilde{Y}_t=Y_0+\int_0^tc^*\tilde{m}_s^*ds+\sigma^*\overline{W}_t,\\
      \label{def-m-tilde}\begin{split}
        &\tilde{m}_i^n(\theta_2)=\exp\left( -\alpha(\hat{\theta}_1^n,\theta_2)t_i\right)m_0\\
       &+\sum_{j=1}^i\exp\left( -\alpha(\hat{\theta}_1^n,\theta_2)(t_i-t_{j-1})\right)\gamma_+(\hat{\theta}_1^n,\theta_2)c(\theta_2)'\Sigma(\hat{\theta}_1^n)^{-1}\Delta_j\tilde{Y}.
       \end{split}
    \end{align}
    and
    \begin{align}
      \label{def-tilde-delta}\tilde{\Delta}_iY=c^*\tilde{m}_{i-1}(\theta_2^*)h+\sigma^*\Delta_i\overline{W}.
    \end{align}
    Then in the same way as Proposition \ref{m-hat-norm}, it holds for any $p>m_1+m_2$
      \begin{align}
        \label{m-tilde-norm}\sup_{i \in \mathbb{N}}E\left[\sup_{\theta_2 \in \Theta_2}|\partial_{\theta_2}^k\tilde{m}_i^n(\theta_2)|^{p}\right]<\infty~(k=0,1,2,3).
      \end{align}

    \begin{proposition}\label{m*-m-diff}
      For any $p>0$ and $t\geq 0$, it holds
        \begin{align*}
          E\left[ |m_t^*-\tilde{m}_t^*|^p \right]^\frac{1}{p}\leq C_pe^{-Ct}.
        \end{align*}
    \end{proposition}
    \begin{proof}
      By (\ref{sol-m^*}), (\ref{def-tilde-m*}), Lemmas \ref{BDG} and \ref{matrix-exponential}, Proposition \ref{gamma-converge} and the stability of $\alpha$, we obtain
      \begin{align*}
        &E\left[ |m_t^*-\tilde{m}_t^*|^p \right]^\frac{1}{p}\\
        \leq &E\left[ |\exp(-a^*t)(m_0^*-m_0)|^p \right]^\frac{1}{p}\\
        &+E\left[\left|\int_0^t\exp(-a^*(t-s))\{\gamma_s^*-\gamma_+(\theta^*)\}{c^*}'(\sigma^*)^{-1}d\overline{W}_s\right|^p\right]^\frac{1}{p}\\
        \leq &C_{p}|\exp(-a^*t)|
        +C_{p}\left(\int_0^t|\exp(-a^*(t-s))|^2|\gamma_s^*-\gamma_+(\theta^*)|^2ds\right)^{\frac{1}{2}}\\
        \leq &C_{p}e^{-Ct}
        +C_{p}\left(\int_0^te^{-C(t-s)}e^{-Cs}ds\right)^{\frac{1}{2}}\\
        \leq &C_p(e^{-Ct}+te^{-Ct})\leq C_pe^{-Ct}.
      \end{align*}
    \end{proof}

    \begin{proposition}\label{exp-Ydiff}
      Let $A:\Theta \to M_{d_1,d_2}(\mathbb{R})$ be a continuous mapping. Then for any $i,n \in \mathbb{N}, p>0$ and $k=0,1,2,\cdots$, it holds
      \begin{align*}
        E\left[ \sup_{\theta_1\in \Theta_1,\theta_2 \in \Theta_2}\left|\sum_{j=1}^i(t_i-t_{j-1})^k\exp\left( -\alpha(\theta_1,\theta_2)(t_i-t_{j-1})\right)\right.\right.\\
        \left.\left.\times A(\theta_1,\theta_2)(\Delta_j\tilde{Y}-\Delta_jY)\right.\Biggr|^p \right.\Biggr]^\frac{1}{p}\leq C_{p,k,A}e^{-Ct_i}.
      \end{align*}
    \end{proposition}
    \begin{proof}
      By Lemma \ref{BDG} and Proposition \ref{m*-m-diff}
      \begin{align*}
        &E\left[ \sup_{\theta_1\in \Theta_1,\theta_2 \in \Theta_2}\left|\sum_{j=1}^i(t_i-t_{j-1})^k\exp\left( -\alpha(\theta_1,\theta_2)(t_i-t_{j-1})\right)\right.\right.\\
        &\left.\left.\times A(\theta_1,\theta_2)(\Delta_j\tilde{Y}-\Delta_jY)\right.\Biggr|^p \right]^\frac{1}{p}\\
        =&E\left[ \sup_{\theta_1\in \Theta_1,\theta_2 \in \Theta_2}\left|\sum_{j=1}^i(t_i-t_{j-1})^k\exp\left( -\alpha(\theta_1,\theta_2)(t_i-t_{j-1})\right)\right.\right.\\
        &\left.\left.\times A(\theta_1,\theta_2)c^*\int_{t_{j-1}}^{t_j}\{\tilde{m}_s^*-m_s^*\}ds\right.\Biggr|^p \right]^\frac{1}{p}\\
        \leq &\sum_{j=1}^i(t_i-t_{j-1})^k\sup_{\theta_1\in \Theta_1,\theta_2 \in \Theta_2}\left|\exp\left( -\alpha(\theta_1,\theta_2)(t_i-t_{j-1})\right)\right|\\
        &\times \left|A(\theta_1,\theta_2)\right||c^*|E\left[\left|\int_{t_{j-1}}^{t_j} \tilde{m}_s^*-m_s^*ds\right|^p \right]^\frac{1}{p}\\
        \leq &C_p\sum_{j=1}^i(t_i-t_{j-1})^ke^{-C(t_i-t_{j-1})}
        \left(h^{p-1}\int_{t_{j-1}}^{t_j}E\left[\left|\tilde{m}_s^*-m_s^*\right|^p \right]ds\right)^\frac{1}{p}\\
        \leq &C_p\sum_{j=1}^i(t_i-t_{j-1})^ke^{-C(t_i-t_{j-1})}e^{-Ct_{j-1}}h\\
        \leq &C_pe^{-Ct_i}\sum_{j=1}^i(t_i-t_{j-1})^k\\
        \leq &C_pe^{-Ct_i}\int_{t_{-1}}^{t_{i-1}}(t_i-s)^kds=C_pe^{-Ct_i}\frac{(t_i+h)^k-h^k}{k}\\
        \leq &C_{p,k}e^{-Ct_i}.
      \end{align*}
    \end{proof}

    By (\ref{def-m-hat}) and (\ref{def-m-tilde}), we obtain the following corollaries.
    \begin{corollary}\label{m-tilde-m-hat-diff}
      For any $i,n \in \mathbb{N}, p>0$ and $k=0,1,2,3,4$, it holds
      \begin{align*}
        E\left[\sup_{\theta_2\in\Theta_2}\left|\partial_{\theta_2}^k\{\tilde{m}_{i}^n(\theta)-\hat{m}_{i}^n(\theta)\}\right|^p \right]^\frac{1}{p}\leq  C_pe^{-Ct_i}.
      \end{align*}
    \end{corollary}

    \begin{corollary}\label{delta-tilde-delta-diff}
      For any $i,n \in \mathbb{N}$ and $p>m_1+m_2$, it holds 
      \begin{align*}
        E\left[ |\Delta_iY-\tilde{\Delta}_iY|^p \right]^\frac{1}{p}\leq C_p(h^{\frac{3}{2}}+n^{-\frac{1}{2}}h+e^{-Ct_i}h).
      \end{align*}
    \end{corollary}
    \begin{proof}
      By (\ref{innovation}), (\ref{def-tilde-delta}), (\ref{m*-conti}), Lemma \ref{BDG}, Propositions \ref{m-m-hat-diff} and \ref{m*-m-diff} and Corollary \ref{m-tilde-m-hat-diff}, we have
      \begin{align*}
        &E\left[ |\Delta_iY-\tilde{\Delta}_iY|^p \right]^\frac{1}{p}\leq E\left[ \left|c^*\int_{t_{i-1}}^{t_i}\{m_s^*-\tilde{m}_{i-1}(\theta_2^*)\}ds\right|^p \right]^\frac{1}{p}\\
        \leq &|c^*|\left\{E\left[ \left|\int_{t_{i-1}}^{t_i}\{m_s^*-m_{t_{i-1}}^*\}ds\right|^p \right]^\frac{1}{p}+E[|\{m_{t_{i-1}}^*-m_{t_{i-1}}(\theta^*)\}h|^p]^\frac{1}{p}\right.\\
        &+E[|\{m_{t_{i-1}}(\theta^*)-\hat{m}_{i-1}^n(\theta^*)\}h|^p]^\frac{1}{p}
        \left.+E[|\{\hat{m}_{i-1}^n(\theta^*)-\tilde{m}_{i-1}^n(\theta^*)\}h|^p]^\frac{1}{p}\right.\Biggr\}\\
        \leq &|c^*|\left\{\left(h^{p-1}\int_{t_{i-1}}^{t_i}E\left[ \left|m_s^*-m_{t_{i-1}}^*\right|^p \right]ds\right)^\frac{1}{p}+E[|m_{t_{i-1}}^*-m_{t_{i-1}}(\theta^*)h|^p]^\frac{1}{p}\right.\\
        &+E[|\{m_{t_{i-1}}(\theta^*)-\hat{m}_{i-1}^n(\theta^*)\}h|^p]^\frac{1}{p}
        \left.+E[|\{\hat{m}_{i-1}^n(\theta^*)-\tilde{m}_{i-1}^n(\theta^*)\}h|^p]^\frac{1}{p}\right.\Biggr\}\\
        \leq &C_p\{(h^{p-1}\times h^{\frac{p}{2}}h)^\frac{1}{p}+e^{-Ct_i}h+(n^{-\frac{1}{2}}+h)h+e^{-Ct_i}h\}\\
        \leq &C_p(h^{\frac{3}{2}}+n^{-\frac{1}{2}}h+e^{-Ct_i}h).
      \end{align*}
    \end{proof}

    \begin{proposition}\label{m-hat-m-tilde-sum-by-Y}
      Let $A:\Theta \to M_{d_1,d_2}(\mathbb{R})$ be a continuous mapping. Then for any $n \in \mathbb{N}$, $p>m_1+m_2$ and $k=0,1,2,3$
      \begin{align*}
        E\left[\sup_{\theta_2 \in \Theta_2}\left|\sum_{i=1}^n\partial_{\theta_2}^k\{\hat{m}_{i-1}^n(\theta_2)-\tilde{m}_{i-1}^n(\theta_2)\}A(\theta_2)\Delta_i Y\right|^{p}\right]<C_p.
      \end{align*}
    \end{proposition}
    \begin{proof}
      By (\ref{def-m-hat}) and (\ref{def-m-tilde}), we have
      \begin{align*}
        &\hat{m}_{i-1}^n(\theta_2)-\tilde{m}_{i-1}^n(\theta_2)\\
        =&\sum_{j=1}^i\exp\left( -\alpha(\hat{\theta}_1^n,\theta_2)(t_i-t_{j-1})\right)\gamma_+(\hat{\theta}_1^n,\theta_2)c(\theta_2)'\Sigma(\hat{\theta}_1^n)^{-1}(\Delta_jY-\Delta_j\tilde{Y}).
      \end{align*}
      Hence for every $k=0,1,2,3$, $\partial_{\theta_2}^k\{\hat{m}_{i-1}^n(\theta_2)-\tilde{m}_{i-1}^n(\theta_2)\}$ is a sum of the form
      \begin{align*}
        \sum_{j=1}^i\partial_{\theta_2}^l\exp\left( -\alpha(\hat{\theta}_1^n,\theta_2)(t_i-t_{j-1})\right)A_1(\theta)(\Delta_jY-\Delta_j\tilde{Y})\\
        (l=0,1,2,3),
      \end{align*}
      where $A_1$ is a $M_{d_1,d_2}(\mathbb{R})$-valued $k$-dimensional tensor of class $C^1$. Thus if we set
      \begin{align*}
        \Phi(\theta)=\sum_{j=1}^i\partial_{\theta_2}^l\exp\left( -\alpha(\theta_1,\theta_2)(t_i-t_{j-1})\right)A_1(\theta)(\Delta_jY-\Delta_j\tilde{Y}),
      \end{align*}
      it is enough to show
      \begin{align}
        \label{eq14}E\left[\sup_{\theta \in \Theta}\left|\sum_{i=1}^n\Phi(\theta)\Delta_i Y\right|^{p}\right]^\frac{1}{p}<C_p.
      \end{align}
      Since by \cite{MatrixExponentialNote}, we have
      \begin{align*}
        E\left[\sup_{\theta \in \Theta}|\Phi(\theta)|^p  \right]^\frac{1}{p}\leq C_pe^{-Ct_i}
      \end{align*}
      and
      \begin{align*}
        E\left[\sup_{\theta \in \Theta}|\partial_{\theta}\Phi(\theta)|^p  \right]^\frac{1}{p}\leq C_pe^{-Ct_i}.
      \end{align*}
      Thus  it holds by (\ref{ineq-lemma2}) and Proposition \ref{exp-Ydiff}
      \begin{align*}
        E\left[\sup_{\theta \in \Theta}\left|\sum_{i=1}^n\Phi(\theta)\Delta_i Y\right|^{p}\right]^\frac{1}{p}
        \leq C_p\sum_{i=1}^ne^{-C{t_i}}\leq C_p.
      \end{align*}
      Hence we obtain (\ref{eq14}).
    \end{proof}

    \begin{proposition}\label{m-hat-dY-norm}
      Let $Z$ be a $M_{d_2}(\mathbb{R})$-valued random variable. Then for any $n \in \mathbb{N}, k=0,1,2,3$ and $p>m_1+m_2$ it holds
      \begin{align*}
        &E\left[ \left|\sup_{\theta_2 \in \Theta_2}\sum_{i=1}^n\partial_{\theta_2}^k\{\hat{m}_{i-1}^n(\theta_2)'c(\theta_2)'\}Z\Delta_jY \right|^p\right]^\frac{1}{p}\\
        \leq &C_p\left( E\left[ |A|^{4p} \right]^\frac{1}{4p}nh+E\left[ |A|^{2p} \right]^\frac{1}{2p}(nh)^\frac{1}{2} \right).
      \end{align*}
    \end{proposition}
    \begin{proof}
      By (\ref{def-m-hat}), $\partial_{\theta_2}^k\{\hat{m}_{i}^n(\theta_2)'c(\theta_2)'\}$ is a sum of the form
      \begin{align*}
        &A_i(\hat{\theta}_1^n,\theta_2)\exp\left( -\alpha(\hat{\theta}_1^n,\theta_2)t_i\right)\\
  &+\sum_{l=0}^k\sum_{j=1}^i\partial_{\theta_2}^l\exp\left( -\alpha(\hat{\theta}_1^n,\theta_2)(t_i-t_{j-1})\right)B_i(\hat{\theta}_1^n,\theta_2)\Delta_jY,
      \end{align*}
      where $A_i$ and $B_i$ are $k$-dimensional tensor valued continuously differentiable mappings on $\Theta$. Thus if we set
      \begin{align*}
        \Psi_i(\theta)=\Psi_i(\theta_1,\theta_2)=&A_i(\theta)\exp\left( -\alpha(\theta)t_i\right)\\
        &+\sum_{l=0}^k\sum_{j=1}^i\partial_{\theta_2}^l\exp\left( -\alpha(\theta_1,\theta_2)(t_i-t_{j-1})\right)B_i(\theta)\Delta_jY,
      \end{align*}
      it is enough to show
      \begin{align}
        \label{eq22}E\left[ \left|\sup_{\theta \in \Theta}\sum_{i=1}^n\Psi_{i-1}(\theta)Z\Delta_jY \right|^p\right]^\frac{1}{p}\leq C_p\left( E\left[ |A|^{4p} \right]^\frac{1}{4p}nh+E\left[ |A|^{2p} \right]^\frac{1}{2p}(nh)^\frac{1}{2} \right).
      \end{align}
      In the same way as Proposition \ref{m-hat-norm}, we first obtain
      \begin{align*}
        E\left[ \left|\Psi_i(\theta)\right|^p \right]\leq C_p
      \end{align*}
      and
      \begin{align*}
        E\left[ \left|\partial_{\theta}\Psi_i(\theta)\right|^p \right]\leq C_p.
      \end{align*}
      Therefore noting that $\Psi_i(\theta)$ is $\mathcal{F}_{t_{i-1}}$-measurable, we obtain (\ref{eq22}) by (\ref{ineq-lemma2}).
    \end{proof}

    Next, we define $\tilde{\mathbb{H}}_n^2$, $\tilde{\Delta}_n^2$, $\tilde{\Gamma}_n^2$ and $\tilde{\mathbb{Y}}_n^2$ by
    \begin{align}
      \label{def-H2-tilde}\begin{split}
        &\tilde{\mathbb{H}}_n^2(\theta_2)=\frac{1}{2}\sum_{i=1}^n\left\{-h{\Sigma^*}^{-1}[(c(\theta_2)\tilde{m}_{i-1}^n(\theta_2))^{\otimes 2}]\right.\\
      &\left.+\tilde{m}_{i-1}^n(\theta_2)'c(\theta_2)'{\Sigma^*}^{-1}\tilde{\Delta}_jY+\tilde{\Delta}_jY'{\Sigma^*}^{-1}c(\theta_2)\tilde{m}_{i-1}^n(\theta_2)\right\}
      \end{split}\\
      &\label{def-bold-Y-tilde}\tilde{\mathbb{Y}}_n^2(\theta_2)=\frac{1}{t_n}\{\tilde{\mathbb{H}}_n^2(\theta_2)-\tilde{\mathbb{H}}_n^2(\theta_2^*)\}\\
      &\label{def-tilde-delta-n}\tilde{\Delta}_n^2=\frac{1}{\sqrt{t_n}}\partial_{\theta}\tilde{\mathbb{H}}_n^2(\theta_2^*),
    \end{align}
    and
    \begin{align}\label{def-tilde-gamma}\tilde{\Gamma}_n^2=-\frac{1}{t_n}\partial_{\theta}^2\tilde{\mathbb{H}}_n^2(\theta_2^*),
    \end{align}
    respectively.

    \begin{proposition}\label{bold-H-diff}
      For any $n \in \mathbb{N}$, $p>m_1+m_2$ and $k=0,1,2,3$, it holds
      \begin{align*}
        E\left[ \sup_{\theta_2 \in \Theta_2}\left|\partial_{\theta_2}^k\{\mathbb{H}_n(\theta_2)-\tilde{\mathbb{H}}_n(\theta_2)\}\right|^p \right]^\frac{1}{p}\leq C_p(nh^{\frac{3}{2}}+n^\frac{1}{2}h+1).
      \end{align*}
    \end{proposition}
    \begin{proof}
      We only consider the case of $k=0$. The rest is the same. By (\ref{def-H2}) and (\ref{def-H2-tilde}), 
      \begin{align}
        &E\left[ \sup_{\theta_2 \in \Theta_2}\left|\mathbb{H}_n(\theta_2)-\tilde{\mathbb{H}}_n(\theta_2)\right|^p \right]^\frac{1}{p}\nonumber\\
        \leq &E\left[\sup_{\theta_2 \in \Theta_2}\left|\frac{1}{2}h\sum_{i=1}^n\{\Sigma(\hat{\theta}_1^n)^{-1}-{\Sigma^*}^{-1}\}[(c(\theta_2)\hat{m}_{j-1}^n(\theta_2))^{\otimes 2}]\right|^p \right]^\frac{1}{p}\nonumber\\
        &+E\left[\sup_{\theta_2 \in \Theta_2}\left|\frac{1}{2}\sum_{i=1}^n\hat{m}_{j-1}^n(\theta_2)'c(\theta_2)'\{\Sigma(\hat{\theta}_1^n)^{-1}-{\Sigma^*}^{-1}\}\Delta_jY\right|^p \right]^\frac{1}{p}\nonumber\\       
        &+E\left[\sup_{\theta_2 \in \Theta_2}\left|\frac{1}{2}\sum_{i=1}^n\Delta_jY'\{\Sigma(\hat{\theta}_1^n)^{-1}-{\Sigma^*}^{-1}\}c(\theta_2)\hat{m}_{j-1}^n(\theta_2)\right|^p \right]^\frac{1}{p}\nonumber\\       
        &+E\left[\sup_{\theta_2 \in \Theta_2}\left|\frac{1}{2}h\sum_{i=1}^n\left\{{\Sigma^*}^{-1}[(c(\theta_2)\hat{m}_{j-1}^n(\theta_2))^{\otimes 2}]\right.\right.\right.\nonumber\\
        &\left.\left.\left.\qquad\qquad\qquad-{\Sigma^*}^{-1}[(c(\theta_2)\tilde{m}_{j-1}^n(\theta_2))^{\otimes 2}]\right\}\right.\biggr|^p \right]^\frac{1}{p}\nonumber\\
        &+E\left[\sup_{\theta_2 \in \Theta_2}\left|\frac{1}{2}\sum_{i=1}^n\{\hat{m}_{j-1}^n(\theta_2)'c(\theta_2)'{\Sigma^*}^{-1}\Delta_jY\right.\right.\nonumber\\
        &\left.\left.\qquad\qquad\qquad-c(\theta_2)\tilde{m}_{j-1}^n(\theta_2){\Sigma^*}^{-1}\tilde{\Delta}_jY\}\right.\biggr|^p \right]^\frac{1}{p}\nonumber\\
        &+E\left[\sup_{\theta_2 \in \Theta_2}\left|\frac{1}{2}\sum_{i=1}^n\{\Delta_jY'{\Sigma^*}^{-1}c(\theta_2)\hat{m}_{j-1}^n(\theta_2)\right.\right.\nonumber\\
        \label{eq16}&\left.\left.\qquad\qquad\qquad-\tilde{\Delta}_jY'{\Sigma^*}^{-1}c(\theta_2)\tilde{m}_{j-1}^n(\theta_2)\}\right.\biggr|^p \right]^\frac{1}{p}.       
      \end{align}

      For the first three terms of the right-hand side, we have by Theorem \ref{main-theorem-1} and Proposition \ref{m-hat-norm}
      \begin{align*}
        &E\left[\sup_{\theta_2 \in \Theta_2}\left|\frac{1}{2}h\sum_{i=1}^n\{\Sigma(\hat{\theta}_1^n)^{-1}-{\Sigma^*}^{-1}\}[(c(\theta_2)\hat{m}_{j-1}^n(\theta_2))^{\otimes 2}]\right|^p \right]^\frac{1}{p}\\
        \leq &\sum_{i=1}^nE\left[\sup_{\theta_2 \in \Theta_2}\left|\frac{1}{2}h\{\Sigma(\hat{\theta}_1^n)^{-1}-{\Sigma^*}^{-1}\}[(c(\theta_2)\hat{m}_{j-1}^n(\theta_2))^{\otimes 2}]\right|^p \right]^\frac{1}{p}\\
        \leq &\frac{1}{2}h\sum_{i=1}^nE\left[\left|\Sigma(\hat{\theta}_1^n)^{-1}-{\Sigma^*}^{-1}]\right|^{2p} \right]^\frac{1}{2p}E\left[\sup_{\theta_2 \in \Theta_2}\left|c(\theta_2)\hat{m}_{j-1}^n(\theta_2)\right|^{4p} \right]^\frac{1}{2p}\\
        \leq &C_pn^\frac{1}{2}h,
      \end{align*}
      and by Proposition \ref{m-hat-dY-norm}
      \begin{align*}
        &E\left[\sup_{\theta_2 \in \Theta_2}\left|\frac{1}{2}\sum_{i=1}^n\hat{m}_{j-1}^n(\theta_2)'c(\theta_2)'\{\Sigma(\hat{\theta}_1^n)^{-1}-{\Sigma^*}^{-1}\}\Delta_jY\right|^p \right]^\frac{1}{p}\\
        &\leq C_pn^{-\frac{1}{2}}\{nh+(nh)^\frac{1}{2}\}\leq C_p(n^{-\frac{1}{2}}h+h^\frac{1}{2}).
      \end{align*}
      In the same way, the third term can be bounded by $C_p(n^{-\frac{1}{2}}h+h^\frac{1}{2})$.

      Furthermore, making use of Proposition \ref{m-hat-norm}, (\ref{m-tilde-norm}) and Corollary \ref{m-tilde-m-hat-diff}, we can bound the fourth term by $\displaystyle C_p\sum_{i=1}^nhe^{-Ct_i}\leq C_ph$, noting that
      \begin{align*}
        &\Sigma(\hat{\theta}_1^n)^{-1}[(c(\theta_2)\hat{m}_{j-1}^n(\theta_2))^{\otimes 2}]-\Sigma(\hat{\theta}_1^n)^{-1}[(c(\theta_2)\tilde{m}_{j-1}^n(\theta_2))^{\otimes 2}]\\
        =&\{\hat{m}_{j-1}^n(\theta_2)+\tilde{m}_{j-1}^n(\theta_2)\}'c(\theta_2)'\Sigma(\hat{\theta}_1^n)^{-1}c(\theta_2)\{\hat{m}_{j-1}^n(\theta_2)-\tilde{m}_{j-1}^n(\theta_2)\}\\
        &+\{\hat{m}_{j-1}^n(\theta_2)-\tilde{m}_{j-1}^n(\theta_2)\}'c(\theta_2)'\Sigma(\hat{\theta}_1^n)^{-1}\tilde{m}_{j-1}^n(\theta_2)\\
        &+\tilde{m}_{j-1}^n(\theta_2)'c(\theta_2)'\Sigma(\hat{\theta}_1^n)^{-1}\{\hat{m}_{j-1}^n(\theta_2)-\tilde{m}_{j-1}^n(\theta_2)\}.
      \end{align*}

     Finally, the last two terms can be bounded by $\displaystyle C_p+C_p\sum_{i=1}^n(h^{\frac{3}{2}}+n^{-\frac{1}{2}}h+e^{-Ct_i}h)\leq C_p(1+nh^{\frac{3}{2}}+n^\frac{1}{2}h+h)$ due to the Corollary \ref{delta-tilde-delta-diff}, Proposition \ref{m-hat-m-tilde-sum-by-Y} and the identity
     \begin{align*}
      &\hat{m}_{j-1}^n(\theta_2)'c(\theta_2)'\Sigma(\hat{\theta}_1^n)^{-1}\Delta_jY
      -\tilde{m}_{j-1}^n(\theta_2)'c(\theta_2)'\Sigma(\hat{\theta}_1^n)^{-1}\tilde{\Delta}_jY\\
      =&\{\hat{m}_{j-1}^n(\theta_2)-\tilde{m}_{j-1}^n(\theta_2)\}'c(\theta_2)'\Sigma(\hat{\theta}_1^n)^{-1}\Delta_jY\\
      &+\tilde{m}_{j-1}^n(\theta_2)'c(\theta_2)'\Sigma(\hat{\theta}_1^n)^{-1}\{\Delta_jY-\tilde{\Delta}_jY\}.
     \end{align*}

     Putting it all together, we obtain
     \begin{align*}
      E\left[ \sup_{\theta_2 \in \Theta_2}\left|\mathbb{H}_n(\theta_2)-\tilde{\mathbb{H}}_n(\theta_2)\right|^p \right]^\frac{1}{p}
      \leq &C_p(1+nh^{\frac{3}{2}}+n^\frac{1}{2}h+h^\frac{1}{2}+h)\\
      \leq &C_p(1+nh^{\frac{3}{2}}+n^\frac{1}{2}h).
     \end{align*}
    \end{proof}

    \begin{proposition}\label{tilde-Dalta-norm}
      For any $p\geq 2$, it holds 
        \begin{align*}
          \sup_{n \in \mathbb{N}}E\left[ |\tilde{\Delta}_n|^p \right]<\infty.
        \end{align*}
    \end{proposition}
    \begin{proof}
      If we set $\tilde{M}_j^n(\theta_2)=c(\theta_2)\tilde{m}_j^n(\theta)$, we have
      \begin{align}
        \label{tilde-Delta-formula}\begin{split}
          \tilde{\Delta}_n^2=&\frac{1}{2\sqrt{t_n}}\sum_{i=1}^n\left\{ \partial_{\theta_2}\tilde{M}_{i-1}^n(\theta_2)'{{\sigma^*}'}^{-1}\Delta_i\overline{W}+\Delta_i \overline{W}'{\sigma^*}^{-1}\partial_{\theta_2}\tilde{M}_{i-1}^n(\theta_2) \right\}\\
        =&\frac{1}{\sqrt{t_n}}\sum_{i=1}^n\left\{ \partial_{\theta_2}\tilde{M}_i^n(\theta_2)'{{\sigma^*}'}^{-1}\Delta_i \overline{W} \right\}.
        \end{split}        
      \end{align}
      by (\ref{def-tilde-delta}), (\ref{def-tilde-delta-n}) and (\ref{def-captal-M}). Thus by Lemma \ref{BDG} and (\ref{m-tilde-norm}),
      \begin{align*}
        E\left[ |\tilde{\Delta}_n|^p \right]^\frac{2}{p}
        &\leq \frac{1}{{t_n}^\frac{p}{2}}E\left[ \left(\sum_{i=1}^n|\partial_{\theta_2}\tilde{M}_i^n(\theta_2)'{{\sigma^*}'}^{-1}|^2h\right)^\frac{p}{2} \right]^\frac{2}{p}\\
        &\leq \frac{1}{{t_n}^\frac{p}{2}}C_p\sum_{i=1}^nE\left[ |\partial_{\theta_2}\tilde{M}_i^n(\theta_2)'{{\sigma^*}'}^{-1}|^2 \right]^\frac{2}{p}h\\
        &\leq \frac{1}{{t_n}^\frac{p}{2}}\times C_pnh=C_p.
      \end{align*}
    \end{proof}

    Next, we define the process $\{\mu_t\}$ by replacing $Y$ with $\tilde{Y}$ (therefore $m_t^*$ with $m_t(\theta^*)$ and $\gamma_t^*$ with $\gamma_+(\theta^*)$) in (\ref{def-m});
    \begin{align}
      \label{def-mu}\begin{split}
        &\mu_t(\theta_2)=\exp\left( -\alpha(\theta_2)t\right)m_0\\
      &+\int_0^t\exp\left( -\alpha(\theta_2)(t-s)\right)\gamma_+(\theta_2)c(\theta_2)'{\Sigma^*}^{-1}d\tilde{Y}_s.
      \end{split}  
    \end{align}
    Then as $m_t$ is the solution of (\ref{eq-m}), so $\mu_t$ is the solution of
    \begin{align}
      \label{eq-mu}\begin{cases}
        d\mu_t(\theta_2)=-\alpha(\theta_2)\mu_tdt+\gamma_{+}(\theta_2)c(\theta_2)'{\Sigma^*}^{-1}d\tilde{Y}_t\\
        \mu_0(\theta_2)=m_0.
      \end{cases}
    \end{align}
    Moreover, it holds $\mu_t(\theta_2^*)=\tilde{m}_t^*$ since by (\ref{def-m-tilde}) $\tilde{m}_t^*$ is the solution of
    \begin{align*}
      d\tilde{m}_t^*=-a^*\tilde{m}_t^*+\gamma_+(\theta^*){c^*}'{{\sigma^*}'}^{-1}d\overline{W}_t,
    \end{align*}
   which is equivalent to
    \begin{align*}
      d\tilde{m}_t^*=-\alpha(\theta_2^*)\tilde{m}_t^*dt+\gamma_{+}(\theta^*){c^*}'{\Sigma^*}^{-1}d\tilde{Y}_t.
    \end{align*}

    Moreover, just as Proposition \ref{m-m-hat-diff}, the following proposition holds:
    \begin{proposition}\label{mu-m-tilde-diff}
      For any $n,i \in \mathbb{N}$ and $p>m_1+m_2$, we have
      \begin{align*}
        E\left[ \sup_{\theta_2 \in \Theta_2}|\mu_{t_i}(\theta_2)-\tilde{m}_i^n(\theta_2)|^p \right]^\frac{1}{p}\leq C_p(n^{-\frac{1}{2}}+h).
      \end{align*}
    \end{proposition}
    Together with (\ref{m-tilde-norm}), we obtain the following corollary.
    \begin{corollary}\label{mu-norm}
      For any $i \in \mathbb{N}$ and $p>m_1+m_2$, we have
      \begin{align*}
        E\left[ \sup_{\theta_2 \in \Theta_2}|\mu_{t_i}(\theta_2)|^p \right]^\frac{1}{p}\leq C_p.
      \end{align*}
    \end{corollary}

    \begin{proposition}\label{c-mu-diff}
      \begin{align*}
        E[{\Sigma^*}^{-1}[\{c(\theta_2)\mu_t(\theta_2)-c(\theta_2^*)\mu_t(\theta_2^*)\}^{\otimes 2}]],=-2\mathbb{Y}(\theta_2)+O(e^{-t_{i}})
      \end{align*}
      where $O(e^{-t})$ is some continuous function $r:\Theta \to \mathbb{R}$ such that
      \begin{align*}
        |r(\theta)|\leq Ce^{-Ct}.
      \end{align*}
    \end{proposition}
    \begin{proof}
      By (\ref{def-mu}) and (\ref{def-m-tilde}), we have 
      \begin{align}
        \mu_t(\theta)=&\exp(-\alpha(\theta_2)t)m_0\nonumber\\
        &+\int_0^t\exp(-\alpha(\theta_2)(t-s))\gamma_{+}(\theta_2)c(\theta_2)'{\Sigma^*}^{-1}c^*\tilde{m}_sds\nonumber\\
        &+\int_0^t\exp(-\alpha(\theta_2)(t-s))\gamma_{+}(\theta_2)c(\theta_2)'{{\sigma^*}'}^{-1}d\overline{W}_s\nonumber\\
        &=\exp(-\alpha(\theta_2)t)m_0\nonumber\\
        &+\int_0^t\exp(-\alpha(\theta_2)(t-s))\gamma_{+}(\theta_2)c(\theta_2)'{\Sigma^*}^{-1}c^*\nonumber\\
        &\times \left\{\exp(-a^*s)m_0+\int_0^s\exp(-a^*(s-u))\gamma_+(\theta^*){c^*}'{{\sigma^*}'}^{-1}d\overline{W}_u\right\}ds\nonumber\\
        &+\int_0^t\exp(-\alpha(\theta_2)(t-s))\gamma_{+}(\theta_2)c(\theta_2)'{{\sigma^*}'}^{-1}d\overline{W}_s\nonumber\\
        =&\exp(-\alpha(\theta_2)t)m_0\nonumber\\
        &+\int_0^t\exp(-\alpha(\theta_2)(t-s))\gamma_{+}(\theta_2)c(\theta_2)'{\Sigma^*}^{-1}c^*\exp(-a^*s)m_0ds\nonumber\\
        &+\int_0^t\int_0^s\exp(-\alpha(\theta_2)(t-s))\gamma_{+}(\theta_2)c(\theta_2)'{\Sigma^*}^{-1}c^*\nonumber\\
        &\qquad\qquad\qquad\times \exp(-a^*(s-u))\gamma_+(\theta^*){c^*}'{{\sigma^*}'}^{-1}d\overline{W}_uds\nonumber\\
        &+\int_0^t\exp(-\alpha(\theta_2)(t-s))\gamma_{+}(\theta_2)c(\theta_2)'{{\sigma^*}'}^{-1}d\overline{W}_s\nonumber\\
        \label{mu-formula}\begin{split}
          =&\exp(-\alpha(\theta_2)t)m_0\\
        &+\int_0^t\exp(-\alpha(\theta_2)(t-s))\gamma_{+}(\theta_2)c(\theta_2)'{\Sigma^*}^{-1}c^*\exp(-a^*s)m_0ds\\
        &+\int_0^t\left\{\int_s^t\exp(-\alpha(\theta_2)(t-u))\gamma_{+}(\theta_2)c(\theta_2)'{\Sigma^*}^{-1}c^*\right.\\
        &\qquad\qquad\qquad\exp(-a^*(u-s))\gamma_+(\theta^*){c^*}'{{\sigma^*}'}^{-1}du\\
        &\left.+\exp(-\alpha(\theta_2)(t-s))\gamma_{+}(\theta_2)c(\theta_2)'{{\sigma^*}'}^{-1}\right\}d\overline{W}_s.
        \end{split}        
      \end{align}
      Therefore
      \begin{align*}
        &E[{\Sigma^*}^{-1}[\{c(\theta_2)\mu_t(\theta_2)-c(\theta_2^*)\mu_t(\theta_2^*)\}^{\otimes 2}]]\\
        =&E[{\Sigma^*}^{-1}[\{c(\theta_2)\mu_t(\theta_2)-c(\theta_2^*)\tilde{m}_t^*\}^{\otimes 2}]]\\
        =&E\left[{\Sigma^*}^{-1}\left[ \left\{\int_0^t\left\{\int_s^tc(\theta_2)\exp(-\alpha(\theta_2)(t-u))\gamma_{+}(\theta_2)c(\theta_2)'{\Sigma^*}^{-1}c^*\right.\right.\right.\right.\\
        &\qquad\qquad\qquad\times \exp(-a^*(u-s))\gamma_+(\theta^*){c^*}'du\\
        &+c(\theta_2)\exp(-\alpha(\theta_2)(t-s))\gamma_{+}(\theta_2)c(\theta_2)'\\
        &\left.\left.\left.\left.-c^*\exp(-a^*(t-s))\gamma_+(\theta^*){c^*}'\right\}{{\sigma^*}'}^{-1}d\overline{W}_s\right\}^{\otimes 2} \right]  \right]+O(e^{-Ct})\\
        =&\mathrm{Tr}\int_0^t{\Sigma^*}^{-1}\left[ \left\{ \int_s^tc(\theta_2)\exp(-\alpha(\theta_2)(t-u))\gamma_{+}(\theta_2)c(\theta_2)'{\Sigma^*}^{-1}c^*\right.\right.\\
        &\qquad\qquad\qquad\times \exp(-a^*(u-s))\gamma_+(\theta^*){c^*}'du\\
        &+c(\theta_2)\exp(-\alpha(\theta_2)(t-s))\gamma_{+}(\theta_2)c(\theta_2)'\\
        &\left.\left.-c^*\exp(-a^*(t-s))\gamma_+(\theta^*){c^*}' \right\}^{\otimes 2} \right][({{\sigma^*}'}^{-1})^{\otimes 2}]ds+O(e^{-Ct})\\
        =&\mathrm{Tr}\int_0^t{\Sigma^*}^{-1}\left.\Biggl[ \left\{ \int_0^{s}c(\theta_2)\exp(-\alpha(\theta_2)u)\gamma_{+}(\theta_2)c(\theta_2)'{\Sigma^*}^{-1}c^*\right.\right.\\
        &\qquad\qquad\qquad\times \exp(-a^*(s-u))\gamma_+(\theta^*){c^*}'du\\
        &+c(\theta_2)\exp(-\alpha(\theta_2)s)\gamma_{+}(\theta_2)c(\theta_2)'\\
        &\left.\left.-c^*\exp(-a^*s)\gamma_+(\theta^*){c^*}' \right.\biggr\}^{\otimes 2} \right][({{\sigma^*}'}^{-1})^{\otimes 2}]ds+O(e^{-Ct}).
      \end{align*}
      
      Now we have
      \begin{align*}
        &\int_0^s|c(\theta_2)\exp(-\alpha(\theta_2)u)\gamma_{+}(\theta_2)c(\theta_2)'{\Sigma^*}^{-1}c^*\exp(-a^*(s-u))\gamma_+(\theta^*){c^*}'|du\\
        \leq &\int_0^sC_pe^{-Cu}e^{-C(s-u)}du\leq C_pse^{-Cs}\leq C_pe^{-Cs}
      \end{align*}
      and thus by (\ref{def-Y2})
      \begin{align*}
        &|E[{\Sigma^*}^{-1}[\{c(\theta_2)\mu_t(\theta_2)-c(\theta_2^*)\mu_t(\theta_2^*)\}^{\otimes 2}]]+2 \mathbb{Y}^2(\theta_2)|\\
        &=\left|\int_{t}^\infty{\Sigma^*}^{-1}\left[ \left\{ \int_0^{s}c(\theta_2)\exp(-\alpha(\theta_2)u)\gamma_{+}(\theta_2)c(\theta_2)'{\Sigma^*}^{-1}c^*\right.\right.\right.\\
        &\qquad\qquad\qquad\times \exp(-a^*(s-u))\gamma_+(\theta^*){c^*}'du\\
        &+c(\theta_2)\exp(-\alpha(\theta_2)s)\gamma_{+}(\theta_2)c(\theta_2)'\\
        &\left.\left.\left.-c^*\exp(-a^*s)\gamma_+(\theta^*){c^*}' \right\}^{\otimes 2} \right][({{\sigma^*}'}^{-1})^{\otimes 2}]ds\right|\\
        \leq &Ce^{-Ct}.
      \end{align*}
    \end{proof}

    \begin{proposition}\label{Y-tilde-Y-diff}
      For any $n \in \mathbb{N}$ and $p>m_1+m_2$, it holds
      \begin{align*}
        E\left[ \sup_{\theta_2 \in \Theta_2}|\tilde{\mathbb{Y}}_n^2(\theta_2)-\mathbb{Y}^2(\theta_2)|^p \right]^\frac{1}{p}\leq C_p\left( h+n^{-\frac{1}{2}}+{t_n}^{-\frac{1}{2}} \right).
      \end{align*}
    \end{proposition}
    \begin{proof}
      By (\ref{def-Y-tilde}) and (\ref{def-bold-Y-tilde})
      \begin{align*}
        \tilde{\mathbb{Y}}_n^2(\theta_2)=&\frac{1}{2t_n}\sum_{i=1}^n\left\{-h{\Sigma^*}^{-1}[(c(\theta_2)\tilde{m}_{i-1}^n(\theta_2))^{\otimes 2}]+h{\Sigma^*}^{-1}[(c^*\tilde{m}_{i-1}^n(\theta_2^*))^{\otimes 2}]\right.\\
        &+\{\tilde{m}_{i-1}^n(\theta_2)'c(\theta_2)'-\tilde{m}_{i-1}^n(\theta_2^*)'{c^*}'\}{\Sigma^*}^{-1}(c^*\tilde{m}_{i-1}(\theta_2^*)h+\sigma^*\Delta_j\overline{W})\\
        &+\left.(\tilde{m}_{i-1}(\theta_2^*){c^*}'h+\Delta_j\overline{W}'{\sigma^*}'){\Sigma^*}^{-1}\{c(\theta_2)\tilde{m}_{i-1}^n(\theta_2)-c^*\tilde{m}_{i-1}^n(\theta_2^*)\}\right\}\\
        =&\frac{1}{2t_n}\sum_{i=1}^n\left\{ -h{\Sigma^*}^{-1}[(c(\theta_2)\tilde{m}_{i-1}^n(\theta_2)-c^*\tilde{m}_{i-1}^n(\theta_2^*))^{\otimes 2}]\right.\\
        &+\{\tilde{m}_{i-1}^n(\theta_2)'c(\theta_2)'-\tilde{m}_{i-1}^n(\theta_2^*)'{c^*}'\}{\Sigma^*}^{-1}\sigma^*\Delta_j\overline{W})\\
        &\left.+\Delta_j\overline{W}'{\sigma^*}'{\Sigma^*}^{-1}\{c(\theta_2)\tilde{m}_{i-1}^n(\theta_2)-c^*\tilde{m}_{i-1}^n(\theta_2^*)\}\right\}.
      \end{align*}
      Thus we have
      \begin{align}
        \label{eq15}\begin{split}
          &E\left[ \sup_{\theta_2 \in \Theta_2}|\tilde{\mathbb{Y}}_n^2(\theta_2)-\mathbb{Y}^2(\theta_2)|^p \right]^\frac{1}{p}\\
          \leq &\frac{h}{2t_n}E\left[ \sup_{\theta_2 \in \Theta_2}\left|\sum_{i=1}^n{\Sigma^*}^{-1}[(c(\theta_2)\tilde{m}_{i-1}^n(\theta_2)-c^*\tilde{m}_{i-1}^n(\theta_2^*))^{\otimes 2}]\right.\right.\\
        &\left.\left.-{\Sigma^*}^{-1}[(c(\theta_2)\mu_{t_{i-1}}(\theta_2)-c^*\mu_{t_{i-1}}(\theta_2^*))^{\otimes 2}]\right.\Biggr|^p \right]^\frac{1}{p}\\
        &+\frac{h}{2t_n}E\left[ \sup_{\theta_2 \in \Theta_2}\left|\sum_{i=1}^n\left\{{\Sigma^*}^{-1}[(c(\theta_2)\mu_{t_{i-1}}(\theta_2)-c^*\mu_{t_{i-1}}(\theta_2^*))^{\otimes 2}]\right.\right.\right.\\
        &\left.\left.\left.\qquad\qquad\qquad\qquad\qquad\qquad\qquad\qquad\qquad\qquad\qquad+2\mathbb{Y}^2(\theta)\right\}\right.\Biggr|^p \right]^\frac{1}{p}\\
        &+\frac{1}{2t_n}E\left[\sup_{\theta_2 \in \Theta_2}\left|\sum_{i=1}^n\{\tilde{m}_{i-1}^n(\theta_2)'c(\theta_2)'-\tilde{m}_{i-1}^n(\theta_2^*)'{c^*}'\}{\Sigma^*}^{-1}\sigma^*\Delta_j\overline{W}\right|^p \right]^\frac{1}{p}\\
        &+\frac{1}{2t_n}E\left[\sup_{\theta_2 \in \Theta_2}\left|\sum_{i=1}^n\Delta_j\overline{W}'{\sigma^*}'{\Sigma^*}^{-1}\{c(\theta_2)\tilde{m}_{i-1}^n(\theta_2)-c^*\tilde{m}_{i-1}^n(\theta_2^*)\right|^p \right]^\frac{1}{p}.
        \end{split}         
      \end{align}
      For the first term of this, making use of Proposition \ref{mu-m-tilde-diff}, Corollary \ref{mu-norm} and (\ref{m-tilde-norm}), we obtain
      \begin{align}
        \label{eq20}\begin{split}
          &\frac{h}{2t_n}E\left[ \sup_{\theta_2 \in \Theta_2}\left|\sum_{i=1}^n{\Sigma^*}^{-1}[(c(\theta_2)\tilde{m}_{i-1}^n(\theta_2)-c^*\tilde{m}_{i-1}^n(\theta_2^*))^{\otimes 2}]\right.\right.\\
        &\left.\left.-{\Sigma^*}^{-1}[(c(\theta_2)\mu_{t_{i-1}}(\theta_2)-c^*\mu_{t_{i-1}}(\theta_2^*))^{\otimes 2}]\right.\Biggr|^p \right]^\frac{1}{p}\\
        \leq &C_p\frac{h}{2t_n}\times (n^{-\frac{1}{2}}+h)\times n\leq C_p(n^{-\frac{1}{2}}+h),
        \end{split}        
      \end{align}
      just as we evaluate the fourth term of (\ref{eq16}). 

      Now we consider the second term. Due to the proof of Proposition \ref{c-mu-diff}, $c(\theta_2)\mu_{t_i}^n(\theta_2)-c^*\mu_{t_i}^n(\theta_2^*)$ has the form
      \begin{align*}
        c(\theta_2)\mu_{t_i}^n(\theta_2)-c^*\mu_{t}^n(\theta_2^*)
        =p_i(\theta_2)+\int_0^{t_i}q_i(s;\theta_2)d\overline{W}_s
      \end{align*}
      where
      \begin{align*}
        p_i(\theta_2)=&\exp(-\alpha(\theta_2)t_i)m_0-\exp(-\alpha(\theta_2^*)t_i)m_0\\
        &+\int_0^{t_i}\{\exp(-\alpha(\theta_2)(t_i-s))\gamma_{+}(\theta_2)c(\theta_2)'\\
        &-\exp(-\alpha(\theta_2^*)(t_i-s))\gamma_{+}(\theta_2^*){c^*}'\}{\Sigma^*}^{-1}c^*\exp(-a^*s)m_0ds,\\
        q_i(s;\theta_2)=&\int_s^{t_i}c(\theta_2)\exp(-\alpha(\theta_2)(t_i-u))\gamma_{+}(\theta_2)c(\theta_2)'\\
        &\times {\Sigma^*}^{-1}c^*\exp(-a^*(u-s))\gamma_+(\theta^*){c^*}'du\\
        &+c(\theta_2)\exp(-\alpha(\theta_2)(t_i-s))\gamma_{+}(\theta_2)c(\theta_2)'\\
        &-c^*\exp(-a^*(t_i-s))\gamma_+(\theta^*){c^*}'.
      \end{align*}
      Then if we set $\displaystyle \nu^i_t(\theta_2)=p_i(\theta_2)+\int_0^{t}q_i(s;\theta_2)d\overline{W}_s$, It\^{o}'s formula gives
      \begin{align*}
        &{\Sigma^*}^{-1}\left[\left\{ c(\theta_2)\mu_{t_i}^n(\theta_2)-c^*\mu_{t}^n(\theta_2^*) \right\}^{\otimes 2}\right]\\
        =&{\Sigma^*}^{-1}[( \nu^i_{t_i}(\theta_2) )^{\otimes 2}]
        =\int_0^{t_i}{\Sigma^*}^{-1}[( \nu^i_{t_i}(\theta_2) )^{\otimes 2}]\\
        =&{\Sigma^*}^{-1}[p_i(\theta_2)^{\otimes 2}]+2\int_{0}^{t_i}{\nu^i_{s}(\theta_2)}'{\Sigma^*}^{-1}q_i(s;\theta_2)d\overline{W}_s\\
        &+\mathrm{Tr}\int_0^{t_i}{\Sigma^*}^{-1}[q_i(s;\theta_2)^{\otimes 2}]ds\\
        =&E\left[ {\Sigma^*}^{-1}\left[\left\{ c(\theta_2)\mu_{t_i}^n(\theta_2)-c^*\mu_{t}^n(\theta_2^*) \right\}^{\otimes 2}\right] \right]\\
        &+2\int_{0}^{t_i}{\nu^i_{s}(\theta_2)(\theta_2)}'{\Sigma^*}^{-1}q_i(s;\theta_2)d\overline{W}_s\\
        =&-2\mathbb{Y}^2(\theta_2)+2\int_{0}^{t_i}{\nu^i_{s}(\theta_2)}'{\Sigma^*}^{-1}q_i(s;\theta_2)d\overline{W}_s+O(e^{-Ct_i}).
      \end{align*}
      Therefore
      \begin{align}
        \label{eq18}\begin{split}
          &\frac{h}{2t_n}E\left[ \sup_{\theta_2 \in \Theta_2}\left|\sum_{i=1}^n\left\{{\Sigma^*}^{-1}[(c(\theta_2)\mu_{t_{i-1}}^n(\theta_2)-c^*\mu_{t_{i-1}}^n(\theta_2^*))^{\otimes 2}]+2\mathbb{Y}^2(\theta)\right\}\right|^p \right]^\frac{1}{p}\\
        \leq &\frac{h}{t_n}E\left[ \sup_{\theta_2 \in \Theta_2}\left|\sum_{i=1}^n\int_{0}^{t_i}{\nu^i_{s}(\theta_2)}'{\Sigma^*}^{-1}q_i(s;\theta_2)d\overline{W}_s\right|^p \right]^\frac{1}{p}+\frac{1}{2t_n}\sum_{i=1}^nCe^{-Ct_i}h\\
        \leq &\frac{h}{t_n}E\left[ \sup_{\theta_2 \in \Theta_2}\left|\sum_{i=1}^n\int_{0}^{t_i}{\nu^i_{s}(\theta_2)}'{\Sigma^*}^{-1}q_i(s;\theta_2)d\overline{W}_s\right|^p \right]^\frac{1}{p}+\frac{C}{t_n}.
        \end{split}        
      \end{align}
      Now by Lemma \ref{differentiable-lemma} and the continuos differentiability of $p_i$ and $q_i$, we can assume $\displaystyle \nu^i_t(\theta_2)$ is continuously differentiable with respect to $\theta_2$ and almost surely
      \begin{align*}
        \partial_{\theta_2}\nu^i_t(\theta_2)=\partial_{\theta_2}p_i(\theta_2)+\int_0^t\partial_{\theta_2}q_i(s;\theta_2)ds.
      \end{align*}
      Thus by Lemma \ref{BDG} (2) we obtain for any $T>0,p\geq 2$ and $\theta_2,\theta_2' \in \Theta_2$
      \begin{align*}
        \sup_{0\leq t\leq T}E\left[ |\nu_t^i(\theta_2)-\nu_t^i(\theta_2')|^p \right]\leq C_p|\theta_2-\theta_2'|^p
      \end{align*}
      and 
      \begin{align*}
        \sup_{0\leq t\leq T}E\left[ |\partial_{\theta_2}\nu_t^i(\theta_2)-\partial_{\theta_2}\nu_t^i(\theta_2')|^p \right]\leq C_p|\theta_2-\theta_2'|^p.
      \end{align*}
      Then again by Lemma \ref{differentiable-lemma}, $\displaystyle \int_{0}^{t_i}{\nu^i_{s}(\theta_2)}'{\Sigma^*}^{-1}q_i(s;\theta_2)d\overline{W}_s$ is continuously differentiable and we have almost surely
      \begin{align*}
        \partial_{\theta_2}\int_{0}^{t_i}{\nu^i_{s}(\theta_2)}'{\Sigma^*}^{-1}q_i(s;\theta_2)d\overline{W}_s
        =\int_{0}^{t_i}\partial_{\theta_2}\{{\nu^i_{s}(\theta_2)}'{\Sigma^*}^{-1}q_i(s;\theta_2)\}d\overline{W}_s.
      \end{align*}
      Therefore the Sobolev inequality gives for any $p>m_1+m_2$
      \begin{align}
        \label{eq17}\begin{split}
          &E\left[ \sup_{\theta_2 \in \Theta_2}\left|\sum_{i=1}^n\int_{0}^{t_i}{\nu^i_{s}(\theta_2)}'{\Sigma^*}^{-1}q_i(s;\theta_2)d\overline{W}_s\right|^p \right]^\frac{1}{p}\\
          =&E\left[ \sup_{\theta_2 \in \Theta_2}\left|\int_{0}^{t_n}\sum_{i=1}^n{\nu^i_{s}(\theta_2)}'{\Sigma^*}^{-1}q_i(s;\theta_2)1_{[0,t_i]}(s)d\overline{W}_s\right|^p \right]^\frac{1}{p}\\
        \leq &C_p\sup_{\theta_2 \in \Theta_2}E\left[ \left|\int_{0}^{t_n}\sum_{i=1}^n{\nu^i_{s}(\theta_2)}'{\Sigma^*}^{-1}q_i(s;\theta_2)1_{[0,t_i]}(s)d\overline{W}_s\right|^p \right]^\frac{1}{p}\\
        &+C_p\sup_{\theta_2 \in \Theta_2}E\left[ \left|\int_{0}^{t_n}\sum_{i=1}^n\partial_{\theta_2}\{{\nu^i_{s}(\theta_2)}'{\Sigma^*}^{-1}q_i(s;\theta_2)\}1_{[0,t_i]}(s)d\overline{W}_s\right|^p \right]^\frac{1}{p}.
        \end{split}        
      \end{align}
      Now we have $|p_t(\theta_2)|\leq Ce^{-Ct_i},|q_i(s;\theta_2)|\leq Ce^{-C(t_i-s)}$ and hence
      \begin{align*}
        E\left[|{\nu^i_{s}(\theta_2)}|^p  \right]\leq C_p.
      \end{align*}
      Thus we obtain
      \begin{align*}
        &E\left[\left|\sum_{i=1}^n{\nu^i_{s}(\theta_2)}'{\Sigma^*}^{-1}q_i(s;\theta_2)1_{[0,t_i]}(s)\right|^p\right]^\frac{1}{p}\\
        \leq &\sum_{i=1}^n|{\Sigma^*}^{-1}q_i(s;\theta_2)|E\left[\left|{\nu^i_{s}(\theta_2)}'\right|^p\right]^\frac{1}{p}1_{[0,t_i]}(s)\\
        \leq &\sum_{i=1}^nC_pe^{-C(t_i-s)}1_{[s,\infty)}(t_i)\\
        \leq &\sum_{i=0}^\infty C_pe^{-Ct_i}=\frac{1}{h}\sum_{i=0}^\infty C_pe^{-Ct_i}h\\
        \leq &\frac{C_p}{h}\int_{t_{-1}}^{\infty}e^{-Ct}dt \leq \frac{C_p}{h},
      \end{align*}
      and therefore by Lemma \ref{BDG}
      \begin{align*}
        &E\left[\left|\int_0^{t_n}\sum_{i=1}^n{\nu^i_{s}(\theta_2)}'{\Sigma^*}^{-1}q_i(s;\theta_2)1_{[0,t_i]}(s)d\overline{W}_s\right|^p\right]\\
        \leq &{t_n}^{\frac{p}{2}-1}\int_0^{t_n}E\left[\left|\sum_{i=1}^n{\nu^i_{s}(\theta_2)}'{\Sigma^*}^{-1}q_i(s;\theta_2)1_{[0,t_i]}(s)\right|^p\right]ds\\
        \leq &\frac{C_p}{h}{t_n}^{\frac{p}{2}}.
      \end{align*}
      In the same way, we obtain
      \begin{align*}
        E\left[ \left|\int_{0}^{t_n}\sum_{i=1}^n\partial_{\theta_2}\{{\nu^i_{s}(\theta_2)}'{\Sigma^*}^{-1}q_i(s;\theta_2)\}1_{[0,t_i]}(s)d\overline{W}_s\right|^p \right]^\frac{1}{p}\leq \frac{C_p}{h}{t_n}^{\frac{p}{2}}.
      \end{align*}
      Hence by (\ref{eq17}), it follows 
      \begin{align*}
        E\left[ \sup_{\theta_2 \in \Theta_2}\left|\sum_{i=1}^n\int_{0}^{t_i}{\nu^i_{s}(\theta_2)}'{\Sigma^*}^{-1}q_i(s;\theta_2)d\overline{W}_s\right|^p \right]^\frac{1}{p}
        \leq \frac{C_p}{h}{t_n}^{\frac{p}{2}}, 
      \end{align*}
      and therefore by (\ref{eq18})
      \begin{align}
        \label{eq19}\begin{split}
          &\frac{h}{2t_n}E\left[ \sup_{\theta_2 \in \Theta_2}\left|\sum_{i=1}^n{\Sigma^*}^{-1}[(c(\theta_2)\mu_{t_{i-1}}^n(\theta_2)-c^*\mu_{t_{i-1}}^n(\theta_2^*))^{\otimes 2}]-\mathbb{Y}^2(\theta)\right|^p \right]^\frac{1}{p}\\
        \leq &C_p\frac{h}{t_n}\frac{{t_n}^{\frac{p}{2}}}{h}+\frac{C}{t_n}
        \leq C_p\frac{1}{{t_n}^\frac{1}{2}}.
        \end{split}        
      \end{align}

      Finally, as for the third and fourth terms in (\ref{eq15}), by the Sobolev inequality, Lemma \ref{BDG} and (\ref{m-tilde-norm}) it holds
      \begin{align}
        \label{eq21}\begin{split}
          &E\left[\sup_{\theta_2 \in \Theta_2}\left|\sum_{i=1}^n\{\tilde{m}_{i-1}^n(\theta_2)'c(\theta_2)'-\tilde{m}_{i-1}^n(\theta_2^*)'{c^*}'\}{\Sigma^*}^{-1}\sigma^*\Delta_j\overline{W}\right|^p \right]\\
        \leq &C_p\sup_{\theta_2 \in \Theta_2}\left(E\left[\left|\sum_{i=1}^n\{\tilde{m}_{i-1}^n(\theta_2)'c(\theta_2)'-\tilde{m}_{i-1}^n(\theta_2^*)'{c^*}'\}{\Sigma^*}^{-1}\sigma^*\Delta_j\overline{W}\right|^p \right]\right.\\
        &\left.+E\left[\left|\sum_{i=1}^n\partial_{\theta_2}\{\tilde{m}_{i-1}^n(\theta_2)'c(\theta_2)'\}{\Sigma^*}^{-1}\sigma^*\Delta_j\overline{W}\right|^p \right]\right)\\
        \leq &C_p\sup_{\theta_2 \in \Theta_2}\left({t_n}^{\frac{p}{2}-1}\sum_{i=1}^nE\left[\left|\{\tilde{m}_{i-1}^n(\theta_2)'c(\theta_2)'-\tilde{m}_{i-1}^n(\theta_2^*)'{c^*}'\}{\Sigma^*}^{-1}\sigma^*\right|^p \right]h\right.\\
        &\left.+{t_n}^{\frac{p}{2}-1}\sum_{i=1}^nE\left[\left|\partial_{\theta_2}\{\tilde{m}_{i-1}^n(\theta_2)'c(\theta_2)'\}{\Sigma^*}^{-1}\sigma^*\right|^p \right]h\right)\\
        \leq &C_p{t_n}^{\frac{p}{2}}.
        \end{split}        
      \end{align}

      We obtain the desired result by putting (\ref{eq15}), (\ref{eq20}), (\ref{eq19}) and (\ref{eq21}) together.      
    \end{proof}

    Now we set
    \begin{align}
      \label{def-captal-M}\tilde{M}_j^n(\theta_2)=c(\theta_2)\tilde{m}_j^n(\theta).
    \end{align}
    Then by  (\ref{def-tilde-delta}) and (\ref{def-tilde-gamma}), we have
    \begin{align*}
      \tilde{\Gamma}_n^2&=\frac{1}{t_n}\sum_{i=1}^n\left\{ {\Sigma^*}^{-1}[\partial_{\theta_2}^{\otimes 2}]\tilde{M}_{i}^n(\theta^*)h\right.\\
      &\left.-\partial_{\theta_2}^2\tilde{M}_{i-1}^n(\theta^*)'{{\sigma^*}'}^{-1}\Delta_j\overline{W}-\Delta_j\overline{W}'{\sigma^*}^{-1}\partial_{\theta_2}^2\tilde{M}_{i-1}^n(\theta^*) \right\}.
    \end{align*}

    Moreover, by (\ref{def-Gamma2}) and (\ref{mu-formula}), we obtain the following results in the same way as Propositions \ref{c-mu-diff} and \ref{Y-tilde-Y-diff}:
    \begin{align}
      &E\left[ {\Sigma^*}^{-1}[\partial_{\theta_2}^{\otimes 2}]\tilde{M}_{i}^n(\theta^*) \right]=\Gamma^2+O(e^{-Ct_i})\\
      &\label{Gamma-converge}E\left[ \left|\frac{1}{t_n}\sum_{i=1}^n{\Sigma^*}^{-1}[\partial_{\theta_2}^{\otimes 2}]\tilde{M}_{i}^n(\theta^*)h-\Gamma^2\right|^p \right]\leq C_p\left( h^p+n^{-\frac{1}{2}p}+\frac{1}{{t_n}^{\frac{p}{2}}} \right)\\
      &\label{Gamma-converge2}E\left[ |\tilde{\Gamma}_n-\Gamma^2|^p \right]\leq C_p\left( h^p+n^{-\frac{1}{2}p}+\frac{1}{{t_n}^{\frac{p}{2}}} \right).
    \end{align}

    \begin{proposition}\label{Delta_n-normal}
      It holds
      \begin{align*}
        \tilde{\Delta}_n^2 \xrightarrow{d} N(0,\Gamma^2).
      \end{align*}
    \end{proposition}
    \begin{proof}
      Since $\tilde{\Delta}_n^2$ is given by the formula (\ref{tilde-Delta-formula}), we set
      \begin{align*}
        \xi_i^n=\frac{1}{\sqrt{t_n}}\partial_{\theta_2}\tilde{M}_{i-1}^n(\theta_2^*)'{{\sigma^*}'}^{-1}\Delta_{i} \overline{W}.
      \end{align*} 
      Then $(\xi_i^n)^{\otimes 2}$ is the matrix whose $(i,j)$ entry is
      \begin{align*}
        &\frac{1}{t_n}\frac{\partial}{\partial \theta_2^{i}}\tilde{M}_{i-1}^n(\theta_2^*)'{{\sigma^*}'}^{-1}\Delta_{i} \overline{W}\frac{\partial}{\partial \theta_2^{j}}\tilde{M}_{i-1}^n(\theta_2^*)'{{\sigma^*}'}^{-1}\Delta_{i} \overline{W}\\
        =&\frac{1}{t_n}\frac{\partial}{\partial \theta_2^{i}}\tilde{M}_{i-1}^n(\theta_2^*)'{{\sigma^*}'}^{-1}\Delta_{i} \overline{W}\Delta_{i} \overline{W}'{{\sigma^*}}^{-1}\frac{\partial}{\partial \theta_2^{j}}\tilde{M}_{i-1}^n(\theta_2^*).
      \end{align*}
      Hence it follows from (\ref{Gamma-converge})
      \begin{align*}
        \sum_{i=1}^nE\left[(\xi_i^n)^{\otimes 2}|\mathcal{F}_{t_{i-1}}  \right]
        =\sum_{i=1}^nE\left[{\Sigma^*}^{-1}[\partial_{\theta_2}^{\otimes 2}]\tilde{M}_{i-1}^n(\theta_2)|\mathcal{F}_{t_{i-1}}  \right]\xrightarrow{P}\Gamma^2~(n \to \infty).
      \end{align*}
      Moreover, we have for $\epsilon>0$ 
        \begin{align*}
          &\sum_{i=1}^nE[|\xi_i^n|^21_{\{|\xi_i^n|>\epsilon\}}|\mathcal{F}_{t_{i-1}}]\\
          \leq &\sum_{i=1}^nE[|\xi_i^n|^4|\mathcal{F}_{t_{i-1}}]^{\frac{1}{2}}P(|\xi_i^n|>\epsilon|\mathcal{F}_{t_{i-1}})^{\frac{1}{2}}\\
          \leq &\sum_{i=1}^nE[|\xi_i^n|^4|\mathcal{F}_{t_{i-1}}]^{\frac{1}{2}}\times \frac{1}{\epsilon^2}E[|\xi_i^n|^4|\mathcal{F}_{t_{i-1}}]^{\frac{1}{2}}\\
          =&\sum_{i=1}^n\frac{|{\sigma^*}^{-1}|^4}{\epsilon^2{t_n}^2}|\partial_{\theta}\tilde{M}_{i-1}^n(\theta^*)|^4E[(\Delta_i\overline{W})^4]\\
          \leq&\frac{|{\sigma^*}^{-1}|^4}{\epsilon^2{t_n}^2}\sum_{i=1}^n|\partial_{\theta}\tilde{M}_{i-1}^n(\theta^*)|^4h^2,
        \end{align*}
        and hence
        \begin{align*}
          E\left[ \sum_{i=1}^nE[|\xi_i^n|^21_{\{|\xi_i^n|>\epsilon\}}|\mathcal{F}_{t_{i-1}}] \right]
          \leq &\sum_{i=1}^n\frac{|{\sigma^*}^{-1}|^4}{\epsilon^2{t_n}^2}E[|\partial_{\theta}\tilde{M}_{i-1}^n(\theta^*)|^4]h^2\\
          \leq  &C_{\epsilon}\sum_{i=1}^n\frac{1}{{t_n}^2}h^2=\frac{C_{\epsilon}}{n}\to 0~(n \to \infty).
        \end{align*}

        Therefore we obtain the desired result by the martingale central limit theorem.
    \end{proof}

    \begin{proposition}\label{partial-3-norm}
      For any $p>m_1+m_2$, it holds
      \begin{align*}
        \sup_{n \in \mathbb{N}}E\left[ \sup_{\theta_2\in\Theta_2}\left|\frac{1}{t_n}\partial_{\theta_2}^3 \mathbb{H}_n^2(\theta_2)\right|^p \right]^\frac{1}{p}<\infty.
      \end{align*}
    \end{proposition}
    \begin{proof}
      By (\ref{m-hat-norm}), we have
      \begin{align*}
        &E\left[ \sup_{\theta_2\in\Theta_2}\left|\sum_{i=1}^nh\partial_{\theta_2}^3\{\Sigma(\hat{\theta}_1^n)^{-1}[(\hat{M}_{j-1}^n(\theta_2))^{\otimes 2}]\} \right|^p \right]^\frac{1}{p}\\
        \leq &h\sum_{i=1}^nE\left[ \sup_{\theta_2\in\Theta_2}\left|\partial_{\theta_2}^3\{\Sigma(\hat{\theta}_1^n)^{-1}[(\hat{M}_{j-1}^n(\theta_2))^{\otimes 2}]\} \right|^p \right]^\frac{1}{p}\\
        \leq &C_pnh.
      \end{align*}
      Moreover, we obtain by Proposition \ref{m-hat-dY-norm} 
      \begin{align*}
        &E\left[\sup_{\theta_2\in\Theta_2}\left|\sum_{j=1}^n\partial_{\theta_2}^3\hat{M}_{j-1}^n(\theta_2)'\Sigma(\hat{\theta}_1^n)^{-1}\Delta_jY\right|^p \right]^\frac{1}{p}\\
        \leq &C_p(nh+(nh)^\frac{1}{2}),
      \end{align*}
      and in the same way
      \begin{align*}
        &E\left[\sup_{\theta_2\in\Theta_2}\left|\sum_{j=1}^n\Delta_jY\Sigma(\hat{\theta}_1^n)^{-1}\partial_{\theta_2}^3\hat{M}_{j-1}^n(\theta_2)\right|^p \right]^\frac{1}{p}\\
        \leq &C_p(nh+(nh)^\frac{1}{2}).
      \end{align*}
      Therefore it follows from (\ref{def-H2})
      \begin{align*}
        &E\left[ \sup_{\theta_2\in\Theta_2}\left|\frac{1}{t_n}\partial_{\theta_2}^3 \mathbb{H}_n^2(\theta_2)\right|^p \right]^\frac{1}{p}
        \leq &\frac{C_p}{t_n}\{nh+(nh)^\frac{1}{2}\}=C_p(1+{t_n}^{-\frac{1}{2}})\leq C_p.
      \end{align*}
    \end{proof}

    \begin{proof}[\bf{Proof of Theorem \ref{main-theorem-2}}]
      We set $\Delta_n^2$, $\Gamma_n^2$ and $\mathbb{Y}_n^2$ by
      \begin{align}
        &\mathbb{Y}_n^2(\theta_2)=\frac{1}{t_n}\{\mathbb{H}_n^2(\theta_2)-\mathbb{H}_n^2(\theta_2^*)\}\\
        &\Delta_n^2=\frac{1}{\sqrt{t_n}}\partial_{\theta}\mathbb{H}_n^2(\theta_2^*)\\
        &\Gamma_n^2=-\frac{1}{t_n}\partial_{\theta}^2\mathbb{H}_n^2(\theta_2^*).
      \end{align}
      Then by Proposition \ref{bold-H-diff} for any $n \in \mathbb{N}$ and $p>m_1+m_2$, it holds
      \begin{align}
        \label{Delta2-tilde-Delta2-diff}&E\left[ \left|\Delta_n^2-\tilde{\Delta}_n^2\right|^p \right]^\frac{1}{p}\leq C_p\left( n^\frac{1}{2}+h^\frac{1}{2}+(nh)^{-1} \right)\\
        &E\left[ \left|\Gamma_n^2-\tilde{\Gamma}_n^2\right|^p \right]^\frac{1}{p}\leq C_p\left( h^\frac{1}{2}+n^{-\frac{1}{2}}+(nh)^{-1} \right)
      \end{align}
      and
      \begin{align*}
        E\left[\sup_{\theta_2 \in \Theta_2}\left|\mathbb{Y}_n^2(\theta_2)-\tilde{\mathbb{Y}}_n^2(\theta_2)\right|^p \right]^\frac{1}{p}\leq C_p\left( h^\frac{1}{2}+n^{-\frac{1}{2}}+(nh)^{-1} \right).
      \end{align*}

      Together with Proposition \ref{tilde-Dalta-norm}, (\ref{Gamma-converge2}) and Proposition \ref{Y-tilde-Y-diff}, we have for any $p>m_1+m_2$ (therefore for any $p>0$)
      \begin{align}
        \label{proof-eq1}&\sup_{n \in \mathbb{N}}E\left[ |\Delta_n^2|^p \right]^\frac{1}{p}<\infty,\\
        \label{proof-eq2}&\sup_{n \in \mathbb{N}}E\left[ \left|{t_n}^\frac{1}{2}(\Gamma_n^2-\Gamma^2)\right|^p \right]^\frac{1}{p}<\infty
      \end{align}
      and
      \begin{align}
        \label{proof-eq3}\sup_{n \in \mathbb{N}}E\left[\sup_{\theta_2 \in \Theta_2}\left|{t_n}^\frac{1}{2}(\mathbb{Y}_n^2(\theta)-\mathbb{Y}^2(\theta_2))\right|^p \right]^\frac{1}{p}<\infty.
      \end{align}
      Moreover, by Proposition \ref{Delta_n-normal} and (\ref{Delta2-tilde-Delta2-diff}) we obtain
      \begin{align}
        \label{proof-eq4}\Delta_n \xrightarrow{d} N(0,\Gamma^2).
      \end{align}

      Then we have proved the theorem by the assumption [A5], Proposition \ref{partial-3-norm}, (\ref{proof-eq1})-(\ref{proof-eq4}) and Theorem 5 in \cite{YOSHIDA2011}.

    \end{proof}

    \section{One-dimensional case}\label{one-dimensional-case-section}
    In this section, we consider the special case where $d_1=d_2=1$. In this case, $a(\theta_2),b(\theta_2),c(\theta_2)$ and $\sigma(\theta_1)$ are scalar valued, so we set $m_1=1$ and $\sigma(\theta_1)=\theta_1$. Moreover, we assume $\Theta_1 \subset (\epsilon,\infty)$ for some $\epsilon>0$. Then (\ref{ARE}) can be reduced to
    \begin{align*}
      \frac{c(\theta_2)^2}{{\theta_1}^2}\gamma^2+2a(\theta_2)\gamma+b(\theta_2)^2=0,
    \end{align*}
    and the larger solution of this is
    \begin{align*}
      \gamma_+(\theta_1,\theta_2)=\frac{{\theta_1}^2a(\theta_2)}{c(\theta_2)^2}\left( \sqrt{1+\frac{b(\theta_2)^2c(\theta_2)^2}{{\theta_1}^2a(\theta_2)^2}}-1 \right).
    \end{align*}
    Thus we have
    \begin{align}
      \label{alpha-onedim}\alpha(\theta_1,\theta_2)=\sqrt{a(\theta_2)^2+\frac{b(\theta_2)^2c(\theta_2)^2}{{\theta_1}^2}}
    \end{align}
    by (\ref{def-alpha}). Furthermore, the eigenvalues of $H(\theta_1,\theta_2)$ in Assumption [A4] is $\pm \alpha(\theta_1,\theta_2)$ and hence one can remove Assumption [A4].     

    As for the estimation of $\theta_1$, one can obtain the explicit expression of $\hat{\theta}_1^n$. In fact, we have
    \begin{align*}
      \mathbb{H}_n^1(\theta_1)=-\frac{1}{2}\sum_{j=1}^n\left\{ \frac{1}{h{\theta_1}^2}(\Delta_jY)^{2}+2\log \theta_1 \right\}
    \end{align*}
    and hence
    \begin{align*}
      \frac{d}{d\theta_1}\mathbb{H}_n^1(\theta_1)=\frac{1}{h{\theta_1}^3}\sum_{j=1}^n(\Delta_jY)^{2}-\frac{n}{\theta_1}.
    \end{align*}
    Thus we obtain the formula
    \begin{align*}
      \hat{\theta}_1^n=\left( \frac{1}{t_n}\sum_{j=1}^n(\Delta_jY)^{2} \right)^\frac{1}{2}.
    \end{align*}    
    Moreover, $\mathbb{Y}_1(\theta_1)$ and $\Gamma^1$ can be written as
    \begin{align*}
      \mathbb{Y}_1(\theta_1)=-\frac{1}{2}\left( \frac{{\theta_1^*}^2}{{\theta_1}^2}-1-2\log\frac{{\theta_1^*}}{{\theta_1}} \right)
    \end{align*}
    and
    \begin{align*}
      \Gamma^1=\frac{1}{2}\left( \frac{2\theta_1^*}{{\theta_1^*}^2} \right)^2=\frac{2}{{\theta_1^*}^2}.
    \end{align*}
    Therefore noting that $\displaystyle x^2-1-2\log x\geq (x-1)^2~(x\geq 0)$ we have
    \begin{align*}
      \mathbb{Y}_1(\theta_1)\leq -\frac{1}{2}\left( \frac{{\theta_1^*}}{{\theta_1}}-1 \right)^2\leq \frac{(\theta_1-\theta_1^*)^2}{2\epsilon^2}
    \end{align*}
    and hence (\ref{Y1-ineq}) holds.

    As for the estimation of $\theta_2$, since we have 
    \begin{align}
      \label{gamma-alpha}\gamma(\theta_1,\theta_2)=\frac{{\theta_1}^2}{c(\theta_2)^2}\left\{ \alpha(\theta_1,\theta_2)-a(\theta_2) \right\}
    \end{align}
    by (\ref{def-alpha}), we obtain for $\alpha(\theta_2)\neq a^*$ 
    \begin{align}
      \label{Y2-onedim}\begin{split}
        &\mathbb{Y}_2(\theta_2)=-\frac{1}{2}\int_0^\infty\left\{-\frac{\{a(\theta_2)-a^*\}(\alpha(\theta_2^*)-a^*)}{\alpha(\theta_2)-a^*}e^{-a^*s}\right.\\
      &\left.+\frac{\{\alpha(\theta_2)-\alpha(\theta_2^*)\}\{\alpha(\theta_2)-a(\theta_2)\}}{\alpha(\theta_2)-a^*}e^{-\alpha(\theta_2)s}  \right\}^2ds\\
      =&-\frac{1}{4a^*\alpha(\theta_2)\{a^*+\alpha(\theta_2)\}}\\
      &\times \left[\{a^*\alpha(\theta_2)-a(\theta_2)\alpha(\theta_2^*)\}^2 +a^*\alpha(\theta_2)\left\{\alpha(\theta_2)-a(\theta_2)-\alpha(\theta_2^*)+a^* \right\}^2 \right]\\
      =&-\frac{a^*a(\theta)^2}{4\alpha(\theta_2)\{a^*+\alpha(\theta_2)\}}\left\{ \frac{\alpha(\theta_2)}{a(\theta_2)}-\frac{\alpha(\theta_2^*)}{a(\theta_2^*)} \right\}^2\\
      &-\frac{1}{4a^*\alpha(\theta_2)\{a^*+\alpha(\theta_2)\}}\left\{\alpha(\theta_2)-a(\theta_2)-\alpha(\theta_2^*)+a^* \right\}^2,
      \end{split}      
    \end{align}
    making use of (\ref{def-Y2}) and the identity
    \begin{align*}
      \int_0^\infty \left( pe^{-\alpha t}+qe^{-\beta t} \right)^2dt
      &=\frac{p^2}{2\alpha}+\frac{2pq}{\alpha+\beta}+\frac{q^2}{2\beta}\\
      &=\frac{1}{2\alpha\beta}\left\{ (\alpha q+\beta p)^2+\alpha\beta(p-q)^2 \right\},
    \end{align*}
    where $\alpha,\beta>0$ and $p,q \in \mathbb{R}$. Even if $\alpha(\theta_2)=a^*$, we obtain the same formula by letting $a^* \to \alpha(\theta_2)$ in (\ref{Y2-onedim}).

     Now we obtain a sufficient condition for (\ref{Y2-ineq}) by the following proposition. 

    \begin{proposition}
      Assume [A3], $\displaystyle \inf_{\theta_2 \in \Theta_2}|c(\theta_2)|>C$ and 
      \begin{align}
        \label{onedim-separate}|a(\theta_2)-a(\theta_2^*)|+|\alpha(\theta_2)-\alpha(\theta_2^*)|\geq C|\theta_2-\theta_2^*|.
      \end{align}
      Then it holds
      \begin{align}
        \label{one-dim-Y}Y(\theta_2)\leq -C|\theta_2-\theta_2^*|^2.
      \end{align}
    \end{proposition}
    \begin{proof}
      Let us assume there is no constant $C$ satisfying (\ref{one-dim-Y}). Then there exists some sequence $\theta_2^{(n)} \in \Theta_2~(n \in \mathbb{N})$ such that
      \begin{align*}
        \left|\frac{\alpha(\theta_2^{(n)})}{a(\theta_2^{(n)})}-\frac{\alpha(\theta_2^*)}{a(\theta_2^*)}\right|\leq \frac{1}{n}|\theta_2^{(n)}-\theta_2^*|
      \end{align*}
      and
      \begin{align*}
        \left|\alpha(\theta_2^{(n)})-a(\theta_2^{(n)})-\alpha(\theta_2^*)+a^*\right|\leq \frac{1}{n}|\theta_2^{(n)}-\theta_2^*|.
      \end{align*}
      Thus if we set
      \begin{align*}
        A(\theta_2)=\alpha(\theta_2)-a(\theta_2)
      \end{align*}
      and
      \begin{align*}
        B(\theta_2)=\frac{\alpha(\theta_2)}{a(\theta_2)},
      \end{align*}
      it follows that 
      \begin{align*}
        &|a(\theta_2^{(n)})-a(\theta_2^*)|
        =\left|\frac{A(\theta_2^{(n)})}{B(\theta_2^{(n)})-1}-\frac{A(\theta_2^*)}{\displaystyle B(\theta_2^*)-1}\right|\\
        \leq &\frac{|A(\theta_2^{(n)})-A(\theta_2^*)|}{B(\theta_2^{(n)})-1}+\frac{|A(\theta_2^*)||B(\theta_2^{(n)})-B(\theta_2^*)|}{\{ B(\theta_2^{(n)})-1 \}\{ B(\theta_2^*)-1 \}}\\
        \leq &\frac{C}{n}|\theta_2^{(n)}-\theta_2^*|,
      \end{align*}
      noting that it holds $B(\theta_2)-1>C$ by the assumptions and (\ref{alpha-onedim}). In the same, way we have
      \begin{align*}
        |\alpha(\theta_2^{(n)})-\alpha(\theta_2^*)|\leq \frac{C}{n}|\theta_2^{(n)}-\theta_2^*|,
      \end{align*}
      but these contradict (\ref{onedim-separate}).
    \end{proof}

    We similarly obtain the explicit expression of $\Gamma^2$ by (\ref{gamma-alpha}):
    \begin{align*}
      \Gamma^2=&\frac{1}{{\theta_1^*}^2}\int_0^\infty[\partial_{\theta_2}^{\otimes 2}]\left\{\int_0^sc(\theta_2)\exp(-\alpha(\theta_2)u)\gamma_{+}(\theta_2)c(\theta_2)'{\Sigma^*}^{-1}c^*\right.\\
    &\qquad\qquad\qquad\exp(-a^*(s-u))\gamma_+(\theta^*){c^*}'du\\
    &\qquad\qquad\qquad\left.\left.+c(\theta_2)\exp(-\alpha(\theta_2)s)\gamma_{+}(\theta_2)c(\theta_2)'\right.\biggr\}\right|_{\theta_2=\theta_2^*}ds\\
    =&\int_0^\infty\{\partial_\theta\alpha(\theta^*)e^{-\alpha^*s}-\partial_\theta a(\theta^*)e^{-a^*s}\}^{\otimes 2}ds\\
    =&\frac{\{\partial_{\theta_2}\alpha(\theta^*)\}^{\otimes 2}}{2\alpha^*}+\frac{\{\partial_{\theta_2}a(\theta^*)\}^{\otimes 2}}{2a^*}\\
    &-\frac{\partial_{\theta_2}\alpha(\theta^*)\otimes \partial_{\theta_2}a(\theta^*)+\partial_{\theta_2} a(\theta^*)\otimes \partial_{\theta_2}\alpha(\theta^*)}{\alpha^*+a^*}\\
    =&\frac{1}{2\alpha^*}\left( \partial_{\theta_2}\alpha(\theta^*)-\frac{2\alpha^*}{\alpha^*+a^*}\partial_{\theta_2}a(\theta^*) \right)^{\otimes 2}+\frac{(\alpha^*)^2+(a^*)^2}{2(\alpha^*+a^*)a^*}\{\partial_{\theta_2}a(\theta^*)\}^{\otimes 2}.
    \end{align*}
    Hence $\Gamma^2$ is positive definite if and only if $\{\partial_{\theta_2}a(\theta^*)\}^{\otimes 2}$ or $\{\partial_{\theta_2}\alpha(\theta^*)\}^{\otimes 2}$ is positive definite. This does not happen if $m_2\geq 3$; in fact, one can take $x \in \mathbb{R}^{m_2}$ so that $x'\partial_{\theta_2}a(\theta^*)=x'\partial_{\theta_2}\alpha(\theta^*)$ if $m_2\geq 3$. Thus we need to assume $m_2 \leq 2$ in the one-dimensional case.

    Putting it all together, we obtain the following result.
    \begin{theorem}
      Let $m_1=1$, $m_2 \leq 2$, $\sigma(\theta_1)=\theta_2$ and $\Theta_1 \subset (\epsilon,\infty)$ for some $\epsilon>0$. Moreover, we assume [A1], [A2] and the following conditions:
      \begin{description}
        \item[{[B1]}] \begin{align*}
          &\inf_{\theta_2 \in \Theta_2}a(\theta_2)>0\\
          &\inf_{\theta_2 \in \Theta_2}|b(\theta_2)|>0\\
          &\inf_{\theta_2 \in \Theta_2}|c(\theta_2)|>0
        \end{align*}
        \item[{[B2]}] For any $\theta_1 \in \Theta_1$ and $\theta_2,\theta_2^* \in \Theta_2$,
        \begin{align*}
          |a(\theta_2,\theta_1)-a(\theta_2^*,\theta_1)|+|\alpha(\theta_2,\theta_1)-\alpha(\theta_2^*,\theta_1)|\geq C_{\theta_1}|\theta_2-\theta_2^*|.
        \end{align*}
        \item[{[B3]}] For any $\theta \in \Theta$, $\{\partial_{\theta_2}a(\theta)\}^{\otimes 2}$ or $\{\partial_{\theta_2}\alpha(\theta)\}^{\otimes 2}$ is positive definite.
      \end{description}
      (1) If we set
      \begin{align*}
        \hat{\theta}_1^n=\left( \frac{1}{t_n}\sum_{j=1}^n(\Delta_jY)^{2} \right)^\frac{1}{2},
      \end{align*}
      then for every $p>0$ and any continuous function $f:\mathbb{R}^d \to \mathbb{R}$ such that
      \begin{align*}
        \limsup_{|x|\to \infty}\frac{|f(x)|}{|x|^p}<\infty,
      \end{align*} 
      it holds that 
      \begin{align*}
        E[f(\sqrt{n}(\hat{\theta}^n_1-\theta_1^*))]\to E[f(Z)]~(n \to \infty),
      \end{align*}
      where $\displaystyle Z \sim N\left(0,\frac{{\theta_1^*}^2}{2}\right)$.
    
      In particular, it holds that 
      \begin{align*}
        \sqrt{n}(\hat{\theta}^n_1-\theta_1^*)\xrightarrow{d}N\left(0,\frac{{\theta_1^*}^2}{2}\right)~(n \to \infty).
      \end{align*}
      (2) Let us define $\gamma_+(\theta_1,\theta_2)$ and $\alpha(\theta_1,\theta_2)$ by
      \begin{align*}
        \gamma_+(\theta_1,\theta_2)=\frac{{\theta_1}^2a(\theta_2)}{c(\theta_2)^2}\left( \sqrt{1+\frac{b(\theta_2)^2c(\theta_2)^2}{{\theta_1}^2a(\theta_2)^2}}-1 \right).
      \end{align*}
      and
      \begin{align*}
        \alpha(\theta_1,\theta_2)=\sqrt{a(\theta_2)^2+\frac{b(\theta_2)^2c(\theta_2)^2}{{\theta_1}^2}},
      \end{align*}
      respectively, and set
      \begin{align*}
        \begin{split}
         &\hat{m}_i^n(\theta_2;m_0)=e^{-\alpha(\hat{\theta}_1^n,\theta_2)t_i}m_0\\
        &+\frac{1}{(\hat{\theta}_1^n)^2}\sum_{j=1}^ie^{-\alpha(\hat{\theta}_1^n,\theta_2)(t_i-t_{j-1})}\gamma_+(\hat{\theta}_1^n,\theta_2)c(\theta_2)\Delta_jY,
        \end{split}\\
        \begin{split}
          &\mathbb{H}_n^2(\theta_2;m_0)=\frac{1}{2}\sum_{i=1}^n\left\{-h(c(\theta_2)\hat{m}_{j-1}^n(\theta_2))^{2}+2\hat{m}_{j-1}^n(\theta_2)c(\theta_2)\Delta_jY\right\},
        \end{split}
      \end{align*}
      and
      \begin{align*}
         \begin{split}
          \Gamma^2=&\frac{\{\partial_{\theta_2}\alpha(\theta^*)\}^{\otimes 2}}{2\alpha^*}+\frac{\{\partial_{\theta_2}a(\theta^*)\}^{\otimes 2}}{2a(\theta^*)}\\
          &-\frac{\partial_{\theta_2}\alpha(\theta^*)\otimes \partial_{\theta_2}a(\theta^*)+\partial_{\theta_2} a(\theta^*)\otimes \partial_{\theta_2}\alpha(\theta^*)}{\alpha(\theta^*)+a(\theta^*)},
         \end{split}
      \end{align*}
      where $m_0 \in \mathbb{R}^{d_1}$ is an arbitrary initial value.

      Then, if $\hat{\theta}^n_2=\hat{\theta}^n_2(m_0)$ is a random variable satisfying
  \begin{align*}
    \mathbb{H}_n^2(\hat{\theta}^n_2)=\max_{\theta_2 \in \overline{\Theta}_2}\mathbb{H}_n^2(\theta_2)
  \end{align*} 
  for each $n \in \mathbb{N}$, then for any $p>0$ and continuous function $f:\mathbb{R}^d \to \mathbb{R}$ such that
  \begin{align*}
    \limsup_{|x|\to \infty}\frac{|f(x)|}{|x|^p}<\infty,
  \end{align*} 
  it holds that 
  \begin{align*}
    E[f(\sqrt{t_n}(\hat{\theta}^n_2-\theta_2^*))]\to E[f(Z)]~(n \to \infty),
  \end{align*}
  where $Z \sim N(0,(\Gamma^2)^{-1})$.

  In particular, it holds that 
  \begin{align*}
    \sqrt{t_n}(\hat{\theta}^n_2-\theta_2^*)\xrightarrow{d}N(0,(\Gamma^2)^{-1})~(n \to \infty).
  \end{align*}
    \end{theorem}

    \section{Simulations}\label{simulation-section}
    In this section, we will verify the results of the previous sections by computational simulations. We set $d_1=d_2=1$ and consider the equations
    \begin{align*}
      \begin{cases}
        dX_t=-aX_tdt+bdW_t^1\\
        dY_t=X_tdt+\sigma dW_t^2
      \end{cases}
    \end{align*}
    with $X_0=Y_0=0$, where we want to estimate $\theta_1=\sigma$ and $\theta_2=(a,b)$ from observations of $Y_t$.

    We generated sample data $Y_{t_i}~(i=0,1,\cdots,n)$ with $n=10^6$, $h=0.0001$ and true parameters $(a,b,\sigma)=(1.5,0.3,0.002)$, and performed three simulations:
    \begin{description}
      \item[{Simulation (i)}] $m_0=0,\gamma_0=0.1$.
      \item[{Simulation (ii)}] $m_0=1,\gamma_0=0.1$.
      \item[{Simulation (iii)}] $m_0=1,\gamma_0=0.1$ and removed first 100 terms of $\hat{m}_i^n$; i.e. we replaced $\mathbb{H}_n^2(\theta_2;m_0)$ with
      \begin{align*}
        \frac{1}{2}\sum_{i=101}^n\left\{-h(c(\theta_2)\hat{m}_{j-1}^n(\theta_2))^{2}+2\hat{m}_{j-1}^n(\theta_2)c(\theta_2)\Delta_jY\right\}.
      \end{align*}
    \end{description}
    For each simulation, we performed 10000 Monte Carlo replications. Table \ref{table:result-table} shows the means and standard deviations for estimators in each simulation, and one can observe asymptotic normalities of them in Figure \ref{hist}.

    We found that the wrong value of $m_0$ can affect the accuracy of our estimator, but it can be improved by leaving out first several terms of $\hat{m}_i^n$. One can observe from Figure \ref{X-m-graph} that $\hat{m}_i^n(\theta^*)$ well approximate $X_{t_i}$ except at the beginning. It will be interesting to consider the way to decide how many terms of $\hat{m}_i^n(\theta)$ should be removed. 
    
    \begin{table}[hbtp]
      \caption{The summary of the simulation results}
      \label{table:data_type}
      \centering
      \begin{tabular}{|c||c|c|c|}
        \hline
        &$\sigma$&$a$&$b$  \\
        \hline
        True value&0.02&1.5&0.3\\
        (Standard error)&$(1.414 \times 10^{-5})$&(0.2115)&(0.01324)\\ 
        \hline
        Simulation (i)  &   & \begin{tabular}{c}
          1.495\\(0.2123)
        \end{tabular}& \begin{tabular}{c}
          0.3011\\(0.01338)
        \end{tabular}\\
        Simulation (ii)  & \begin{tabular}{c}
          0.02007\\(1.640$\times 10^{-5}$)
        \end{tabular} & \begin{tabular}{c}
          1.715\\(0.2452)
        \end{tabular}&\begin{tabular}{c}
          0.3249\\(0.01338)
        \end{tabular}  \\
        Simulation (iii)  &   & \begin{tabular}{c}
          1.535\\(0.2177)
        \end{tabular}& \begin{tabular}{c}
          0.3059\\(0.01304)
        \end{tabular}\\
        \hline
      \end{tabular}
      \label{table:result-table}
    \end{table}

    \begin{figure}[H]
      \begin{minipage}[htbp]{0.6\linewidth}
        \centering
        \includegraphics[keepaspectratio, scale=0.4]{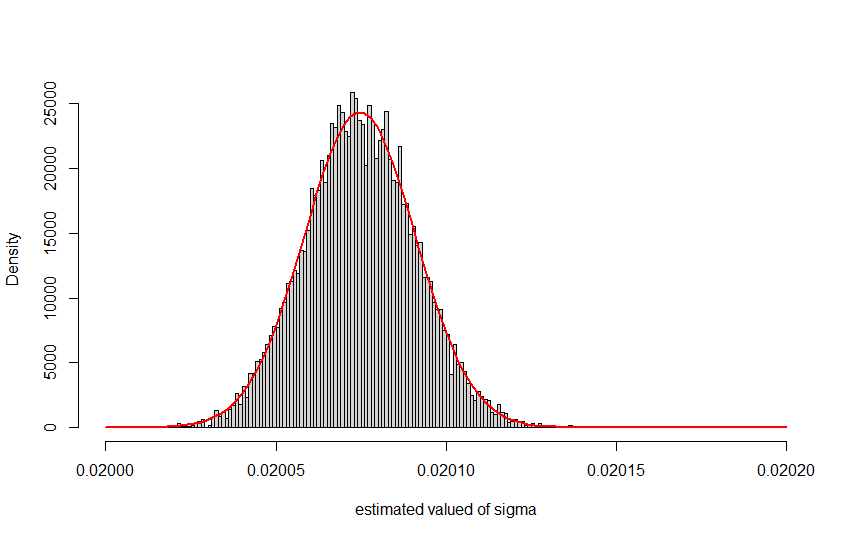}
        \subcaption{Estimated values of $\sigma$.}
      \end{minipage}\\
      \begin{minipage}[htbp]{0.6\linewidth}
        \centering
        \includegraphics[keepaspectratio, scale=0.4]{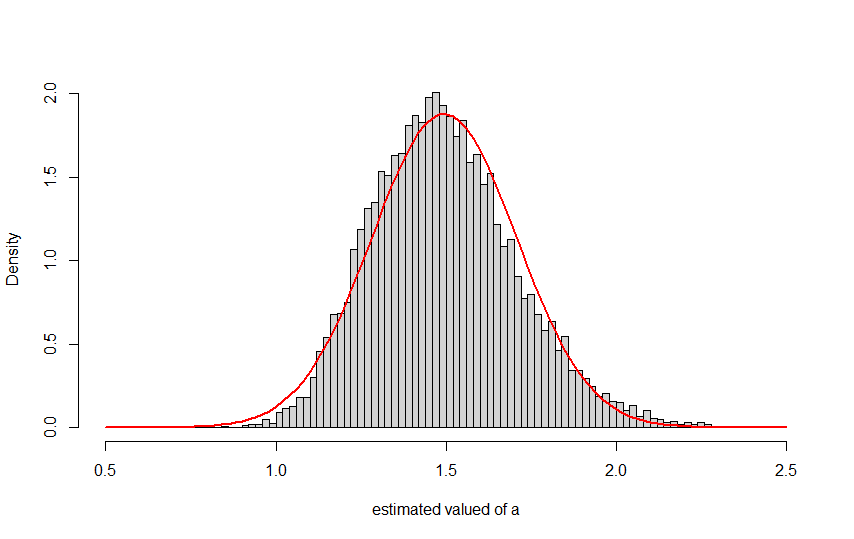}
        \subcaption{Estimated values of $a$.}
      \end{minipage}\\
      \begin{minipage}[htbp]{0.6\linewidth}
        \centering
        \includegraphics[keepaspectratio, scale=0.4]{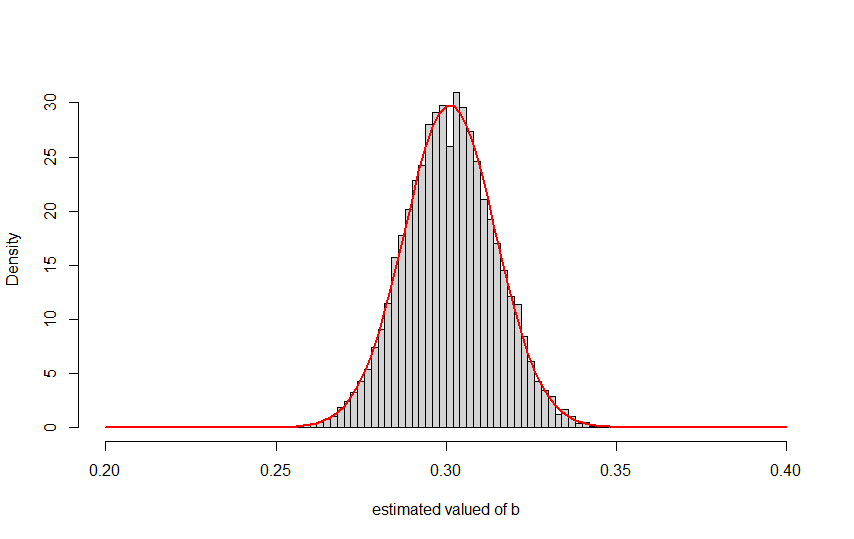}
        \subcaption{Estimated values of $b$.}
      \end{minipage}
      \caption{Histograms of estimators in Simulation (i). The red lines are the density of the normal distribution.}
      \label{hist}
    \end{figure}

    \begin{figure}[htbp]
      \centering
        \includegraphics[keepaspectratio, scale=0.5]{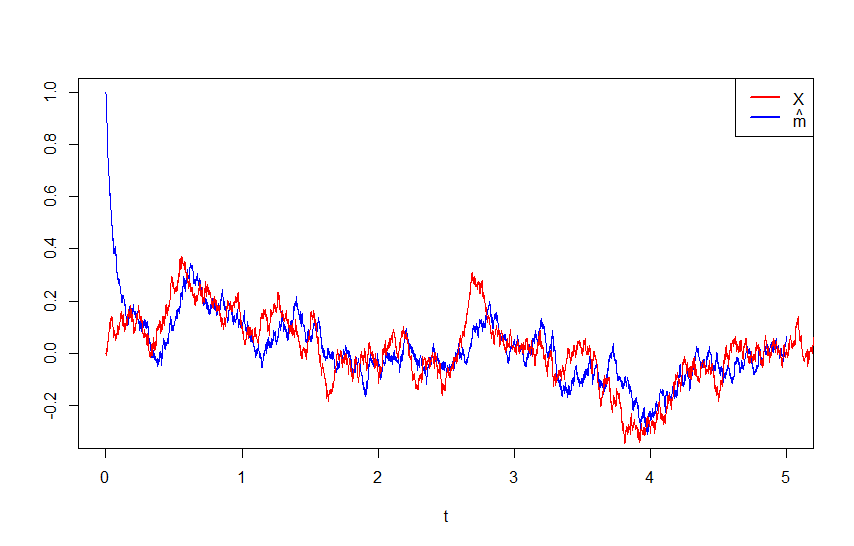}
      \caption{A path of $X_t$ and $\hat{m}_i^n(\theta^*)$ with $m_0=1$.}
      \label{X-m-graph}
    \end{figure}

    \section*{Acknowledgement}
    The author is grateful to N.Yoshida for important advice and useful discussions. I also thank Y.Koike for his help with accomplishing the computational simulations.

    \section*{Conflict of Interest}
    The corresponding author states that there is no conflict of interest.

\bibliographystyle{abbrvnat}

\bibliography{master}

\end{document}